\documentclass[british]{article}
\usepackage{ae,aecompl}
\usepackage[T1]{fontenc}
\usepackage[latin9]{inputenc}
\usepackage{parskip}
\usepackage{mathtools}
\usepackage{amsmath}
\usepackage{amsthm}
\usepackage{amssymb}
\usepackage{graphicx}
\usepackage{geometry}
\makeatletter
\theoremstyle{plain}
\newtheorem{thm}{\protect\theoremname}
\theoremstyle{definition}
\newtheorem{defn}[thm]{\protect\definitionname}
\theoremstyle{remark}
\newtheorem{rem}[thm]{\protect\remarkname}
\theoremstyle{plain}
\newtheorem{prop}[thm]{\protect\propositionname}
\theoremstyle{plain}
\newtheorem{lem}[thm]{\protect\lemmaname}

\usepackage{capt-of}
\usepackage{bbm}
\date{}

\theoremstyle{definition}
\newtheorem{notation}[thm]{Notation}
\newtheorem{algo}[thm]{Algorithm}
\makeatother

\usepackage{babel}
\providecommand{\definitionname}{Definition}
\providecommand{\lemmaname}{Lemma}
\providecommand{\propositionname}{Proposition}
\providecommand{\remarkname}{Remark}
\providecommand{\theoremname}{Theorem}

\begin{document}
\global\long\def\IN{\mathbb{N}}%
\global\long\def\II{\mathbbm{1}}%
\global\long\def\IZ{\mathbb{Z}}%
\global\long\def\IQ{\mathbb{Q}}%
\global\long\def\IR{\mathbb{R}}%
\global\long\def\IC{\mathbb{C}}%
\global\long\def\IP{\mathbb{P}}%
\global\long\def\IE{\mathbb{E}}%
\global\long\def\IV{\mathbb{V}}%

\title{Approximation Techniques for the Reconstruction of the Probability
Measure and the Coupling Parameters in a Curie-Weiss Model for Large
Populations}
\author{Miguel Ballesteros\thanks{IIMAS-UNAM, Mexico City, Mexico}, Ivan
Naumkin\footnotemark[1], and Gabor Toth\footnotemark[1]
\thanks{Corresponding author, e-mail: gabor.toth@iimas.unam.mx}}
\maketitle
\begin{abstract}
\noindent The Curie-Weiss model, originally used to study phase transitions
in statistical mechanics, has been adapted to model phenomena in social
sciences where many agents interact with each other. Reconstructing
the probability measure of a Curie-Weiss model via the maximum likelihood
method runs into the problem of computing the partition function which
scales exponentially with the population. We study the estimation
of the coupling parameters of a multi-group Curie-Weiss model using
large population asymptotic approximations for the relevant moments
of the probability distribution in the case that there are no interactions
between groups. As a result, we obtain an estimator which can be calculated
at a low and constant computational cost for any size of the population.
The estimator is consistent (under the added assumption that the population
is large enough), asymptotically normal, and satisfies large deviation
principles. The estimator is potentially useful in political science,
sociology, automated voting, and in any application where the degree
of social cohesion in a population has to be identified. The Curie-Weiss
model's coupling parameters provide a natural measure of social cohesion.
We discuss the problem of estimating the optimal weights in two-tier
voting systems.
\end{abstract}
\textbf{MSC 2020}: 62F10, 82B20, 60F05, 91B12

\textbf{Keywords}: Curie-Weiss model, Mathematical physics, Statistical
mechanics, Gibbs measures, Large population approximation, Mathematical
analysis, Maximum likelihood estimator, Two-tier voting systems, Optimal
weights

\section{Introduction}

Models of ferromagnetism are a mathematical representation of a lump
of material which consists of small magnets (also called `spins')
taking values in $\left\{ -1,1\right\} $ with a tendency to align
with each other or as a response to an external magnetic field. These
models are employed in statistical mechanics to understand the collective
behaviour of systems comprised of a large number interacting components.
Chief among these models are the Ising and the Curie-Weiss model.
The Ising model features spins located on the lattice $\mathbb{Z}^{d}$
for some fixed $d\in\IN$, where each spin interacts with its $2d$
nearest neighbours. It was first defined and studied by Ising in his
PhD thesis \cite{Ising1925}. What seems like a fairly simple setup
turns out to make for a difficult to analyse model in higher dimensions.
While the Ising model does not exhibit phase transitions for $d=1$,
Onsager \cite{Onsager1944} showed a phase transition for $d=2$.

The Curie-Weiss model arose before the Ising model in the early twentieth
century as a simplified model of ferromagnetism to study the phenomenon
of phase transitions and is named after Pierre Curie and Pierre Weiss.
It employs a mean-field approximation to the local interaction structure
of the spins, meaning that each spin interacts with the average value
of all other spins, independently of their location. This simplification
yields a much more tractable model and allows for rigorous mathematical
exploration of phase transitions. These models have long been studied
by both physicists and mathematical physicists, the latter group being
mainly interested in the rigorous analysis of the models' behaviour.
See \cite{FrieVele2017} for a modern mathematical discussion of both
models. 

While initially conceived to describe ferromagnets, the Curie-Weiss
model\textquoteright s insights extend far beyond physical systems.
Applications have surfaced in a variety of social science contexts,
such as economics \cite{OEA2018}, political science \cite{Ki2007},
and sociology \cite{CGh2007}, demonstrating how collective decision-making
or consensus can emerge in groups of interacting agents. Here, the
spins stand for opinions, votes, attitudes, etc. One direction of
generalisation of the Curie-Weiss model was motivated by social science
applications: instead of supposing that all spins interact equally
with each other, the assumption of having several distinguishable
groups of spins was introduced. Spins in different groups interact
in different ways, allowing even for negative interactions modelling
competition between groups of agents motivated by distinct cultures,
ideologies, or attitudes. Thus the Curie-Weiss model has become a
powerful framework for exploring social structures.

Prior work has been done on the estimation problem of parameters of
spin models such as those mentioned. A classical reference is the
book \cite{Kastelei1956}. The estimation of the interaction parameters
from data in the social sciences was studied in \cite{GalBarCo2009}
and \cite{FedVerCo2013}. Some work has also been done on the estimation
of parameters in more complicated models with random interaction structure
referred to as spin glass models (see \cite{Chatterj2007,CheSenWu2024}).
Another statistical problem is the detection of groups in models of
voting, explored in \cite{BRS2019,LoweSchu2020,BaMePeTo2023}.

A previous article \cite{BaMeSiTo2025} studied the reconstruction
of the coupling parameters of a Curie-Weiss model, such as the one
which is also the subject of this article. It found that the maximum
likelihood estimator has many desirable properties: it is well defined
for all values of the coupling parameters, including values close
to the critical regime, it is consistent (i.e.\! it converges in
probability to the true parameter value as the number of observations
increases), it is asymptotically normal, which permits the construction
of confidence intervals and hypothesis tests, and the estimator satisfies
large deviation principles, providing robust upper bounds on the probability
of large estimation errors from finite samples. The main weakness
of the maximum likelihood estimator in practice lies in the necessity
to calculate the partition function, which grows exponentially with
the population size. Since this may present a barrier to the practical
application of the estimator, we propose using an asymptotic approximation
of the relevant moments (see the left hand side of equation (\ref{eq:opt})).
This technique allows us to calculate an approximation of the maximum
likelihood estimator at a constant and low computational cost independent
of the population size. The main trade-off consists in no longer being
able to calculate a value of the estimator for all possible sample
values. Specifically, for sample realisations indicative of parameter
values very close to the critical regime, it is not possible to distinguish
between different values of the estimator lying all close to the critical
regime. However, in this case it is still possible to conclude that
with high probability the true parameter value lies close to the critical
value. Also, for any value of the parameters lying outside the critical
regime, the probability of such an atypical sample occurring decays
exponentially with sample size. For a fixed population size, as the
sample size grows, the estimator converges in probability to a value
in close proximity to the true parameter provided the population is
large. Other properties of the estimator include asymptotic normality
and large deviation principles, and are shared with the maximum likelihood
estimator.

For a voting population split into $M$ non-interacting groups, the
Curie-Weiss model is determined by a set of non-negative coupling
parameters $\beta_{\lambda}$, $\lambda\in\IN_{M}$\footnote{We will write $\IN_{m}\coloneq\left\{ 1,\ldots,m\right\} $ for any
$m\in\IN$.}. These parameters quantify the strength of interactions among voters
in each group. The voters are independent if the parameter value equals
0, and the larger the parameter, the stronger the interaction becomes.
If we can estimate $\beta_{\lambda}$ accurately from empirical voting
data, we gain insight into the extent of ideological cohesion or social
influence within groups. This, in turn, opens the way to applications
such as how voting weights should be assigned in decision-making bodies
to ensure fairness and reflect the collective will of the population.
A real-world example is the European Union, composed of 27 member
states, each of which delegates a representative to the Council of
the European Union. We explore weighted voting systems in such a council.
Each country\textquoteright s representative casts a vote that is
scaled by a predefined voting weight. Generally speaking, the determination
of fair weights demands some assumption on the underlying voting behaviour
of the population. A multi-group Curie-Weiss model serves as a probabilistic
voting model to that end. By estimating the interaction strengths
that govern intra-group voting dynamics, we can derive statistically
grounded proposals for optimal voting weights which can be calculated
directly from observations of voting outcomes. Section \ref{sec:Optimal-Weights}
delves into this topic under the simplifying assumption that interactions
occur only within, and not between, the $M$ groups. This sets the
stage for applying theoretical models to real-world governance scenarios
and decision-making systems in business and technology. Through this
lens, reconstructing the voting model from empirical data enables
not only a better understanding of group alignment but also the estimation
of representative weights that reflect these dynamics based on a fairness
criterion which demands that the council votes should reflect the
popular will.

The Curie-Weiss model is located at the nexus of statistical physics,
statistics, and social science. Our work defines and analyses an estimator
of the coupling constants based on the maximum likelihood paradigm,
but which is also easy to calculate and can be applied in practice,
providing a quantitative bridge from physical models to the challenges
of fair decision-making in councils, algorithms, and social systems.
We aim for a thorough exposition, with detailed proofs, that hopefully
resonates with an interdisciplinary audience of mathematicians, physicists,
statisticians, economists, and policy analysts.

The main results of the article are Proposition \ref{prop:error}
and Theorem \ref{thm:properties_bML_inf}. Proposition \ref{prop:error}
gives exponentially decaying upper bounds for the probabilities of
misidentifying the regime of the interactions between voters. Theorem
\ref{thm:properties_bML_inf} collects the statistical properties
of the estimator.

We present the Curie-Weiss model next in Section \ref{sec:Curie-Weiss-Model}.
Section \ref{sec:estimator_results} introduces the estimator of the
coupling parameters and the main results of the article. Section \ref{sec:Proofs}
contains the proofs of the main results. We explore the estimation
of the optimal weights in two-tier voting systems in Section \ref{sec:Optimal-Weights}.
Finally, the Appendix collects auxiliary results we need for the proofs.

\section{\label{sec:Curie-Weiss-Model}The Curie-Weiss Model}

We represent the voters of the overall population by the sets $\IN_{N_{\lambda}}$,
$\lambda\in\IN_{M}$, where $\lambda$ is the index of the group in
question. The votes of the overall population of size $N_{1}+\cdots+N_{M}$
take values in the space
\[
\Omega_{N_{1}+\cdots+N_{M}}\coloneq\left\{ -1,1\right\} ^{N_{1}+\cdots+N_{M}}.
\]
Each individual is represented by a double index $\lambda i$, where
$\lambda\in\IN_{M}$ refers to the group and $i\in\IN_{N_{\lambda}}$to
the individual. Each vote takes a value $x_{\lambda i}\in\Omega_{1}$.
The elements $\left(x_{11},\ldots,x_{1N_{1}},\ldots,x_{M1},\ldots,x_{MN_{M}}\right)\in\Omega_{N_{1}+\cdots+N_{M}}$
will be called voting configurations, with each voting configuration
being a complete record of the votes cast by the population on a specific
issue. The voting behaviour is described by the following voting model:
\begin{defn}
\label{def:CWM}Let $N_{\lambda}\in\IN$ and $\beta_{\lambda}\in\IR$,
$\lambda\in\IN_{M}$. We set $\boldsymbol{N}\coloneq\left(N_{1},\ldots,N_{M}\right)$
and $\boldsymbol{\beta}\coloneq\left(\beta_{1},\ldots,\beta_{M}\right)$.
The\emph{ Curie-Weiss model} (CWM) is defined for all voting configurations
$\left(x_{11},\ldots,x_{1N_{1}},\ldots,x_{M1},\ldots,x_{MN_{M}}\right)\in\Omega_{N_{1}+\cdots+N_{M}}$
by
\begin{equation}
\IP_{\boldsymbol{\beta},\boldsymbol{N}}\left(X_{11}=x_{11},\ldots,X_{MN_{M}}=x_{MN_{M}}\right)\coloneq Z_{\boldsymbol{\beta},\boldsymbol{N}}^{-1}\,\exp\left(\frac{1}{2}\sum_{\lambda=1}^{M}\frac{\beta_{\lambda}}{N_{\lambda}}\left(\sum_{i=1}^{N_{\lambda}}x_{\lambda i}\right)^{2}\right),\label{eq:CWM}
\end{equation}
where $Z_{\boldsymbol{\beta},\boldsymbol{N}}$ is a normalisation
constant called the partition function which depends on both $\boldsymbol{\beta}$
and $\boldsymbol{N}$. The constants $\beta_{\lambda}$ are the coupling
parameters.
\end{defn}

Originally, the CWM was conceived as a model of ferromagnetism, where
$M=1$, and hence there is a single coupling parameter $\beta$ which
is the inverse temperature of the system. The range of values for
$\beta$ is therefore usually $\left[0,\infty\right)$. We allow values
of $\beta$ outside this range for technical reasons related to the
range of the sample statistic $\boldsymbol{T}$ (see Definition \ref{def:T_stat}
below).

For our application as a model of voting, the CWM has non-negative
coupling parameters to properly reflect social cohesion within each
group. Each coupling parameter $\beta_{\lambda}$ measures the degree
of influence the voters in group $\lambda$ exert over each other.
This influence is stronger for larger $\beta_{\lambda}$. Inspecting
(\ref{eq:CWM}), we note that the most probable voting configurations
are those featuring unanimous votes either in favour of or against
the proposal. There are only two of these extreme configurations,
contrary to the large number of configurations with roughly equal
numbers of votes for and against. While the latter configurations
are individually of low probability, there are far more of them. This
discrepancy gives rise to the `conflict between energy and entropy.'
Which one of these pseudo forces dominates depends on the the regime
the model is in. Since the groups are non-interacting by assumption,
each group can be in any of the three regimes determined by the value
of the parameter $\beta_{\lambda}\geq0$. The three regimes are $\beta_{\lambda}<1$,
$\beta_{\lambda}=1$, and $\beta_{\lambda}>1$, called the high temperature,
critical, and low temperature regime, respectively.

The partition function $Z_{\boldsymbol{\beta},\boldsymbol{N}}$ can
be calculated as a sum 
\begin{equation}
Z_{\boldsymbol{\beta},\boldsymbol{N}}=\sum_{x\in\Omega_{N_{1}+\cdots+N_{M}}}\exp\left(\frac{1}{2}\sum_{\lambda=1}^{M}\frac{\beta_{\lambda}}{N_{\lambda}}\left(\sum_{i=1}^{N_{\lambda}}x_{\lambda i}\right)^{2}\right).\label{eq:part_fn}
\end{equation}
We define the group voting margins (i.e.\! the difference between
the numbers of yes and no votes) for each $\lambda\in\IN_{M}$ by
\begin{equation}
S_{\lambda}\coloneq\sum_{i=1}^{N_{\lambda}}X_{\lambda i},\quad\lambda\in\IN_{M}.\label{eq:S_lambda}
\end{equation}
Whenever we have a function of $f:\Omega_{N_{\lambda}}\rightarrow\IR$
that depends only on the votes belonging to group $\lambda$, such
as $S_{\lambda}$, we will write the expectation $\IE_{\boldsymbol{\beta},\boldsymbol{N}}f$
as $\IE_{\beta_{\lambda},N_{\lambda}}f$, making the dependence on
only the marginal distribution explicit.

The behaviour of the CWM can be understood by studying the random
vector
\[
\left(S_{1},\ldots,S_{M}\right)
\]
consisting of the group voting margins.

\begin{notation}We will use the symbol $\IE X$ for the expectation
and $\IV X$ for the variance of some random variable $X$. Capital
letters such as $X$ will denote random variables, while lower case
letters such as $x$ will stand for realisations of the corresponding
random variable.\end{notation}

\section{\label{sec:estimator_results}Estimation of $\boldsymbol{\beta}$
Using a Large Population Approximation}

\subsection{The Estimator $\hat{\boldsymbol{\beta}}_{\boldsymbol{N}}^{\infty}$}

Let $n\in\IN$ be the number of observations in a sample. Then each
sample takes values in the space
\[
\Omega_{N_{1}+\cdots+N_{M}}^{n}\coloneq\prod_{i=1}^{n}\Omega_{N_{1}+\cdots+N_{M}}.
\]
For our estimation, we assume that we have access to a set of voting
configurations $\left(x^{(1)},\ldots,x^{(n)}\right)\in\Omega_{N_{1}+\cdots+N_{M}}^{n}$
composed of $n\in\IN$ i.i.d.\! realisations of $\left(X_{11},\ldots,X_{MN_{\lambda}}\right)$
from the CWM. The density function for a random sample is the $n$-fold
product of (\ref{eq:CWM}):
\begin{equation}
f\left(x^{(1)},\ldots,x^{(n)};\boldsymbol{\beta}\right)\coloneq Z_{\boldsymbol{\beta},\boldsymbol{N}}^{-n}\,\prod_{t=1}^{n}\exp\left(\frac{1}{2}\sum_{\lambda=1}^{M}\frac{\beta_{\lambda}}{N_{\lambda}}\left(\sum_{i=1}^{N_{\lambda}}x_{\lambda i}^{\left(t\right)}\right)^{2}\right).\label{eq:density_fn}
\end{equation}
For a fixed sample $\left(x^{(1)},\ldots,x^{(n)}\right)$, $\boldsymbol{\beta}\in\IR^{M}\mapsto f\left(x^{(1)},\ldots,x^{(n)};\boldsymbol{\beta}\right)\in\IR$
is the likelihood function, and
\begin{equation}
\ln f\left(x^{(1)},\ldots,x^{(n)};\boldsymbol{\beta}\right)=-n\ln Z_{\boldsymbol{\beta},\boldsymbol{N}}+\frac{1}{2}\sum_{t=1}^{n}\sum_{\lambda=1}^{M}\frac{\beta_{\lambda}}{N_{\lambda}}\left(\sum_{i=1}^{N_{\lambda}}x_{\lambda i}^{\left(t\right)}\right)^{2}\label{eq:_log_like}
\end{equation}
is the log-likelihood function.

The maximum likelihood estimator $\hat{\boldsymbol{\beta}}_{ML}$
of $\boldsymbol{\beta}$ is defined for each sample $\left(x^{(1)},\ldots,x^{(n)}\right)$
to be the value which maximises the likelihood function:
\[
\hat{\boldsymbol{\beta}}_{ML}\coloneq\underset{\beta'}{\arg\max}\;f\left(x^{(1)},\ldots,x^{(n)};\boldsymbol{\beta}'\right).
\]
Since $x\mapsto\ln x$ is strictly increasing, we can also find $\hat{\boldsymbol{\beta}}_{ML}$
as the value which maximises the log-likelihood function
\[
\hat{\boldsymbol{\beta}}_{ML}=\underset{\beta'}{\arg\max}\;\ln f\left(x^{(1)},\ldots,x^{(n)};\boldsymbol{\beta}'\right).
\]
To calculate the maximum of the log-likelihood function, we derive
with respect to each $\beta_{\lambda}$ and equate each derivative
to 0:
\begin{align}
\frac{\textup{d}\ln f\left(x^{(1)},\ldots,x^{(n)};\boldsymbol{\beta}\right)}{\textup{d}\beta_{\lambda}} & =-\frac{n}{Z_{\boldsymbol{\beta},\boldsymbol{N}}}\frac{\textup{d}Z_{\boldsymbol{\beta},\boldsymbol{N}}}{\textup{d}\beta_{\lambda}}+\frac{1}{2N_{\lambda}}\sum_{t=1}^{n}\left(\sum_{i=1}^{N_{\lambda}}x_{\lambda i}^{\left(t\right)}\right)^{2}\overset{!}{=}0.\label{eq:FOC}
\end{align}
The squared sums $S_{\lambda}^{2}$ defined in (\ref{eq:S_lambda})
have expectation
\begin{align}
\IE_{\beta_{\lambda},N_{\lambda}}S_{\lambda}^{2} & =\frac{\textup{d}Z_{\boldsymbol{\beta},\boldsymbol{N}}}{\textup{d}\beta_{\lambda}}\cdot2N_{\lambda}\,Z_{\boldsymbol{\beta},\boldsymbol{N}}^{-1}\label{eq:S2_Z}
\end{align}
by Lemma \ref{lem:Z_deriv}. Now we substitute (\ref{eq:S2_Z}) into
(\ref{eq:FOC}): 
\[
\frac{\textup{d}\ln f\left(x^{(1)},\ldots,x^{(n)};\boldsymbol{\beta}\right)}{\textup{d}\beta_{\lambda}}=-\frac{n}{Z_{\boldsymbol{\beta},\boldsymbol{N}}}\frac{1}{2N_{\lambda}}Z_{\boldsymbol{\beta},\boldsymbol{N}}\,\IE_{\beta_{\lambda},N_{\lambda}}S_{\lambda}^{2}+\frac{1}{2N_{\lambda}}\sum_{t=1}^{n}\left(\sum_{i=1}^{N}x_{\lambda i}^{\left(t\right)}\right)^{2}.
\]
Then the optimality condition (\ref{eq:FOC}) is equivalent to
\begin{equation}
\IE_{\gamma,N}S_{\lambda}^{2}=\frac{1}{n}\sum_{t=1}^{n}\left(\sum_{i=1}^{N}x_{\lambda i}^{\left(t\right)}\right)^{2},\label{eq:opt}
\end{equation}
where $\gamma$ is the maximum likelihood estimate for the parameters
$\boldsymbol{\beta}$ given the sample $\left(x^{(1)},\ldots,x^{(n)}\right)$.
\begin{defn}
\label{def:T_stat}We define the statistic $\,\boldsymbol{T}:\Omega_{N_{1}+\cdots+N_{M}}^{n}\rightarrow\IR$
for any realisation of the sample $x=\left(x^{(1)},\ldots,x^{(n)}\right)\in\Omega_{N_{1}+\cdots+N_{M}}^{n}$
by
\[
\boldsymbol{T}\left(x\right)\coloneq\frac{1}{n}\sum_{t=1}^{n}\left(\left(\sum_{i=1}^{N_{1}}x_{1i}^{(t)}\right)^{2},\ldots,\left(\sum_{i=1}^{N_{M}}x_{Mi}^{(t)}\right)^{2}\right).
\]
\end{defn}

\begin{rem}
$\boldsymbol{T}$ is a random vector on the probability space $\Omega_{N_{1}+\cdots+N_{M}}^{n}$
with the power set of $\Omega_{N_{1}+\cdots+N_{M}}^{n}$ as the $\sigma$-algebra
and the probability measure defined by the density function (\ref{eq:density_fn}).
We will write $\boldsymbol{T}$ for this random variable and $\boldsymbol{T}\left(x\right)$
for its realisation given a sample $x\in\Omega_{N_{1}+\cdots+N_{M}}^{n}$.
\end{rem}

\begin{prop}
\label{prop:suff}$\boldsymbol{T}$ is a sufficient statistic for
$\boldsymbol{\beta}$.
\end{prop}

\begin{proof}
This is Proposition 5 in \cite{BaMeSiTo2025}.
\end{proof}
A function of central importance for the estimation when the coupling
parameters are in the low temperature regime is found in
\begin{defn}
\label{def:m_beta}Let $\beta\geq0$. The equation
\begin{equation}
\tanh\left(\beta x\right)=x,\quad x\in\IR,\label{eq:CW}
\end{equation}
is called the Curie-Weiss equation. We define $m\left(\beta\right)$
to be the largest solution to (\ref{eq:CW}). In order to obtain a
function $m:\left[0,\infty\right]\rightarrow\left[0,1\right]$, we
set $m\left(\infty\right)\coloneq1$.
\end{defn}

Section \ref{sec:Proofs} has some lemmas concerning properties of
the function $m$.

Before we can define the estimator $\hat{\boldsymbol{\beta}}_{\boldsymbol{N}}^{\infty}$
which uses an asymptotic approximation for the left hand side of condition
(\ref{eq:opt}), we need to establish the intervals for the coupling
parameters which can be estimated by $\hat{\boldsymbol{\beta}}_{\boldsymbol{N}}^{\infty}$.
\begin{defn}
\label{def:intervals}Let $N\in\IN$. Let $\boldsymbol{D}_{\textup{high}}$
and $\boldsymbol{D}_{\textup{low}}$ be positive constants such as
in Proposition \ref{prop:appr_moments}. Given $0<b_{1}<1<b_{2}$
such that 
\begin{equation}
\frac{N}{1-b_{1}}+\boldsymbol{D}_{\textup{high}}\sqrt{N}<m\left(b_{2}\right)^{2}N^{2}-\boldsymbol{D}_{\textup{low}}\left(\ln N\right)^{\frac{3}{2}}N^{\frac{3}{2}},\label{eq:separation}
\end{equation}
define the intervals
\begin{align*}
I_{h} & \coloneq\left[0,b_{1}\right],\quad I_{c}\coloneq\left(b_{1},b_{2}\right),\quad I_{l}\coloneq\left[b_{2},\infty\right),\\
J_{h} & \coloneq\left[\min\textup{Range}\left(S^{2}\right),\frac{N}{1-b_{1}}+\boldsymbol{D}_{\textup{high}}\sqrt{N}\right],\quad J_{c}\coloneq\left(\frac{N}{1-b_{1}}+\boldsymbol{D}_{\textup{high}}\sqrt{N},m\left(b_{2}\right)^{2}N^{2}-\boldsymbol{D}_{\textup{low}}\left(\ln N\right)^{\frac{3}{2}}N^{\frac{3}{2}}\right),\\
J_{l} & \coloneq\left[m\left(b_{2}\right)^{2}N^{2}-\boldsymbol{D}_{\textup{low}}\left(\ln N\right)^{\frac{3}{2}}N^{\frac{3}{2}},\infty\right),
\end{align*}
where $S$ is the sum defined in (\ref{eq:S_lambda}) for a group
of size $N$.
\end{defn}

\begin{rem}
\label{rem:intervals}Inequality (\ref{eq:separation}) holds for
all $N$ large enough since $m\left(\beta\right)>0$ for all $\beta>1$
by Lemma \ref{lem:m_beta_increasing}. By Proposition \ref{prop:appr_moments},
we have $\IE_{\beta_{\lambda},N_{\lambda}}S_{\lambda}^{2}\in J_{k}$
for all $\beta_{\lambda}\in I_{k}$, $k\in\left\{ h,l\right\} $.
In fact, Proposition \ref{prop:appr_moments} states the stronger
condition that for all $\beta_{\lambda}\in I_{k}$, $\IE_{\beta_{\lambda},N_{\lambda}}S_{\lambda}^{2}$
lies in the interior of $J_{k}$ for $k\in\left\{ h,l\right\} $.
We will define $\hat{\boldsymbol{\beta}}_{\boldsymbol{N}}^{\infty}$
supposing true parameter values $\beta_{\lambda}$ in $I_{h}\cup I_{l}$
for each group, where we will assume that for each $\lambda$ condition
(\ref{eq:separation}) holds for $N_{\lambda}$ instead of the generic
$N$.
\end{rem}

We next formally define the estimator for $\boldsymbol{\beta}$ which
employs the asymptotic expressions for $\IE_{\beta_{\lambda},N_{\lambda}}S_{\lambda}^{2}$
given in Proposition \ref{prop:appr_moments}. Note that contrary
to the maximum likelihood estimator (see Definition 7 in Article 1)
we do not define a numerical value for $\hat{\boldsymbol{\beta}}_{\boldsymbol{N}}^{\infty}$
for all realisations of the sample $\left(x^{(1)},\ldots,x^{(n)}\right)\in\Omega_{N_{1}+\cdots+N_{M}}^{n}$
for the reasons discussed in Remark \ref{rem:overlap}. We will assign
the symbol `u' below to those samples which do not permit the calculation
of an estimate for $\beta_{\lambda}$ for some group $\lambda$. The
introduction of this notation saves us from having to define a subset
of the sample space on which to define the estimator, while allowing
us to estimate the coupling parameter of each group independently
of the estimates for the parameters of other groups.

\begin{notation}\label{notation:infty}We will write $\left[-\infty,\infty\right]$
for the compactification $\IR\cup\left\{ -\infty,\infty\right\} $
and $\left[0,\infty\right]$ for $\left[0,\infty\right)\cup\left\{ \infty\right\} $.\end{notation}
\begin{defn}
\label{def:estimator_large_N}Let $0<b_{1}<1<b_{2}$ be constants
satisfying condition (\ref{eq:separation}), and consider the intervals
from Definition \ref{def:intervals}. Assume $\boldsymbol{\beta}\in\left(I_{h}\cup I_{l}\right)^{M}$.
The\emph{ estimator employing asymptotic approximations for the moments}
$\IE_{\beta_{\lambda},N_{\lambda}}S_{\lambda}^{2}$ is given by the
statistic $\hat{\boldsymbol{\beta}}_{\boldsymbol{N}}^{\infty}:\Omega_{N_{1}+\cdots N_{M}}^{n}\rightarrow\left(\left[-\infty,\infty\right]\cup\left\{ \textup{u}\right\} \right)^{M}$
defined coordinate-wise for each $\lambda\in\IN_{M}$ by
\begin{enumerate}
\item If $\left(\boldsymbol{T}\left(x^{(1)},\ldots,x^{(n)}\right)\right)_{\lambda}\in J_{h}$,
then $\hat{\beta}_{N_{\lambda}}^{\infty}\left(\lambda\right)\left(x^{(1)},\ldots,x^{(n)}\right)\coloneq1-\frac{N_{\lambda}}{\left(\boldsymbol{T}\left(x^{(1)},\ldots,x^{(n)}\right)\right)_{\lambda}}.$
\item If $\left(\boldsymbol{T}\left(x^{(1)},\ldots,x^{(n)}\right)\right)_{\lambda}\in J_{l}$,
then $\hat{\beta}_{N_{\lambda}}^{\infty}\left(\lambda\right)\left(x^{(1)},\ldots,x^{(n)}\right)>1$
is given by the unique value for which $\left(\boldsymbol{T}\left(x^{(1)},\ldots,x^{(n)}\right)\right)_{\lambda}=m\left(\hat{\beta}_{N_{\lambda}}^{\infty}\left(\lambda\right)\left(x^{(1)},\ldots,x^{(n)}\right)\right)^{2}N_{\lambda}^{2}$
is satisfied.
\item If $\left(\boldsymbol{T}\left(x^{(1)},\ldots,x^{(n)}\right)\right)_{\lambda}\in J_{c}$,
then we say there is insufficient evidence in the sample to conclude
that $\beta_{\lambda}$ is significantly different from 1, and $\hat{\beta}_{N_{\lambda}}^{\infty}\left(\lambda\right)\left(x^{(1)},\ldots,x^{(n)}\right)\coloneq\textup{u}$
is undefined.
\end{enumerate}
Similarly to the maximum likelihood estimator, the estimator $\hat{\boldsymbol{\beta}}_{\boldsymbol{N}}^{\infty}$
can also take the values $\pm\infty$ for extremely atypical sample
realisations. This is addressed in Proposition \ref{prop:error}.
\end{defn}

\subsection{Main Results of the Article}
\begin{prop}
\label{prop:error}Fix $\boldsymbol{N}\in\IN^{M}$. Let $0<b_{1}<1<b_{2}$
and $I_{h},I_{l},J_{h},J_{l}$ be in accordance with Definition \ref{def:intervals}.
\begin{enumerate}
\item For all $\boldsymbol{\beta}\in\left(I_{h}\cup I_{l}\right)^{M}$ with
$\boldsymbol{\beta}>0$ \footnote{For all $x\in\IR^{M}$, $x>0$ is to be interpreted coordinate-wise,
i.e.\! $x_{i}>0,i\in\IN_{M}$.}, there is a constant $\eta_{1}>0$ such that
\[
\IP\left\{ \hat{\beta}_{N_{\lambda}}^{\infty}\left(\lambda\right)\in\left(-\infty,0\right)\cup\left\{ \pm\infty\right\} \textup{ for some }\lambda\in\IN_{M}\right\} \leq2\exp\left(-\eta_{1}n\right),\quad n\in\IN.
\]
\item There are constants $\eta_{2},\eta_{3}>0$ such that:
\begin{align*}
\sup_{\boldsymbol{\beta}\in I_{h}^{M}}\left\{ \IP\left\{ \hat{\beta}_{N_{\lambda}}^{\infty}\left(\lambda\right)\in I_{l}\textup{ for some }\lambda\in\IN_{M}\right\} \right\}  & \leq\exp\left(-\eta_{2}n\right)\quad\textup{and}\\
\sup_{\boldsymbol{\beta}\in I_{l}^{M}}\left\{ \IP\left\{ \hat{\beta}_{N_{\lambda}}^{\infty}\left(\lambda\right)\in I_{h}\textup{ for some }\lambda\in\IN_{M}\right\} \right\}  & \leq\exp\left(-\eta_{3}n\right),\quad n\in\IN.
\end{align*}
\end{enumerate}
\end{prop}

The first statement above assures us that for any positive parameter
value in either the high or the low temperature regime, the probability
of a negative or infinite realisation of the estimator for any of
the coupling parameters decays exponentially with the sample size.
The second statement says that the probability of misidentifying any
of the regimes for the group parameter also decays exponentially.

The next theorem contains the main statistical properties of the estimator
$\hat{\boldsymbol{\beta}}_{\boldsymbol{N}}^{\infty}$. In preparation,
we will need the following item:
\begin{defn}
\label{def:beta_tilde}Let the intervals $I_{h}$ and $I_{l}$ be
as in Definition \ref{def:intervals}. Let, for all $\lambda\in\IN_{M}$,
$N_{\lambda}\in\IN$ and all $\beta_{\lambda}\in I_{h}\cup I_{l}$,
$\tilde{\beta}_{N_{\lambda}}\left(\lambda\right)\geq0$ be the value
which satisfies
\[
\IE_{\beta_{\lambda},N_{\lambda}}S_{\lambda}^{2}=\begin{cases}
\frac{N_{\lambda}}{1-\tilde{\beta}_{N_{\lambda}}\left(\lambda\right)} & \textup{if }\beta_{\lambda}\in I_{h},\\
m\left(\tilde{\beta}_{N_{\lambda}}\left(\lambda\right)\right)^{2}N_{\lambda}^{2} & \textup{if }\beta_{\lambda}\in I_{l}.
\end{cases}
\]
The vector $\tilde{\boldsymbol{\beta}}_{\boldsymbol{N}}$ is defined
by
\[
\tilde{\boldsymbol{\beta}}_{\boldsymbol{N}}\coloneq\left(\tilde{\beta}_{N_{1}}\left(1\right),\ldots,\tilde{\beta}_{N_{M}}\left(M\right)\right).
\]
\end{defn}

\begin{rem}
Note that, as $x\in\left(0,1\right)\mapsto\frac{1}{1-x}\in\left(0,\infty\right)$
and $x\in\left(1,\infty\right)\mapsto m\left(x\right)^{2}\in\left[0,1\right]$
are strictly increasing (see Lemma \ref{lem:m_beta_increasing}),
$\tilde{\boldsymbol{\beta}}_{\boldsymbol{N}}$ is well defined.

By using the asymptotic approximations for $\IE_{\beta_{\lambda},N_{\lambda}}S_{\lambda}^{2}$
provided in Proposition \ref{prop:appr_moments}, we incur a bias
which is reflected in the value of $\hat{\beta}_{N_{\lambda}}^{\infty}\left(\lambda\right)$.
Ideally, we would like $\hat{\beta}_{N_{\lambda}}^{\infty}\left(\lambda\right)\xrightarrow[n\rightarrow\infty]{\textup{p}}\beta_{\lambda}$,
which holds for the maximum likelihood estimator by Theorem 10 in
\cite{BaMeSiTo2025}. However, for $\hat{\beta}_{N_{\lambda}}^{\infty}\left(\lambda\right)$
we instead have $\hat{\beta}_{N_{\lambda}}^{\infty}\left(\lambda\right)\xrightarrow[n\rightarrow\infty]{\textup{p}}\tilde{\beta}_{N_{\lambda}}\left(\lambda\right)$
for fixed $N_{\lambda}$, and $\tilde{\beta}_{N_{\lambda}}\left(\lambda\right)$
is not equal to $\beta_{\lambda}$ in general.
\end{rem}

\begin{thm}
\label{thm:properties_bML_inf}The following statements hold:
\begin{enumerate}
\item $\hat{\boldsymbol{\beta}}_{\boldsymbol{N}}^{\infty}\xrightarrow[n\rightarrow\infty]{\textup{p}}\tilde{\boldsymbol{\beta}}_{\boldsymbol{N}}$.
\item $\tilde{\boldsymbol{\beta}}_{\boldsymbol{N}}\xrightarrow[N\rightarrow\infty]{}\boldsymbol{\beta}$.
\item $\sqrt{n}\left(\hat{\boldsymbol{\beta}}_{\boldsymbol{N}}^{\infty}-\tilde{\boldsymbol{\beta}}_{\boldsymbol{N}}\right)\xrightarrow[n\rightarrow\infty]{\textup{d}}\mathcal{N}\left(0,\Sigma_{\boldsymbol{N}}\right)$.
The covariance matrix $\Sigma_{\boldsymbol{N}}$ is diagonal, and
its entries $\left(\Sigma_{\boldsymbol{N}}\right)_{\lambda\lambda}$
are given by
\begin{enumerate}
\item If $\beta_{\lambda}\in I_{h}$, then $\left(\Sigma_{\boldsymbol{N}}\right)_{\lambda\lambda}=\left(1-\tilde{\boldsymbol{\beta}}_{N}\left(\lambda\right)\right)^{4}\IV_{\beta_{\lambda},N_{\lambda}}\frac{S_{\lambda}^{2}}{N_{\lambda}}\xrightarrow[N_{\lambda}\rightarrow\infty]{}2\left(1-\beta_{\lambda}\right)^{2}$.
\item If $\beta_{\lambda}\in I_{l}$, then $\left(\Sigma_{\boldsymbol{N}}\right)_{\lambda\lambda}=\frac{1}{\left[2m\left(\tilde{\boldsymbol{\beta}}_{N}\left(\lambda\right)\right)m'\left(\tilde{\boldsymbol{\beta}}_{N}\left(\lambda\right)\right)\right]^{2}}\IV_{\beta_{\lambda},N_{\lambda}}\left(\frac{S_{\lambda}}{N_{\lambda}}\right)^{2}\xrightarrow[N_{\lambda}\rightarrow\infty]{}0$.
\end{enumerate}
\item $\hat{\boldsymbol{\beta}}_{\boldsymbol{N}}^{\infty}$ satisfies a
large deviation principle with rate $n$ and rate function $\boldsymbol{J}$
defined in (\ref{eq:bold_J}). Furthermore, for any $\varepsilon>0$
there exist an $\boldsymbol{N}_{\varepsilon}\in\IN^{M}$ and $B_{\varepsilon},C_{\varepsilon}>0$
such that for all $\boldsymbol{N}\geq\boldsymbol{N}_{\varepsilon}$,
\[
\IP\left\{ \left|\hat{\boldsymbol{\beta}}_{\boldsymbol{N}}^{\infty}-\boldsymbol{\beta}\right|\geq\varepsilon\right\} \leq C_{\varepsilon}\exp\left(-B_{\varepsilon}n\right),\quad n\in\IN.
\]
\end{enumerate}
\end{thm}

\begin{rem}
The large deviations principle for $\hat{\boldsymbol{\beta}}_{\boldsymbol{N}}^{\infty}$
in statement 4 is for fixed $\boldsymbol{N}\in\IN^{M}$, and the unique
minimum of rate function $\boldsymbol{J}$ is located at $\tilde{\boldsymbol{\beta}}_{\boldsymbol{N}}$
rather than $\boldsymbol{\beta}$. As such, the significance of the
second part of statement 4 above lies in providing an exponentially
decaying upper bound for the probability of a large deviation of the
estimator $\hat{\boldsymbol{\beta}}_{\boldsymbol{N}}^{\infty}$ from
the true parameter value $\boldsymbol{\beta}$.
\end{rem}

\section{\label{sec:Proofs}Proofs of Proposition \ref{prop:error} and Theorem
\ref{thm:properties_bML_inf}}

\subsection{Asymptotic Approximation of $\protect\IE_{\beta_{\lambda},N_{\lambda}}$}
\begin{lem}
\label{lem:m_beta_increasing}The mapping $m:\left[1,\infty\right)\rightarrow\left[0,1\right]$
from Definition \ref{def:m_beta} is strictly increasing and $\lim_{\beta\rightarrow\infty}m\left(\beta\right)=1$.
\end{lem}

\begin{rem}
The statement $\lim_{\beta\rightarrow\infty}m\left(\beta\right)=1$
implies that $m:\left[0,\infty\right]\rightarrow\left[0,1\right]$
is continuous.
\end{rem}

\begin{proof}
The first derivative of the function $x\in\left[0,\infty\right)\mapsto\tanh x\in\left[0,1\right)$
is $1-\tanh^{2}x\leq1$, and in particular it is equal to 1 for $x=0$.
The second derivative $-2\tanh x/\cosh^{2}x$ is strictly negative
on $\left(0,\infty\right)$. It follows that for all $\beta\geq0$
and all $x\in\left(0,\infty\right)$,
\begin{equation}
\tanh\left(\beta x\right)>x\quad\textup{implies}\quad x<m\left(\beta\right).\label{eq:m_beta_ineq}
\end{equation}

Let $0\leq\beta_{1}<\beta_{2}$. As the function $x\in\left[0,\infty\right)\mapsto\tanh x\in\left[0,1\right)$
is strictly increasing, we have
\begin{equation}
\tanh\left(\beta_{1}x\right)<\tanh\left(\beta_{2}x\right),\quad x\in\left(0,\infty\right).\label{eq:tanh_increasing}
\end{equation}
Therefore,
\[
m\left(\beta_{1}\right)=\tanh\left(\beta_{1}m\left(\beta_{1}\right)\right)<\tanh\left(\beta_{2}m\left(\beta_{1}\right)\right),
\]
where the equality is by definition of $m\left(\beta_{1}\right)$
and the inequality due to (\ref{eq:tanh_increasing}). By (\ref{eq:m_beta_ineq}),
$m\left(\beta_{1}\right)<m\left(\beta_{2}\right)$ holds.

For the limit statement $\lim_{\beta\rightarrow\infty}m\left(\beta\right)=1$,
fix $x\in\left(0,1\right)$ and consider $\lim_{\beta\rightarrow\infty}\tanh\left(\beta x\right)=1$.
Therefore, for any $\beta$ large enough, we have $\tanh\left(\beta x\right)>x$,
and by (\ref{eq:m_beta_ineq}), $x<m\left(\beta\right)$. This proves
$\lim_{\beta\rightarrow\infty}m\left(\beta\right)=1$.
\end{proof}
\begin{lem}
\label{lem:m_beta_diff}The mapping $m:\left(1,\infty\right)\rightarrow\left[0,1\right]$
is continuously differentiable.
\end{lem}

\begin{proof}
Let $h:\left(1,\infty\right)\times\left(0,\infty\right)\rightarrow\IR$
be defined by $h\left(\beta,x\right)\coloneq\tanh\left(\beta x\right)-x$
for all $\left(\beta,x\right)\in\left(1,\infty\right)\times\left(0,\infty\right)$.
By definition of $m\left(\beta\right)$, we have, for all $\beta\in\left(1,\infty\right)$,
$h\left(\beta,m\left(\beta\right)\right)=0$. The function $h$ is
continuously differentiable, and
\[
\frac{\textup{d}h\left(\beta,x\right)}{\textup{d}x}=-\tanh^{2}\left(\beta x\right)\neq0,\quad\left(\beta,x\right)\in\left(1,\infty\right)\times\left(0,\infty\right).
\]
Therefore, the implicit function theorem applies, and the function
$\beta\in\left(1,\infty\right)\mapsto m\left(\beta\right)\in\left[0,1\right]$
is continuously differentiable.
\end{proof}
In this section, we will be working with a single group $\lambda$
at a time and the corresponding marginal distribution of $\IP_{\beta_{\lambda},N_{\lambda}}$
of $\IP_{\boldsymbol{\beta},\boldsymbol{N}}$, and therefore there
will be no confusion if we omit the subindex $\lambda$ from all our
expressions to improve readability. As such, we will be writing $\IP_{\beta,N}$
instead of $\IP_{\beta_{\lambda},N_{\lambda}}$, $T$ instead of $\left(\boldsymbol{T}\right)_{\lambda}$,
$\hat{\beta}_{N}^{\infty}$ for $\hat{\beta}_{N_{\lambda}}^{\infty}\left(\lambda\right)$,
$X_{i}$ and not $X_{\lambda i}$, etc. We will return to the multi-group
setting in Section \ref{sec:Optimal-Weights} about the optimal weights
in two-tier voting systems.

Propositions \ref{prop:appr_correlations} and \ref{prop:appr_moments}
will allow us to define the estimator $\hat{\beta}_{N}^{\infty}$
based on the approximations for $\IE_{\beta,N}S^{2}$ provided in
Proposition \ref{prop:appr_moments}. These results are known for
$\beta\geq0$. We supply a proof for all $\beta\in\IR$ for the readers'
convenience. $k!!$ for any $k\in\IN$ stands for the double factorial
\[
k!!=\prod_{1\leq\ell\leq k,\,k-\ell\hspace*{-0.22cm}\mod 2\,=\,0}\ell.
\]

\begin{defn}
Real-valued sequences $\left(f_{n}\right)_{n\in\IN},\left(g_{n}\right)_{n\in\IN}$
with $g_{n}\neq0$, $n\in\IN$, are called \emph{asymptotically equal}
(as $n\to\infty$), in short $f_{n}\approx g_{n}$, if
\[
\lim_{n\rightarrow\infty}\frac{f_{n}}{g_{n}}=1.
\]
\end{defn}

\begin{prop}
\label{prop:appr_correlations}For all $\beta\in\IR,\beta\neq1$,
and all $k\in\IN$, the correlation $\IE_{\beta,N}X_{1}\cdots X_{k}$
is equal to $0$ for all $k$ odd and all $N\in\IN$. For $k$ even,
$\IE_{\beta,N}X_{1}\cdots X_{k}$ is asymptotically equal to
\[
\IE_{\beta,N}X_{1}\cdots X_{k}\approx\begin{cases}
\left(k-1\right)!!\left(\frac{\beta}{1-\beta}\right)^{\frac{k}{2}}\,\frac{1}{N^{\frac{k}{2}}} & \text{if }\,\beta<1,\\
m\left(\beta\right)^{k} & \text{if }\,\beta>1,
\end{cases}
\]
where $0<m\left(\beta\right)<1$ for $\beta>1$ is the constant from
Definition \ref{def:m_beta}.

Let $k$ be even. There are constants $\boldsymbol{C}_{\textup{high}},\boldsymbol{C}_{\textup{low}}>0$
with the following property. For all $0<\beta<1$, the bound 
\begin{equation}
\left|\IE_{\beta,N}X_{1}\cdots X_{k}-\left(k-1\right)!!\left(\frac{\beta}{1-\beta}\right)^{\frac{k}{2}}\,\frac{1}{N^{\frac{k}{2}}}\right|<\boldsymbol{C}_{\textup{high}}\left(\frac{\ln N}{N}\right)^{\frac{k+2}{2}}\label{eq:corr_UB_h}
\end{equation}
holds for all $N\in\IN$. For all $\beta>1$, the bound 
\begin{equation}
\left|\IE_{\beta,N}X_{1}\cdots X_{k}-m\left(\beta\right)^{k}\right|<\boldsymbol{C}_{\textup{low}}\frac{\left(\ln N\right)^{\frac{3}{2}}}{\sqrt{N}}\label{eq:corr_UB_l}
\end{equation}
holds for all $N\in\IN$.
\end{prop}

\begin{proof}
Let first $\beta<0$. We set $\gamma\coloneq-\beta>0$. We first find
a de Finetti representation for the expression $\IE_{\beta,N}X_{1}\cdots X_{k}$.
For any $\left(x_{1},\ldots,x_{N}\right)\in\Omega_{N}$, we express
the probabilities $\IP_{\beta,N}\left(X_{1}=x_{1},\ldots,X_{N}=x_{N}\right)$
in integral form (referred to as a Hubbard-Stratonovich transformation).
Let $a\in\IR$. Then
\[
\exp\left(-\frac{a^{2}}{2}\right)=\frac{1}{\sqrt{2\pi}}\int_{\IR}\exp\left(ias\right)\exp\left(-\frac{s^{2}}{2}\right)\textup{d}s
\]
holds due to the definition of the characteristic function of a standard
normal distribution. In particular, if $a=\sqrt{\frac{\gamma}{N}}\sum_{j=1}^{N}x_{j}$,
then
\[
\exp\left(-\frac{\gamma}{2N}\left(\sum_{j=1}^{N}x_{j}\right)^{2}\right)=\frac{1}{\sqrt{2\pi}}\int_{\IR}\exp\left(is\sqrt{\frac{\gamma}{N}}\sum_{j=1}^{N}x_{j}\right)\exp\left(-\frac{s^{2}}{2}\right)\textup{d}s.
\]
We use the previous expression for $\exp\left(-\frac{\gamma}{2N}\left(\sum_{j=1}^{N}x_{j}\right)^{2}\right)$
to calculate the expectation
\begin{equation}
\IE_{\beta,N}X_{1}\cdots X_{k}=\frac{\sum_{x\in\Omega_{N}}\exp\left(-\frac{\gamma}{2N}\left(\sum_{j=1}^{N}x_{j}\right)^{2}\right)x_{1}\cdots x_{k}}{\sum_{y\in\Omega_{N}}\exp\left(-\frac{\gamma}{2N}\left(\sum_{j=1}^{N}y_{j}\right)^{2}\right)}.\label{eq:corr_sum_formula}
\end{equation}
We first calculate the denominator
\begin{align*}
\sum_{y\in\Omega_{N}}\exp\left(-\frac{\gamma}{2N}\left(\sum_{j=1}^{N}y_{j}\right)^{2}\right) & =\sum_{y\in\Omega_{N}}\frac{1}{\sqrt{2\pi}}\int_{\IR}\exp\left(is\sqrt{\frac{\gamma}{N}}\sum_{j=1}^{N}y_{j}\right)\exp\left(-\frac{s^{2}}{2}\right)\textup{d}s\\
 & =\frac{1}{\sqrt{2\pi}}\int_{\IR}\exp\left(-\frac{s^{2}}{2}\right)\sum_{y\in\Omega_{N}}\prod_{j=1}^{N}\exp\left(is\sqrt{\frac{\gamma}{N}}y_{j}\right)\textup{d}s\\
 & =\frac{1}{\sqrt{2\pi}}\int_{\IR}\exp\left(-\frac{s^{2}}{2}\right)\prod_{j=1}^{N}\sum_{y_{j}\in\Omega_{1}}\exp\left(is\sqrt{\frac{\gamma}{N}}y_{j}\right)\textup{d}s\\
 & =\frac{1}{\sqrt{2\pi}}\int_{\IR}\exp\left(-\frac{s^{2}}{2}\right)\prod_{j=1}^{N}\left(\exp\left(is\sqrt{\frac{\gamma}{N}}\right)+\exp\left(-is\sqrt{\frac{\gamma}{N}}\right)\right)\textup{d}s\\
 & =\frac{1}{\sqrt{2\pi}}\int_{\IR}\exp\left(-\frac{s^{2}}{2}\right)\prod_{j=1}^{N}2\cos\left(s\sqrt{\frac{\gamma}{N}}\right)\;\textup{d}s\\
 & =2^{N}\frac{1}{\sqrt{2\pi}}\int_{\IR}\exp\left(-\frac{s^{2}}{2}\right)\cos^{N}\left(s\sqrt{\frac{\gamma}{N}}\right)\;\textup{d}s,
\end{align*}
 and next the numerator
\begin{align*}
 & \quad\sum_{x\in\Omega_{N}}\exp\left(-\frac{\gamma}{2N}\left(\sum_{j=1}^{N}x_{j}\right)^{2}\right)x_{1}\cdots x_{k}\\
 & =\sum_{x\in\Omega_{N}}\frac{1}{\sqrt{2\pi}}\int_{\IR}\exp\left(-\frac{s^{2}}{2}\right)\prod_{j=1}^{N}\exp\left(is\sqrt{\frac{\gamma}{N}}x_{j}\right)x_{1}\cdots x_{k}\,\textup{d}s\\
 & =\frac{1}{\sqrt{2\pi}}\int_{\IR}\exp\left(-\frac{s^{2}}{2}\right)\sum_{x\in\Omega_{N}}x_{1}\cdots x_{k}\prod_{j=1}^{N}\exp\left(is\sqrt{\frac{\gamma}{N}}x_{j}\right)\textup{d}s\\
 & =\frac{1}{\sqrt{2\pi}}\int_{\IR}\exp\left(-\frac{s^{2}}{2}\right)\prod_{j=1}^{k}\left[\sum_{x_{j}\in\Omega_{1}}x_{j}\exp\left(is\sqrt{\frac{\gamma}{N}}x_{j}\right)\right]\\
 & \quad\cdot\prod_{j=k+1}^{N}\left[\sum_{x_{j}\in\Omega_{1}}\exp\left(is\sqrt{\frac{\gamma}{N}}x_{j}\right)\right]\textup{d}s\\
 & =\frac{1}{\sqrt{2\pi}}\int_{\IR}\exp\left(-\frac{s^{2}}{2}\right)\prod_{j=1}^{k}\left(\exp\left(is\sqrt{\frac{\gamma}{N}}\right)-\exp\left(-is\sqrt{\frac{\gamma}{N}}\right)\right)\\
 & \quad\cdot\prod_{j=k+1}^{N}\left(\exp\left(is\sqrt{\frac{\gamma}{N}}\right)+\exp\left(-is\sqrt{\frac{\gamma}{N}}\right)\right)\textup{d}s\\
 & =\frac{1}{\sqrt{2\pi}}\int_{\IR}\exp\left(-\frac{s^{2}}{2}\right)\prod_{j=1}^{k}2i\sin\left(s\sqrt{\frac{\gamma}{N}}\right)\prod_{j=k+1}^{N}2\cos\left(s\sqrt{\frac{\gamma}{N}}\right)\;\textup{d}s\\
 & =2^{N}i^{k}\frac{1}{\sqrt{2\pi}}\int_{\IR}\exp\left(-\frac{s^{2}}{2}\right)\sin^{k}\left(s\sqrt{\frac{\gamma}{N}}\right)\;\cos^{N-k}\left(s\sqrt{\frac{\gamma}{N}}\right)\;\textup{d}s.
\end{align*}
Putting together the last two displays, we obtain
\begin{align}
\IE_{\beta,N}X_{1}\cdots X_{k} & =\frac{i^{k}\int_{\IR}\exp\left(-\frac{s^{2}}{2}\right)\sin^{k}\left(s\sqrt{\frac{\gamma}{N}}\right)\;\cos^{N-k}\left(s\sqrt{\frac{\gamma}{N}}\right)\;\textup{d}s}{\int_{\IR}\exp\left(-\frac{s^{2}}{2}\right)\cos^{N}\left(s\sqrt{\frac{\gamma}{N}}\right)\;\textup{d}s}\nonumber \\
 & =\frac{i^{k}\int_{\IR}\exp\left(-\frac{s^{2}}{2}\right)\tan^{k}\left(\sqrt{\frac{\gamma}{N}}s\right)\;\cos^{N}\left(\sqrt{\frac{\gamma}{N}}s\right)\;\textup{d}s}{\int_{\IR}\exp\left(-\frac{s^{2}}{2}\right)\cos^{N}\left(\sqrt{\frac{\gamma}{N}}s\right)\;\textup{d}s}\label{eq:corr_neg_beta}
\end{align}
At this point we note that the numerator is 0 if and only if $k$
is odd, which gives the claim $\IE_{\beta,N}X_{1}\cdots X_{k}=0$
for all $\beta<0$ and all $k$ odd. Assume from now on that $k$
is even.

We expand the trigonometric functions in Taylor series. Fix $s\in\IR$.
We have
\begin{align*}
\tan^{k}\left(\sqrt{\frac{\gamma}{N}}s\right) & =\left(\sqrt{\frac{\gamma}{N}}s+\frac{1}{3}\left(\sqrt{\frac{\gamma}{N}}s\right)^{3}+\frac{2}{15}\left(\sqrt{\frac{\gamma}{N}}s\right)^{5}+\cdots\right)^{k}\\
 & =\frac{\gamma^{\frac{k}{2}}s^{k}}{N^{\frac{k}{2}}}+O\left(\frac{1}{N^{\frac{k}{2}+1}}\right),
\end{align*}
so
\[
\tan^{k}\left(\sqrt{\frac{\gamma}{N}}s\right)-\frac{\gamma^{\frac{k}{2}}s^{k}}{N^{\frac{k}{2}}}=O\left(\frac{1}{N^{\frac{k}{2}+1}}\right).
\]
Also,
\[
\cos^{N}z=\left(1-\frac{1}{2}\left(\sqrt{\frac{\gamma}{N}}s\right)^{2}+\frac{1}{4!}\left(\sqrt{\frac{\gamma}{N}}s\right)^{4}-+\cdots\right)^{N}=\left(1-\frac{1}{2}\frac{\gamma s^{2}}{N}+O\left(\frac{1}{N^{2}}\right)\right)^{N}\xrightarrow[N\rightarrow\infty]{}\exp\left(-\frac{1}{2}\gamma s^{2}\right)
\]
holds.

As
\begin{align*}
\exp\left(-\frac{s^{2}}{2}\right)\sin^{k}\left(s\sqrt{\frac{\gamma}{N}}\right)\;\cos^{N-k}\left(s\sqrt{\frac{\gamma}{N}}\right) & \leq\exp\left(-\frac{s^{2}}{2}\right),\\
\exp\left(-\frac{s^{2}}{2}\right)\cos^{N}\left(\sqrt{\frac{\gamma}{N}}s\right) & \leq\exp\left(-\frac{s^{2}}{2}\right)
\end{align*}
for all $N\in\IN$, $s\mapsto\exp\left(-\frac{s^{2}}{2}\right)$ is
a majorant for the integrand in both the numerator and the denominator
of (\ref{eq:corr_neg_beta}) for each $N\in\IN$. An application of
the theorem of dominated convergence yields
\begin{align*}
\IE_{\beta,N}X_{1}\cdots X_{k} & \approx\frac{\frac{1}{N^{\frac{k}{2}}}i^{k}\gamma^{\frac{k}{2}}\int_{\IR}\exp\left(-\frac{s^{2}}{2}\right)s^{k}\exp\left(-\frac{1}{2}\gamma s^{2}\right)\textup{d}s}{\int_{\IR}\exp\left(-\frac{s^{2}}{2}\right)\exp\left(-\frac{1}{2}\gamma s^{2}\right)\textup{d}s}\\
 & \approx\frac{i^{k}\gamma^{\frac{k}{2}}\int_{\IR}\exp\left(-\frac{1+\gamma}{2}s^{2}\right)s^{k}\textup{d}s}{\int_{\IR}\exp\left(-\frac{1+\gamma}{2}s^{2}\right)\textup{d}s}\frac{1}{N^{\frac{k}{2}}}\\
 & =i^{k}\gamma^{\frac{k}{2}}\frac{\sqrt{\frac{1+\gamma}{2\pi}}\int_{\IR}\exp\left(-\frac{1+\gamma}{2}s^{2}\right)s^{k}\textup{d}s}{\sqrt{\frac{1+\gamma}{2\pi}}\int_{\IR}\exp\left(-\frac{1+\gamma}{2}s^{2}\right)\textup{d}s}\frac{1}{N^{\frac{k}{2}}}\\
 & =i^{k}\gamma^{\frac{k}{2}}\frac{1}{N^{\frac{k}{2}}}\left(k-1\right)!!\left(\frac{1}{1+\gamma}\right)^{\frac{k}{2}}.
\end{align*}
Taking into account $i^{k}=\left(-1\right)^{\frac{k}{2}}$ and $\gamma=-\beta$,
we obtain the claim
\[
\IE_{\beta,N}X_{1}\cdots X_{k}\approx\left(k-1\right)!!\left(\frac{\beta}{1-\beta}\right)^{\frac{k}{2}}
\]
for all $\beta<0$ and all $k$ even.

If $\beta=0$, the random variables $X_{1},\ldots,X_{N}$ are independent.
Therefore,
\[
\IE_{0,N}X_{1}\cdots X_{k}=\IE_{\beta,N}X_{1}\cdots\IE_{0,N}X_{k}=0
\]
holds for all $k$.

Now let $\beta>0$. Let $a\in\IR$. The calculation
\begin{align*}
1 & =\frac{1}{\sqrt{2\pi}}\int_{\IR}\exp\left(-\frac{\left(s-a\right)^{2}}{2}\right)\textup{d}s\\
 & =\frac{1}{\sqrt{2\pi}}\int_{\IR}\exp\left(-\frac{s^{2}}{2}\right)\exp\left(as\right)\exp\left(-\frac{a^{2}}{2}\right)\textup{d}s,
\end{align*}
gives
\[
\exp\left(\frac{a^{2}}{2}\right)=\frac{1}{\sqrt{2\pi}}\int_{\IR}\exp\left(-\frac{s^{2}}{2}\right)\exp\left(as\right)\textup{d}s.
\]
We apply this formula for $a=\sqrt{\frac{\beta}{N}}\sum_{j=1}^{N}x_{j}$:
\[
\exp\left(\frac{\beta}{2N}\left(\sum_{j=1}^{N}x_{j}\right)^{2}\right)=\frac{1}{\sqrt{2\pi}}\int_{\IR}\exp\left(-\frac{s^{2}}{2}\right)\exp\left(s\sqrt{\frac{\beta}{N}}\sum_{j=1}^{N}x_{j}\right)\textup{d}s.
\]
As above, we express the correlation $\IE_{\beta,N}X_{1}\cdots X_{k}$
as in (\ref{eq:corr_sum_formula}), and calculate first the denominator
\begin{align*}
\sum_{y\in\Omega_{N}}\exp\left(\frac{\beta}{2N}\left(\sum_{j=1}^{N}y_{j}\right)^{2}\right) & =\sum_{y\in\Omega_{N}}\frac{1}{\sqrt{2\pi}}\int_{\IR}\exp\left(-\frac{s^{2}}{2}\right)\exp\left(s\sqrt{\frac{\beta}{N}}\sum_{j=1}^{N}y_{j}\right)\textup{d}s\\
 & =\frac{1}{\sqrt{2\pi}}\int_{\IR}\exp\left(-\frac{s^{2}}{2}\right)\sum_{y\in\Omega_{N}}\prod_{j=1}^{N}\exp\left(s\sqrt{\frac{\beta}{N}}y_{j}\right)\textup{d}s\\
 & =\frac{1}{\sqrt{2\pi}}\int_{\IR}\exp\left(-\frac{s^{2}}{2}\right)\prod_{j=1}^{N}\sum_{y_{j}\in\Omega_{1}}\exp\left(s\sqrt{\frac{\beta}{N}}y_{j}\right)\textup{d}s\\
 & =\frac{1}{\sqrt{2\pi}}\int_{\IR}\exp\left(-\frac{s^{2}}{2}\right)\prod_{j=1}^{N}\left(\exp\left(s\sqrt{\frac{\beta}{N}}\right)+\exp\left(-s\sqrt{\frac{\beta}{N}}\right)\right)\textup{d}s\\
 & =\frac{1}{\sqrt{2\pi}}\int_{\IR}\exp\left(-\frac{s^{2}}{2}\right)\prod_{j=1}^{N}2\cosh\left(s\sqrt{\frac{\beta}{N}}\right)\;\textup{d}s\\
 & =2^{N}\frac{1}{\sqrt{2\pi}}\int_{\IR}\exp\left(-\frac{s^{2}}{2}\right)\cosh^{N}\left(s\sqrt{\frac{\beta}{N}}\right)\;\textup{d}s\\
 & =2^{N}\sqrt{\frac{N}{2\pi\beta}}\int_{\IR}\exp\left(-\frac{N}{\beta}\frac{z^{2}}{2}\right)\cosh^{N}z\;\textup{d}z\\
 & =2^{N}\sqrt{\frac{N}{2\pi\beta}}\int_{\IR}\exp\left(-N\left(\frac{z^{2}}{2\beta}-\ln\cosh z\right)\right)\;\textup{d}z
\end{align*}
where in the step before last we substituted $z\coloneq s\sqrt{\frac{\beta}{N}}$.
Next we calculate the numerator
\begin{align*}
\sum_{x\in\Omega_{N}}\exp\left(-\frac{\gamma}{2N}\left(\sum_{j=1}^{N}x_{j}\right)^{2}\right)x_{1}\cdots x_{k} & =\sum_{x\in\Omega_{N}}\frac{1}{\sqrt{2\pi}}\int_{\IR}\exp\left(-\frac{s^{2}}{2}\right)\exp\left(s\sqrt{\frac{\beta}{N}}\sum_{j=1}^{N}x_{j}\right)x_{1}\cdots x_{k}\,\textup{d}s\\
 & =\frac{1}{\sqrt{2\pi}}\int_{\IR}\exp\left(-\frac{s^{2}}{2}\right)\sum_{x\in\Omega_{N}}x_{1}\cdots x_{k}\prod_{j=1}^{N}\exp\left(s\sqrt{\frac{\beta}{N}}x_{j}\right)\textup{d}s\\
 & =\frac{1}{\sqrt{2\pi}}\int_{\IR}\exp\left(-\frac{s^{2}}{2}\right)\prod_{j=1}^{k}\left[\sum_{x_{j}\in\Omega_{1}}x_{j}\exp\left(s\sqrt{\frac{\beta}{N}}x_{j}\right)\right]\\
 & \quad\cdot\prod_{j=k+1}^{N}\left[\sum_{x_{j}\in\Omega_{1}}\exp\left(s\sqrt{\frac{\beta}{N}}x_{j}\right)\right]\textup{d}s\\
 & =\frac{1}{\sqrt{2\pi}}\int_{\IR}\exp\left(-\frac{s^{2}}{2}\right)\prod_{j=1}^{k}\left(\exp\left(s\sqrt{\frac{\beta}{N}}x_{j}\right)-\exp\left(-s\sqrt{\frac{\beta}{N}}x_{j}\right)\right)\\
 & \quad\cdot\prod_{j=k+1}^{N}\left(\exp\left(s\sqrt{\frac{\beta}{N}}x_{j}\right)+\exp\left(-s\sqrt{\frac{\beta}{N}}x_{j}\right)\right)\textup{d}s\\
 & =\frac{1}{\sqrt{2\pi}}\int_{\IR}\exp\left(-\frac{s^{2}}{2}\right)\prod_{j=1}^{k}2\sinh\left(s\sqrt{\frac{\gamma}{N}}\right)\prod_{j=k+1}^{N}2\cosh\left(s\sqrt{\frac{\gamma}{N}}\right)\;\textup{d}s\\
 & =2^{N}\frac{1}{\sqrt{2\pi}}\int_{\IR}\exp\left(-\frac{s^{2}}{2}\right)\sinh^{k}\left(s\sqrt{\frac{\gamma}{N}}\right)\;\cosh^{N-k}\left(s\sqrt{\frac{\gamma}{N}}\right)\;\textup{d}s\\
 & =2^{N}\frac{1}{\sqrt{2\pi}}\int_{\IR}\exp\left(-\frac{s^{2}}{2}\right)\tanh^{k}\left(s\sqrt{\frac{\gamma}{N}}\right)\;\cosh^{N}\left(s\sqrt{\frac{\gamma}{N}}\right)\;\textup{d}s\\
 & =2^{N}\sqrt{\frac{N}{2\pi\beta}}\int_{\IR}\exp\left(-\frac{N}{\beta}\frac{z^{2}}{2}\right)\tanh^{k}z\;\cosh^{N}z\;\textup{d}z\\
 & =2^{N}\sqrt{\frac{N}{2\pi\beta}}\int_{\IR}\exp\left(-N\left(\frac{z^{2}}{2\beta}-\ln\cosh z\right)\right)\tanh^{k}z\;\textup{d}z.
\end{align*}
The numerator in
\begin{equation}
\IE_{\beta,N}X_{1}\cdots X_{k}=\frac{\int_{\IR}\exp\left(-N\left(\frac{z^{2}}{2\beta}-\ln\cosh z\right)\right)\tanh^{k}z\;\textup{d}z}{\int_{\IR}\exp\left(-N\left(\frac{z^{2}}{2\beta}-\ln\cosh z\right)\right)\;\textup{d}z}\label{eq:corr_int}
\end{equation}
is 0 if and only if $k$ is odd, thus we have $\IE_{\beta,N}X_{1}\cdots X_{k}=0$
for all $\beta>0$ and all $k$ odd. From now on assume $k$ is even.
We determine the minima of the function
\[
F\left(z\right)\coloneq\frac{z^{2}}{2\beta}-\ln\cosh z,\quad z\in\IR.
\]
The first derivative of $F$ is
\[
F'\left(z\right)=\frac{z}{\beta}-\tanh z.
\]
Equating $F'$ to 0 leads to
\[
z=\beta\tanh z.
\]
For $x=\tanh z$, we get
\[
x=\tanh\left(\beta x\right),
\]
which we recognise as the Curie-Weiss equation (\ref{eq:CW}). Hence,
for all $0<\beta<1$, there is unique solution $x=0$, and
\[
F'\left(z\right)=0\quad\iff\quad z=0.
\]
The second derivative of $F$ is
\begin{equation}
F''\left(z\right)=\frac{1}{\beta}-\frac{1}{\cosh^{2}z}.\label{eq:F^(2)}
\end{equation}
We have
\[
F''\left(0\right)=\frac{1}{\beta}-1>0
\]
due to $\beta<1$, and $0$ is the unique minimum of $F$. Also, $F$
is strictly convex. For $\beta>1$, there are three critical points
of $F$: $-\textup{artanh }m\left(\beta\right),0$, and $l\left(\beta\right)\coloneq\textup{artanh }m\left(\beta\right)$.
Since we have
\[
F''\left(0\right)=\frac{1}{\beta}-1<0,
\]
the origin is a maximum of $F$ for $\beta>1$, and as $F''$ is a
strictly increasing function with $\lim_{z\rightarrow\infty}F''\left(z\right)=\frac{1}{\beta}>0$,
there is a unique critical point of $F$ on $\left(0,\infty\right)$,
which is $l\left(\beta\right)$. $l\left(\beta\right)$ must therefore
be a minimum.

We calculate derivatives of third of fourth order which we will use
in the following. The third derivative of $F$ is
\[
F'''\left(z\right)=-\frac{\tanh z}{\cosh^{2}z},
\]
and we note that, for all $z\in\IR$, $\left|F'''\left(z\right)\right|\leq1$
holds. The fourth derivative is
\[
F^{(4)}\left(z\right)=\frac{2-6\tanh^{2}z}{\cosh^{2}z},
\]
and for all $z\in\IR$, we have $\left|F^{(4)}\left(z\right)\right|\leq4$.

Assume $\beta<1$ holds. Set $r_{N}\coloneq\sqrt{\frac{\left(k+2\right)\ln N}{2b\left(N-1\right)}}$,
$N\in\IN$. Let $B_{r_{N}}\left(0\right)$ be the open ball around
the origin of radius $r_{N}$ and $B_{r_{N}}\left(0\right)^{c}$ its
complement. By the discussion above, $F$ has a unique minimum at
0 and is strictly convex, and its Taylor expansion with the remainder
in Lagrange form is
\[
F\left(z\right)=\frac{F''\left(z\right)}{2}z^{2}+\frac{F^{(4)}\left(\zeta\right)}{4!}z^{4},\quad z\in B_{r_{N}}\left(0\right),
\]
where $\zeta$ lies between 0 and $z$. Also we note that the fourth
derivative of $F$, $F^{(4)}$, is uniformly bounded on $\IR$. Hence,
there is a constant $b>0$ such that for all $N\in\IN$ and all $z\in B_{r_{N}}\left(0\right)^{c}$,
\[
F\left(z\right)\geq bz^{2}.
\]

In this proof, we will use the symbol $C$ for positive constants
which are independent of $N$ but may depend

on $\beta$ and $k$. The instances of $C$ may not all have the same
value.

As a first step, we separate the integrals in (\ref{eq:corr_int})
into two integrals over $B_{r_{N}}\left(0\right)$ and $B_{r_{N}}\left(0\right)^{c}$:
\begin{align*}
\int_{\IR}\exp\left(-NF\left(z\right)\right)\tanh^{k}z\;\textup{d}z & =\int_{B_{r_{N}}\left(0\right)}\exp\left(-NF\left(z\right)\right)\tanh^{k}z\;\textup{d}z+\int_{B_{r_{N}}\left(0\right)^{c}}\exp\left(-NF\left(z\right)\right)\tanh^{k}z\;\textup{d}z.
\end{align*}
We obtain an upper bound for the tail integral:
\begin{align*}
\int_{B_{r_{N}}\left(0\right)^{c}}\exp\left(-NF\left(z\right)\right)\tanh^{k}z\;\textup{d}z & \leq\exp\left(-\left(N-1\right)br_{N}^{2}\right)\int_{B_{r_{N}}\left(0\right)^{c}}\exp\left(-F\left(z\right)\right)\tanh^{k}z\;\textup{d}z\\
 & =C\left(\frac{\ln N}{N}\right)^{\frac{k+2}{2}}.
\end{align*}
Analogously,
\[
\int_{B_{r_{N}}\left(0\right)^{c}}\exp\left(-NF\left(z\right)\right)\;\textup{d}z\leq C\left(\frac{\ln N}{N}\right)^{\frac{k+2}{2}}.
\]
In the next step, using a Taylor expansion of the exponential function
and the boundedness of $F^{(4)}$, we have the bound
\[
\left|\exp\left(\frac{F^{(4)}\left(\zeta\right)}{4!}z^{4}\right)-1\right|\leq Cr_{N}^{4}=C\left(\frac{\ln N}{N}\right)^{2},\quad z\in B_{r_{N}}\left(0\right),
\]
and therefore
\begin{align*}
 & \quad\left|\int_{B_{r_{N}}\left(0\right)}\exp\left(-NF\left(z\right)\right)\tanh^{k}z\;\textup{d}z-\int_{B_{r_{N}}\left(0\right)}\exp\left(-N\frac{F''\left(z\right)}{2}z^{2}\right)\tanh^{k}z\;\textup{d}z\right|\\
 & \leq\int_{B_{r_{N}}\left(0\right)}CN\left(\frac{\ln N}{N}\right)^{2}\tanh^{k}z\;\textup{d}z\\
 & \leq CN\left(\frac{\ln N}{N}\right)^{2}\int_{B_{r_{N}}\left(0\right)}z^{k}\;\textup{d}z\\
 & \leq CN\left(\frac{\ln N}{N}\right)^{2}r_{N}^{k}\\
 & =CN\left(\frac{\ln N}{N}\right)^{\frac{k}{2}+2}\\
 & \leq C\frac{\left(\ln N\right)^{\frac{k}{2}+2}}{N^{\frac{k}{2}+1}}.
\end{align*}
Similarly,
\begin{align*}
\left|\int_{B_{r_{N}}\left(0\right)}\exp\left(-NF\left(z\right)\right)\;\textup{d}z-\int_{B_{r_{N}}\left(0\right)}\exp\left(-N\frac{F''\left(z\right)}{2}z^{2}\right)\;\textup{d}z\right| & \leq C\frac{\left(\ln N\right)^{2}}{N}.
\end{align*}
By the same procedure as above, it can be seen that
\[
\left|\int_{B_{r_{N}}\left(0\right)}\exp\left(-N\frac{F''\left(z\right)}{2}z^{2}\right)z^{k}\;\textup{d}z-\int_{\IR}\exp\left(-N\frac{F''\left(z\right)}{2}z^{2}\right)z^{k}\;\textup{d}z\right|\leq C\left(\frac{\ln N}{N}\right)^{\frac{k+2}{2}}
\]
and
\[
\left|\int_{B_{r_{N}}\left(0\right)}\exp\left(-N\frac{F''\left(z\right)}{2}z^{2}\right)\;\textup{d}z-\int_{\IR}\exp\left(-N\frac{F''\left(z\right)}{2}z^{2}\right)\;\textup{d}z\right|\leq C\left(\frac{\ln N}{N}\right)^{\frac{k+2}{2}}
\]
hold.

Putting together the bounds obtained so far, we have
\[
\left|\int_{\IR}\exp\left(-NF\left(z\right)\right)\tanh^{k}z\;\textup{d}z-\int_{\IR}\exp\left(-N\frac{F''\left(z\right)}{2}z^{2}\right)z^{k}\;\textup{d}z\right|\leq C\left(\frac{\ln N}{N}\right)^{\frac{k+2}{2}}
\]
and
\[
\left|\int_{\IR}\exp\left(-NF\left(z\right)\right)\;\textup{d}z-\int_{\IR}\exp\left(-N\frac{F''\left(z\right)}{2}z^{2}\right)\;\textup{d}z\right|\leq C\frac{\left(\ln N\right)^{2}}{N}.
\]

Using the last two bounds, we calculate
\begin{align*}
\left|\IE_{\beta,N}X_{1}\cdots X_{k}-\frac{\int_{\IR}\exp\left(-N\frac{F''\left(z\right)}{2}z^{2}\right)z^{k}\;\textup{d}z}{\int_{\IR}\exp\left(-N\frac{F''\left(z\right)}{2}z^{2}\right)\;\textup{d}z}\right|\leq C\left(\frac{\ln N}{N}\right)^{\frac{k+2}{2}} & .
\end{align*}
This supplies the claimed upper bound (\ref{eq:corr_UB_h}). Next
we calculate the value
\begin{align*}
\frac{\int_{\IR}\exp\left(-N\frac{F''\left(0\right)}{2}z^{2}\right)z^{k}\;\textup{d}z}{\int_{\IR}\exp\left(-N\frac{F''\left(z\right)}{2}z^{2}\right)\;\textup{d}z} & =\frac{\frac{1}{\left(F''\left(0\right)N\right)^{\frac{k+1}{2}}}\int_{\IR}\exp\left(-\frac{y^{2}}{2}\right)y^{k}\;\textup{d}y}{\frac{1}{\left(F''\left(0\right)N\right)^{\frac{1}{2}}}\int_{\IR}\exp\left(-\frac{y^{2}}{2}\right)\;\textup{d}y}\\
 & =\frac{1}{\left(F''\left(0\right)N\right)^{\frac{k}{2}}}\left(k-1\right)!!\\
 & =\left(k-1\right)!!\left(\frac{\beta}{1-\beta}\right)^{\frac{k}{2}}\frac{1}{N^{\frac{k}{2}}}.
\end{align*}
Now let $\beta>1$. We use a Taylor expansion of $F$:
\begin{align*}
F\left(z\right) & =F\left(l\left(\beta\right)\right)+F'\left(l\left(\beta\right)\right)\left(z-l\left(\beta\right)\right)+\frac{F''\left(l\left(\beta\right)\right)}{2}\left(z-l\left(\beta\right)\right)^{2}+\frac{F'''\left(\zeta_{1}\right)}{3!}\left(z-l\left(\beta\right)\right)^{3},
\end{align*}
and so
\[
F\left(z\right)-F\left(l\left(\beta\right)\right)=\frac{F''\left(l\left(\beta\right)\right)}{2}\left(z-l\left(\beta\right)\right)^{2}+\frac{F'''\left(\zeta_{1}\right)}{3!}\left(z-l\left(\beta\right)\right)^{3},
\]
where $\zeta_{1}$ lies between $l\left(\beta\right)$ and $z$. Next
we do the same for $\tanh^{k}$:
\[
\tanh^{k}z=\left(m\left(\beta\right)+\left(1-m\left(\beta\right)^{2}\right)\left(z-l\left(\beta\right)\right)-\tanh\zeta_{2}\left(1-\tanh^{2}\zeta_{2}\right)\left(z-l\left(\beta\right)\right)^{2}\right)^{k},
\]
where $\zeta_{2}$ lies between $l\left(\beta\right)$ and $z$. Note
that the derivatives of all orders of the $\tanh$ function are bounded
on $\IR$.

Using (\ref{eq:corr_int}), we obtain
\begin{align*}
\IE_{\beta,N}X_{1}\cdots X_{k} & =\frac{\int_{\IR}\exp\left(-N\left(F\left(z\right)-F\left(l\left(\beta\right)\right)\right)\right)\tanh^{k}z\;\textup{d}z}{\int_{\IR}\exp\left(-N\left(F\left(z\right)-F\left(l\left(\beta\right)\right)\right)\right)\;\textup{d}z}\\
 & =\frac{\int_{\IR}\exp\left(-N\left(\frac{F''\left(l\left(\beta\right)\right)}{2}\left(z-l\left(\beta\right)\right)^{2}+\frac{F'''\left(\zeta_{1}\right)}{3!}\left(z-l\left(\beta\right)\right)^{3}\right)\right)\tanh^{k}z\;\textup{d}z}{\int_{\IR}\exp\left(-N\left(\frac{F''\left(l\left(\beta\right)\right)}{2}\left(z-l\left(\beta\right)\right)^{2}+\frac{F'''\left(\zeta_{1}\right)}{3!}\left(z-l\left(\beta\right)\right)^{3}\right)\right)\;\textup{d}z}.
\end{align*}
Set $r_{N}\coloneq\sqrt{\frac{\left(k+2\right)\ln N}{2b\left(N-1\right)}}$,
$N\in\IN$. We proceed similarly to the $\beta<1$ case, and separate
the integrals above into integrals over open balls around the two
minima of $F$:
\begin{align*}
 & \quad\int_{\IR}\exp\left(-N\left(\frac{F''\left(l\left(\beta\right)\right)}{2}\left(z-l\left(\beta\right)\right)^{2}+\frac{F'''\left(\zeta_{1}\right)}{3!}\left(z-l\left(\beta\right)\right)^{3}\right)\right)\tanh^{k}z\;\textup{d}z\\
 & =\int_{B_{r_{N}}\left(-l\left(\beta\right)\right)}\exp\left(-N\left(\frac{F''\left(l\left(\beta\right)\right)}{2}\left(z-l\left(\beta\right)\right)^{2}+\frac{F'''\left(\zeta_{1}\right)}{3!}\left(z-l\left(\beta\right)\right)^{3}\right)\right)\tanh^{k}z\;\textup{d}z\\
 & \quad+\int_{B_{r_{N}}\left(l\left(\beta\right)\right)}\exp\left(-N\left(\frac{F''\left(l\left(\beta\right)\right)}{2}\left(z-l\left(\beta\right)\right)^{2}+\frac{F'''\left(\zeta_{1}\right)}{3!}\left(z-l\left(\beta\right)\right)^{3}\right)\right)\tanh^{k}z\;\textup{d}z\\
 & \quad+\int_{\left(B_{r_{N}}\left(-l\left(\beta\right)\right)\cup B_{r_{N}}\left(l\left(\beta\right)\right)\right)^{c}}\exp\left(-N\left(\frac{F''\left(l\left(\beta\right)\right)}{2}\left(z-l\left(\beta\right)\right)^{2}+\frac{F'''\left(\zeta_{1}\right)}{3!}\left(z-l\left(\beta\right)\right)^{3}\right)\right)\tanh^{k}z\;\textup{d}z.
\end{align*}
We treat the last summand first. There is a constant $b>0$ such that
for all $z\in\left(B_{r_{N}}\left(-l\left(\beta\right)\right)\cup B_{r_{N}}\left(l\left(\beta\right)\right)\right)^{c}$,
$\frac{F''\left(l\left(\beta\right)\right)}{2}\left(z-l\left(\beta\right)\right)^{2}+\frac{F'''\left(\zeta_{1}\right)}{3!}\left(z-l\left(\beta\right)\right)^{3}\geq b\left(z-l\left(\beta\right)\right)^{2}$
holds.
\begin{align*}
 & \quad\int_{\left(B_{r_{N}}\left(-l\left(\beta\right)\right)\cup B_{r_{N}}\left(l\left(\beta\right)\right)\right)^{c}}\exp\left(-N\left(\frac{F''\left(l\left(\beta\right)\right)}{2}\left(z-l\left(\beta\right)\right)^{2}+\frac{F'''\left(\zeta_{1}\right)}{3!}\left(z-l\left(\beta\right)\right)^{3}\right)\right)\tanh^{k}z\;\textup{d}z\\
 & \leq\text{\ensuremath{\exp\left(-\left(N-1\right)br_{N}^{2}\right)}}\\
 & \quad\cdot\int_{\left(B_{r_{N}}\left(-l\left(\beta\right)\right)\cup B_{r_{N}}\left(l\left(\beta\right)\right)\right)^{c}}\exp\left(-\left(\frac{F''\left(l\left(\beta\right)\right)}{2}\left(z-l\left(\beta\right)\right)^{2}+\frac{F'''\left(\zeta_{1}\right)}{3!}\left(z-l\left(\beta\right)\right)^{3}\right)\right)\tanh^{k}z\;\textup{d}z\\
 & =C\left(\frac{\ln N}{N}\right)^{\frac{k+2}{2}}.
\end{align*}
Similarly,
\[
\int_{\left(B_{r_{N}}\left(-l\left(\beta\right)\right)\cup B_{r_{N}}\left(l\left(\beta\right)\right)\right)^{c}}\exp\left(-N\frac{F''\left(l\left(\beta\right)\right)}{2}\left(z-l\left(\beta\right)\right)^{2}\right)z^{k}\;\textup{d}z\leq C\left(\frac{\ln N}{N}\right)^{\frac{k+2}{2}}
\]
and
\[
\int_{\left(B_{r_{N}}\left(-l\left(\beta\right)\right)\cup B_{r_{N}}\left(l\left(\beta\right)\right)\right)^{c}}\exp\left(-N\left(\frac{F''\left(l\left(\beta\right)\right)}{2}\left(z-l\left(\beta\right)\right)^{2}+\frac{F'''\left(\zeta_{1}\right)}{3!}\left(z-l\left(\beta\right)\right)^{3}\right)\right)\;\textup{d}z\leq C\left(\frac{\ln N}{N}\right)^{\frac{k+2}{2}}.
\]
Next we note that
\[
\left|\frac{F'''\left(\zeta_{1}\right)}{3!}\left(z-l\left(\beta\right)\right)^{3}\right|\leq Cr_{N}^{3}=C\left(\frac{\ln N}{N}\right)^{\frac{3}{2}},\quad z\in B_{r_{N}}\left(-l\left(\beta\right)\right)\cup B_{r_{N}}\left(l\left(\beta\right)\right),
\]
and the bounds
\begin{align*}
 & \left|\int_{B_{r_{N}}\left(-l\left(\beta\right)\right)\cup B_{r_{N}}\left(l\left(\beta\right)\right)}\exp\left(-N\left(\frac{F''\left(l\left(\beta\right)\right)}{2}\left(z-l\left(\beta\right)\right)^{2}+\frac{F'''\left(\zeta_{1}\right)}{3!}\left(z-l\left(\beta\right)\right)^{3}\right)\right)\tanh^{k}z\;\textup{d}z\right.\\
 & \quad\left.-\int_{B_{r_{N}}\left(-l\left(\beta\right)\right)\cup B_{r_{N}}\left(l\left(\beta\right)\right)}\exp\left(-N\frac{F''\left(l\left(\beta\right)\right)}{2}\left(z-l\left(\beta\right)\right)^{2}\right)\tanh^{k}z\;\textup{d}z\right|\\
 & \leq CN\left(\frac{\ln N}{N}\right)^{\frac{3}{2}}\int_{B_{r_{N}}\left(-l\left(\beta\right)\right)\cup B_{r_{N}}\left(l\left(\beta\right)\right)}\left(z-l\left(\beta\right)\right)^{3}\tanh^{k}z\;\textup{d}z\\
 & \leq CN\left(\frac{\ln N}{N}\right)^{\frac{3}{2}}\\
 & =C\frac{\left(\ln N\right)^{\frac{3}{2}}}{\sqrt{N}}
\end{align*}
and
\begin{align*}
 & \left|\int_{B_{r_{N}}\left(-l\left(\beta\right)\right)\cup B_{r_{N}}\left(l\left(\beta\right)\right)}\exp\left(-N\left(\frac{F''\left(l\left(\beta\right)\right)}{2}\left(z-l\left(\beta\right)\right)^{2}+\frac{F'''\left(\zeta_{1}\right)}{3!}\left(z-l\left(\beta\right)\right)^{3}\right)\right)\;\textup{d}z\right.\\
 & \quad\left.-\int_{B_{r_{N}}\left(-l\left(\beta\right)\right)\cup B_{r_{N}}\left(l\left(\beta\right)\right)}\exp\left(-N\frac{F''\left(l\left(\beta\right)\right)}{2}\left(z-l\left(\beta\right)\right)^{2}\right)\;\textup{d}z\right|\\
 & \leq C\frac{\left(\ln N\right)^{\frac{3}{2}}}{\sqrt{N}}
\end{align*}
 hold.

Now we employ the Taylor expansion for $\tanh^{k}$:
\begin{align*}
 & \int_{B_{r_{N}}\left(-l\left(\beta\right)\right)\cup B_{r_{N}}\left(l\left(\beta\right)\right)}\exp\left(-N\frac{F''\left(l\left(\beta\right)\right)}{2}\left(z-l\left(\beta\right)\right)^{2}\right)\tanh^{k}z\;\textup{d}z\\
 & \quad-\int_{B_{r_{N}}\left(-l\left(\beta\right)\right)\cup B_{r_{N}}\left(l\left(\beta\right)\right)}\exp\left(-N\frac{F''\left(l\left(\beta\right)\right)}{2}\left(z-l\left(\beta\right)\right)^{2}\right)m\left(\beta\right)^{k}\;\textup{d}z\\
 & \leq C\int_{B_{r_{N}}\left(-l\left(\beta\right)\right)\cup B_{r_{N}}\left(l\left(\beta\right)\right)}\exp\left(-N\frac{F''\left(l\left(\beta\right)\right)}{2}\left(z-l\left(\beta\right)\right)^{2}\right)\left(z-l\left(\beta\right)\right)^{3}\;\textup{d}z\\
 & \leq Cr_{N}^{3}\int_{B_{r_{N}}\left(-l\left(\beta\right)\right)\cup B_{r_{N}}\left(l\left(\beta\right)\right)}\exp\left(-N\frac{F''\left(l\left(\beta\right)\right)}{2}\left(z-l\left(\beta\right)\right)^{2}\right)\;\textup{d}z\\
 & \leq C\left(\frac{\ln N}{N}\right)^{\frac{3}{2}}.
\end{align*}
Putting together the bounds calculated above, we obtain
\begin{align*}
 & \left|\int_{\IR}\exp\left(-N\left(F\left(z\right)-F\left(l\left(\beta\right)\right)\right)\right)\tanh^{k}z\;\textup{d}z\right.\\
 & \quad\left.-m\left(\beta\right)^{k}\int_{B_{r_{N}}\left(-l\left(\beta\right)\right)\cup B_{r_{N}}\left(l\left(\beta\right)\right)}\exp\left(-N\frac{F''\left(l\left(\beta\right)\right)}{2}\left(z-l\left(\beta\right)\right)^{2}\right)\;\textup{d}z\right|\\
 & \leq C\frac{\left(\ln N\right)^{\frac{3}{2}}}{\sqrt{N}}
\end{align*}
and
\[
\left|\int_{\IR}\exp\left(-N\left(F\left(z\right)-F\left(l\left(\beta\right)\right)\right)\right)\;\textup{d}z-\int_{B_{r_{N}}\left(-l\left(\beta\right)\right)\cup B_{r_{N}}\left(l\left(\beta\right)\right)}\exp\left(-N\frac{F''\left(l\left(\beta\right)\right)}{2}\left(z-l\left(\beta\right)\right)^{2}\right)\;\textup{d}z\right|\leq C\frac{\left(\ln N\right)^{\frac{3}{2}}}{\sqrt{N}},
\]
and hence
\begin{align*}
 & \left|\IE_{\beta,N}X_{1}\cdots X_{k}-m\left(\beta\right)^{k}\right|\\
 & =\left|\IE_{\beta,N}X_{1}\cdots X_{k}-m\left(\beta\right)^{k}\frac{\int_{B_{r_{N}}\left(-l\left(\beta\right)\right)\cup B_{r_{N}}\left(l\left(\beta\right)\right)}\exp\left(-N\frac{F''\left(l\left(\beta\right)\right)}{2}\left(z-l\left(\beta\right)\right)^{2}\right)\;\textup{d}z}{\int_{B_{r_{N}}\left(-l\left(\beta\right)\right)\cup B_{r_{N}}\left(l\left(\beta\right)\right)}\exp\left(-N\frac{F''\left(l\left(\beta\right)\right)}{2}\left(z-l\left(\beta\right)\right)^{2}\right)\;\textup{d}z}\right|\leq C\frac{\left(\ln N\right)^{\frac{3}{2}}}{\sqrt{N}}.
\end{align*}
\end{proof}
We will use the fact that the variables $X_{1},X_{2},\ldots,X_{N}$
are exchangeable (see Definition \ref{def:exchange} and Lemma \ref{lem:exchange}).
In order to calculate expectations such as $\IE_{\beta,N}S^{2k}$,
we introduce the concept of profile vectors.
\begin{defn}
\label{def:ind_prof}Let $k,n\in\IN$ with $k\leq n$. We will call
all $\underbar{\ensuremath{\boldsymbol{i}}}\in\IN_{n}^{k}$ \emph{index
vectors} and set
\[
\Pi\coloneq\left\{ \left(r_{1},\ldots,r_{k}\right)\in\left(\IN_{k}\right)^{k}\,\left|\,\sum_{\ell=1}^{k}\ell r_{\ell}=k\right.\right\} ,
\]
and we will refer to $\Pi$ as the set of profile vectors and to the
elements $\underbar{\ensuremath{\boldsymbol{r}}}\in\Pi$ as \emph{profile
vectors}. For any index vector $\underbar{\ensuremath{\boldsymbol{i}}}\in\IN_{n}^{k}$,
the expression $\underbar{\ensuremath{\boldsymbol{r}}}\coloneq\left(r_{1},\ldots,r_{k}\right)\coloneq\underbar{\ensuremath{\boldsymbol{\rho}}}\left(\underbar{\ensuremath{\boldsymbol{i}}}\right)$
is defined as follows: for each $\ell\in\IN_{k}$, $r_{\ell}$ is
the number of indices in $\underbar{\ensuremath{\boldsymbol{i}}}$
that appear exactly $\ell$ times. We will call $\underbar{\ensuremath{\boldsymbol{r}}}$
the profile vector of $\underbar{\ensuremath{\boldsymbol{i}}}$.
\end{defn}

\begin{prop}
\label{prop:appr_moments}For all $\beta\in\IR,\beta\neq1$, and all
$k\in\IN$ , the moment $\IE_{\beta,N}S^{2k}$ is asymptotically equal
to
\[
\IE_{\beta,N}S^{2k}\approx\begin{cases}
\left(\frac{1}{1-\beta}\right)^{k}\,N^{k} & \text{if }\,\beta<1,\\
m\left(\beta\right)^{2k}\,N^{2k} & \text{if }\,\beta>1,
\end{cases}
\]
where $0<m\left(\beta\right)<1$ for $\beta>1$ is the constant from
Definition \ref{def:m_beta}.

Let the model be in the high temperature regime, i.e.\! $\beta<1$.
Then there is a positive constant $\boldsymbol{D}_{\textup{high}}$
such that for all $N\in\IN$
\[
\left|\IE_{\beta,N}\frac{S^{2k}}{N^{k}}-\left(\frac{1}{1-\beta}\right)^{k}\right|<\boldsymbol{D}_{\textup{high}}\frac{1}{\sqrt{N}}.
\]
Now let the model be in the low temperature regime, i.e.\! $\beta>1$.
Then there is a positive constant $\boldsymbol{D}_{\textup{low}}$
such that for all $N\in\IN$
\[
\left|\IE_{\beta,N}\frac{S^{2k}}{N^{2k}}-m\left(\beta\right)^{2k}\right|<\boldsymbol{D}_{\textup{low}}\frac{\left(\ln N\right)^{\frac{3}{2}}}{\sqrt{N}}.
\]
\end{prop}

\begin{proof}
Let $\beta<1$. We express the expectation $\IE_{\beta,N}S^{2k}$
as a sum of correlations $\IE_{\beta,N}X_{1}\cdots X_{\ell}$ of the
type treated in Proposition \ref{prop:appr_correlations}:
\[
\IE_{\beta,N}\frac{S^{2k}}{N^{k}}=\frac{1}{N^{k}}\sum_{i_{1},\ldots,i_{2k}=1}^{N}\IE_{\beta,N}X_{i_{1}}\cdots X_{i_{2k}}.
\]
Due to Lemma \ref{lem:exchange}, the specific values of the indices
$i_{1},\ldots,i_{2k}\in\IN_{N}$ only matter in as far as the number
of repeated indices and the multiplicity of each one is concerned.
Lemma \ref{lem:exchange} implies $\IE_{\beta,N}X_{i_{1}}\cdots X_{i_{2k}}=\IE_{\beta,N}X_{j_{1}}\cdots X_{j_{2k}}$
for all index vectors $\underbar{\ensuremath{\boldsymbol{i}}},\underbar{\ensuremath{\boldsymbol{j}}}$
with $\underbar{\ensuremath{\boldsymbol{\rho}}}\left(\underbar{\ensuremath{\boldsymbol{i}}}\right)=\underbar{\ensuremath{\boldsymbol{\rho}}}\left(\underbar{\ensuremath{\boldsymbol{j}}}\right)$.
Therefore, we can write for any profile $\underbar{\ensuremath{\boldsymbol{r}}}=\underbar{\ensuremath{\boldsymbol{\rho}}}\left(\underbar{\ensuremath{\boldsymbol{i}}}\right)$,
\[
\IE_{\beta,N}X\left(\underbar{\ensuremath{\boldsymbol{r}}}\right)\coloneq\IE_{\beta,N}X_{i_{1}}\cdots X_{i_{2k}}.
\]
For all index vectors $\underbar{\ensuremath{\boldsymbol{i}}}$, the
corresponding profile vector $\underbar{\ensuremath{\boldsymbol{\rho}}}\left(\underbar{\ensuremath{\boldsymbol{i}}}\right)=\left(r_{1},\ldots,r_{2k}\right)\in\left(\IN_{2k}\right)^{2k}$
satisfies the equality $\sum_{\ell=1}^{2k}\ell r_{\ell}=2k$. Thus,
we have $\underbar{\ensuremath{\boldsymbol{\rho}}}\left(\underbar{\ensuremath{\boldsymbol{i}}}\right)\in\Pi$
for all $\underbar{\ensuremath{\boldsymbol{i}}}\in\left(\IN_{N}\right)^{2k}$.

Recall Lemma \ref{lem:multiplicity}, which gives the number of index
vectors $\underbar{\ensuremath{\boldsymbol{i}}}$ that correspond
to a certain profile $\underbar{\ensuremath{\boldsymbol{r}}}$. We
use Lemma \ref{lem:multiplicity} to rewrite the sum above as
\[
\IE_{\beta,N}\frac{S^{2k}}{N^{k}}=\frac{1}{N^{k}}\sum_{\underbar{\ensuremath{\boldsymbol{r}}}\in\Pi}\frac{N!}{r_{1}!\cdots r_{2k}!\left(N-\sum_{\ell=1}^{2k}r_{\ell}\right)!}\frac{\left(2k\right)!}{1!^{r_{1}}\cdots\left(2k\right)!^{r_{2k}}}\IE_{\beta,N}X\left(\underbar{\ensuremath{\boldsymbol{r}}}\right).
\]
Now we will partition the set of all profile vectors $\Pi$ into two
subsets
\[
\Pi_{1}\coloneq\left\{ \left.\left(r_{1},\ldots,r_{2k}\right)\in\Pi\,\right|\,\exists\ell>2:\,r_{\ell}>0\right\} \quad\textup{and}\quad\Pi_{2}\coloneq\left\{ \left.\left(r_{1},\ldots,r_{2k}\right)\in\Pi\,\right|\,\forall\ell>2:\,r_{\ell}=0\right\} .
\]
We show that the profile vectors belonging to $\Pi_{1}$ do not contribute
asymptotically to the expectation $\IE_{\beta,N}\frac{S^{2k}}{N^{k}}$.
Let $\underbar{\ensuremath{\boldsymbol{r}}}\in\Pi_{1}$, and set $o\coloneq\sum_{\ell=1}^{k}r_{2\ell-1}$.
We note that, due to $X_{i}^{2\ell}=1$ and $X_{i}^{2\ell+1}=X_{i}$
for all $\ell\in\IN$ and all $i\in\IN_{N}$,
\[
\IE_{\beta,N}X\left(\underbar{\ensuremath{\boldsymbol{r}}}\right)=\IE_{\beta,N}X_{1}\cdots X_{o}
\]
holds. We will use Proposition \ref{prop:appr_correlations}, and
the upper bound for the correlation
\[
\left|\IE_{\beta,N}X_{1}\cdots X_{o}\right|\leq\left(o-1\right)!!\left(\frac{\left|\beta\right|}{1-\beta}\right)^{\frac{o}{2}}\,\frac{1}{N^{\frac{o}{2}}}+\boldsymbol{C}_{\textup{high}}\left(\frac{\ln N}{N}\right)^{\frac{o+2}{2}},
\]
which by said proposition holds no matter the parity of $o$. Therefore,
\begin{align*}
 & \left|\frac{1}{N^{k}}\frac{N!}{r_{1}!\cdots r_{2k}!\left(N-\sum_{\ell=1}^{2k}r_{\ell}\right)!}\frac{\left(2k\right)!}{1!^{r_{1}}\cdots\left(2k\right)!^{r_{2k}}}\IE_{\beta,N}X\left(\underbar{\ensuremath{\boldsymbol{r}}}\right)\right|\\
 & \leq C\frac{1}{N^{k}}\frac{N!}{\left(N-\sum_{\ell=1}^{2k}r_{\ell}\right)!}\left|\IE_{\beta,N}X_{1}\cdots X_{o}\right|\\
 & \leq C\frac{1}{N^{k}}N^{\sum_{\ell=1}^{2k}r_{\ell}}\left[\left(o-1\right)!!\left(\frac{\left|\beta\right|}{1-\beta}\right)^{\frac{o}{2}}\,\frac{1}{N^{\frac{o}{2}}}+\boldsymbol{C}_{\textup{high}}\left(\frac{\ln N}{N}\right)^{\frac{o+2}{2}}\right]
\end{align*}
is satisfied. The assumption $\underbar{\ensuremath{\boldsymbol{r}}}\in\Pi_{1}$
and the implied existence of some $\ell^{*}>2$ with $r_{\ell^{*}}>0$
allow us to find an upper bound for
\begin{align*}
\sum_{\ell=1}^{2k}r_{\ell}-\frac{o}{2} & =\frac{o}{2}+\sum_{\ell=1}^{k}r_{2\ell}\\
 & =\frac{1}{2}\left(\sum_{\ell=1}^{k}r_{2\ell-1}+\sum_{\ell=1}^{k}2r_{2\ell}\right).
\end{align*}

Suppose first that $\ell^{*}=2j^{*}-1$ is odd and therefore $j^{*}\geq2$.
Then
\begin{align*}
\frac{1}{2}\left(\sum_{\ell=1}^{k}r_{2\ell-1}+\sum_{\ell=1}^{k}2r_{2\ell}\right) & =\frac{1}{2}\left(r_{2j^{*}-1}+\sum_{1\leq j\leq k,j\neq j^{*}}r_{2j-1}+\sum_{\ell=1}^{k}2r_{2\ell}\right)\\
 & <\frac{1}{2}\left(\left(2j^{*}-2\right)r_{2j^{*}-1}+\sum_{1\leq j\leq k,j\neq j^{*}}r_{2j-1}+\sum_{\ell=1}^{k}2r_{2\ell}\right)\\
 & =\frac{1}{2}\left(\left(2j^{*}-1\right)r_{2j^{*}-1}+\sum_{1\leq j\leq k,j\neq j^{*}}r_{2j-1}+\sum_{\ell=1}^{k}2r_{2\ell}\right)-\frac{1}{2}r_{2j^{*}-1}\\
 & \leq\frac{1}{2}\left(\sum_{\ell=1}^{k}\left(2j-1\right)r_{2j-1}+\sum_{\ell=1}^{k}2\ell r_{2\ell}\right)-\frac{1}{2}\\
 & =\frac{1}{2}\sum_{\ell=1}^{2k}\ell r_{\ell}-\frac{1}{2}\\
 & =k-\frac{1}{2}.
\end{align*}
Next suppose $\ell^{*}=2j^{*}$ is even and therefore $j^{*}\geq1$.
Then
\begin{align*}
\frac{1}{2}\left(\sum_{\ell=1}^{k}r_{2\ell-1}+\sum_{\ell=1}^{k}2r_{2\ell}\right) & =\frac{1}{2}\left(r_{2j^{*}}+\sum_{\ell=1}^{k}r_{2\ell-1}+\sum_{1\leq j\leq k,j\neq j^{*}}2r_{2j}\right)\\
 & \leq\frac{1}{2}\left(\left(2j^{*}-1\right)r_{2j^{*}}+\sum_{\ell=1}^{k}r_{2\ell-1}+\sum_{1\leq j\leq k,j\neq j^{*}}2r_{2j}\right)\\
 & =\frac{1}{2}\left(2j^{*}r_{2j^{*}}+\sum_{\ell=1}^{k}r_{2\ell-1}+\sum_{1\leq j\leq k,j\neq j^{*}}2r_{2j}\right)-\frac{1}{2}r_{2j^{*}}\\
 & \leq\frac{1}{2}\left(\sum_{\ell=1}^{k}\left(2\ell-1\right)r_{2\ell-1}+\sum_{\ell=1}^{k}2\ell r_{2\ell}\right)-\frac{1}{2}\\
 & =\frac{1}{2}\sum_{\ell=1}^{2k}\ell r_{\ell}-\frac{1}{2}\\
 & =k-\frac{1}{2}.
\end{align*}
So we have the same upper bound $k-\frac{1}{2}$ for $\sum_{\ell=1}^{2k}r_{\ell}-\frac{o}{2}$
independently of the parity of $\ell^{*}$. We continue with the process
of finding an upper bound for
\begin{align}
 & \left|\frac{1}{N^{k}}\frac{N!}{r_{1}!\cdots r_{2k}!\left(N-\sum_{\ell=1}^{2k}r_{\ell}\right)!}\frac{\left(2k\right)!}{1!^{r_{1}}\cdots\left(2k\right)!^{r_{2k}}}\IE_{\beta,N}X\left(\underbar{\ensuremath{\boldsymbol{r}}}\right)\right|\nonumber \\
 & \leq C\frac{1}{N^{k}}N^{\sum_{\ell=1}^{2k}r_{\ell}}\left[\left(o-1\right)!!\left(\frac{\left|\beta\right|}{1-\beta}\right)^{\frac{o}{2}}\,\frac{1}{N^{\frac{o}{2}}}+\boldsymbol{C}_{\textup{high}}\left(\frac{\ln N}{N}\right)^{\frac{o+2}{2}}\right]\nonumber \\
 & \leq C\frac{1}{N^{k}}N^{\sum_{\ell=1}^{2k}r_{\ell}}\frac{1}{N^{\frac{o}{2}}}\nonumber \\
 & =CN^{\sum_{\ell=1}^{2k}r_{\ell}-k-\frac{o}{2}}\nonumber \\
 & \leq C\frac{1}{\sqrt{N}}.\label{eq:UB_Pi_1}
\end{align}
The previous upper bound, which shows that any summand with $\underbar{\ensuremath{\boldsymbol{r}}}\in\Pi_{1}$
decays to 0 as $N\rightarrow\infty$, is complemented by the next
calculation which will show that for any $\underbar{\ensuremath{\boldsymbol{r}}}\in\Pi_{2}$,
the corresponding summand converges to a positive constant, for which
we will find bounds under the assumption that $0<\beta<1$. Let $\underbar{\ensuremath{\boldsymbol{r}}}\in\Pi_{2}$.
We first show that $o=\sum_{\ell=1}^{k}r_{2\ell-1}$ is even:
\[
2k=\sum_{\ell=1}^{2k}\ell r_{\ell}=\sum_{\ell=1}^{k}\left(2\ell-1\right)r_{2\ell-1}+\sum_{\ell=1}^{k}2\ell r_{2\ell}.
\]
Since $\sum_{\ell=1}^{k}2\ell r_{2\ell}$ and $2k$ are even, so is
their difference $\sum_{\ell=1}^{k}\left(2\ell-1\right)r_{2\ell-1}$.
We also have
\[
\sum_{\ell=1}^{k}\left(2\ell-1\right)r_{2\ell-1}=\sum_{\ell=1}^{k}2\ell r_{2\ell-1}-\sum_{\ell=1}^{k}r_{2\ell-1},
\]
where once again the left hand side is even and so is $\sum_{\ell=1}^{k}2\ell r_{2\ell-1}$,
hence $o=\sum_{\ell=1}^{k}r_{2\ell-1}$ is even. Also, we have $o=r_{1}$,
since $\underbar{\ensuremath{\boldsymbol{r}}}\in\Pi_{2}$ is assumed.
Therefore, $r_{2}=\frac{2k-r_{1}}{2}=k-\frac{r_{1}}{2}$ holds.

Then
\begin{align}
 & \frac{1}{N^{k}}\frac{N!}{r_{1}!\cdots r_{2k}!\left(N-\sum_{\ell=1}^{2k}r_{\ell}\right)!}\frac{\left(2k\right)!}{1!^{r_{1}}\cdots\left(2k\right)!^{r_{2k}}}\IE_{\beta,N}X\left(\underbar{\ensuremath{\boldsymbol{r}}}\right)\nonumber \\
 & =\frac{1}{N^{k}}\frac{N!}{r_{1}!\left(k-\frac{r_{1}}{2}\right)!\left(N-\sum_{\ell=1}^{2k}r_{\ell}\right)!}\frac{\left(2k\right)!}{2^{2k-r_{1}}}\IE_{\beta,N}X_{1}\cdots X_{r_{1}}\nonumber \\
 & \approx\frac{1}{r_{1}!\left(k-\frac{r_{1}}{2}\right)!}\frac{\left(2k\right)!}{2^{k-\frac{r_{1}}{2}}}\left(r_{1}-1\right)!!\left(\frac{\beta}{1-\beta}\right)^{\frac{r_{1}}{2}}\frac{N^{k+\frac{r_{1}}{2}}}{N^{k}}\,\frac{1}{N^{\frac{r_{1}}{2}}}\nonumber \\
 & =\frac{1}{r_{1}!\left(k-\frac{r_{1}}{2}\right)!}\frac{\left(2k\right)!}{2^{k-\frac{r_{1}}{2}}}\left(r_{1}-1\right)!!\left(\frac{\beta}{1-\beta}\right)^{\frac{r_{1}}{2}}.\label{eq:asymp_Pi_2}
\end{align}
For the next calculation, assume $0<\beta<1$. Then, due to Proposition
\ref{prop:appr_correlations}, an upper bound for the summand is given
by
\begin{align}
 & \frac{1}{N^{k}}\frac{N!}{r_{1}!\cdots r_{2k}!\left(N-\sum_{\ell=1}^{2k}r_{\ell}\right)!}\frac{\left(2k\right)!}{1!^{r_{1}}\cdots\left(2k\right)!^{r_{2k}}}\IE_{\beta,N}X\left(\underbar{\ensuremath{\boldsymbol{r}}}\right)\nonumber \\
 & =\frac{1}{N^{k}}\frac{N!}{r_{1}!\left(k-\frac{r_{1}}{2}\right)!\left(N-\sum_{\ell=1}^{2k}r_{\ell}\right)!}\frac{\left(2k\right)!}{2^{k-\frac{r_{1}}{2}}}\IE_{\beta,N}X_{1}\cdots X_{r_{1}}\nonumber \\
 & \leq\frac{1}{N^{k}}N^{k+\frac{r_{1}}{2}}\frac{1}{r_{1}!\left(k-\frac{r_{1}}{2}\right)!}\frac{\left(2k\right)!}{2^{k-\frac{r_{1}}{2}}}\left[\left(r_{1}-1\right)!!\left(\frac{\beta}{1-\beta}\right)^{\frac{r_{1}}{2}}\,\frac{1}{N^{\frac{r_{1}}{2}}}+\boldsymbol{C}_{\textup{high}}\left(\frac{\ln N}{N}\right)^{\frac{r_{1}+2}{2}}\right]\nonumber \\
 & =\frac{1}{r_{1}!\left(k-\frac{r_{1}}{2}\right)!}\frac{\left(2k\right)!}{2^{k-\frac{r_{1}}{2}}}\left(r_{1}-1\right)!!\left(\frac{\beta}{1-\beta}\right)^{\frac{r_{1}}{2}}+\boldsymbol{C}_{\textup{high}}\frac{\left(\ln N\right)^{\frac{r_{1}+2}{2}}}{N},\label{eq:UB_Pi_2}
\end{align}
and a lower bound is
\begin{align}
 & \frac{1}{N^{k}}\frac{N!}{r_{1}!\cdots r_{2k}!\left(N-\sum_{\ell=1}^{2k}r_{\ell}\right)!}\frac{\left(2k\right)!}{1!^{r_{1}}\cdots\left(2k\right)!^{r_{2k}}}\IE_{\beta,N}X\left(\underbar{\ensuremath{\boldsymbol{r}}}\right)\nonumber \\
 & \geq\frac{1}{N^{k}}\left(N-\left(k+\frac{r_{1}}{2}\right)\right)^{k+\frac{r_{1}}{2}}\frac{1}{r_{1}!\left(k-\frac{r_{1}}{2}\right)!}\frac{\left(2k\right)!}{2^{k-\frac{r_{1}}{2}}}\left[\left(r_{1}-1\right)!!\left(\frac{\beta}{1-\beta}\right)^{\frac{r_{1}}{2}}\,\frac{1}{N^{\frac{r_{1}}{2}}}-\boldsymbol{C}_{\textup{high}}\left(\frac{\ln N}{N}\right)^{\frac{r_{1}+2}{2}}\right]\nonumber \\
 & =\left(1-\frac{k+\frac{r_{1}}{2}}{N}\right)^{k+\frac{r_{1}}{2}}\left[\frac{1}{r_{1}!\left(k-\frac{r_{1}}{2}\right)!}\frac{\left(2k\right)!}{2^{k-\frac{r_{1}}{2}}}\left(r_{1}-1\right)!!\left(\frac{\beta}{1-\beta}\right)^{\frac{r_{1}}{2}}-\boldsymbol{C}_{\textup{high}}\frac{\left(\ln N\right)^{\frac{r_{1}+2}{2}}}{N}\right].\label{eq:LB_Pi_2}
\end{align}
Since the number of different profile vectors is finite and independent
of $N$, the cardinalities $\left|\Pi_{1}\right|,\left|\Pi_{2}\right|$
are constant. Therefore, using (\ref{eq:UB_Pi_1}) and (\ref{eq:asymp_Pi_2}),
we obtain
\begin{align*}
\IE_{\beta,N}\frac{S^{2k}}{N^{k}} & =\frac{1}{N^{k}}\sum_{\underbar{\ensuremath{\boldsymbol{r}}}\in\Pi}\frac{N!}{r_{1}!\cdots r_{2k}!\left(N-\sum_{\ell=1}^{2k}r_{\ell}\right)!}\frac{\left(2k\right)!}{1!^{r_{1}}\cdots\left(2k\right)!^{r_{2k}}}\IE_{\beta,N}X\left(\underbar{\ensuremath{\boldsymbol{r}}}\right)\\
 & =\frac{1}{N^{k}}\sum_{\underbar{\ensuremath{\boldsymbol{r}}}\in\Pi_{1}}\frac{N!}{r_{1}!\cdots r_{2k}!\left(N-\sum_{\ell=1}^{2k}r_{\ell}\right)!}\frac{\left(2k\right)!}{1!^{r_{1}}\cdots\left(2k\right)!^{r_{2k}}}\IE_{\beta,N}X\left(\underbar{\ensuremath{\boldsymbol{r}}}\right)\\
 & \quad+\frac{1}{N^{k}}\sum_{\underbar{\ensuremath{\boldsymbol{r}}}\in\Pi_{2}}\frac{N!}{r_{1}!\cdots r_{2k}!\left(N-\sum_{\ell=1}^{2k}r_{\ell}\right)!}\frac{\left(2k\right)!}{1!^{r_{1}}\cdots\left(2k\right)!^{r_{2k}}}\IE_{\beta,N}X\left(\underbar{\ensuremath{\boldsymbol{r}}}\right)\\
 & \leq C_{1}\frac{1}{\sqrt{N}}+\sum_{r_{1}=0}^{k}\left[\frac{1}{r_{1}!\left(k-\frac{r_{1}}{2}\right)!}\frac{\left(2k\right)!}{2^{k-\frac{r_{1}}{2}}}\left(r_{1}-1\right)!!\left(\frac{\beta}{1-\beta}\right)^{\frac{r_{1}}{2}}+\boldsymbol{C}_{\textup{high}}\frac{\left(\ln N\right)^{\frac{r_{1}+2}{2}}}{N}\right]\\
 & =C_{1}\frac{1}{\sqrt{N}}+C_{2}\frac{\left(\ln N\right)^{\frac{r_{1}+2}{2}}}{N}+\sum_{r_{1}=0}^{k}\frac{1}{r_{1}!\left(k-\frac{r_{1}}{2}\right)!}\frac{\left(2k\right)!}{2^{k-\frac{r_{1}}{2}}}\left(r_{1}-1\right)!!\left(\frac{\beta}{1-\beta}\right)^{\frac{r_{1}}{2}}\\
 & \leq\sum_{r_{1}=0}^{k}\frac{1}{r_{1}!\left(k-\frac{r_{1}}{2}\right)!}\frac{\left(2k\right)!}{2^{k-\frac{r_{1}}{2}}}\left(r_{1}-1\right)!!\left(\frac{\beta}{1-\beta}\right)^{\frac{r_{1}}{2}}+C\frac{1}{\sqrt{N}}.
\end{align*}
A lower bound can be calculated in similar fashion:
\begin{align*}
\IE_{\beta,N}\frac{S^{2k}}{N^{k}} & =\frac{1}{N^{k}}\sum_{\underbar{\ensuremath{\boldsymbol{r}}}\in\Pi}\frac{N!}{r_{1}!\cdots r_{2k}!\left(N-\sum_{\ell=1}^{2k}r_{\ell}\right)!}\frac{\left(2k\right)!}{1!^{r_{1}}\cdots\left(2k\right)!^{r_{2k}}}\IE_{\beta,N}X\left(\underbar{\ensuremath{\boldsymbol{r}}}\right)\\
 & =\frac{1}{N^{k}}\sum_{\underbar{\ensuremath{\boldsymbol{r}}}\in\Pi_{1}}\frac{N!}{r_{1}!\cdots r_{2k}!\left(N-\sum_{\ell=1}^{2k}r_{\ell}\right)!}\frac{\left(2k\right)!}{1!^{r_{1}}\cdots\left(2k\right)!^{r_{2k}}}\IE_{\beta,N}X\left(\underbar{\ensuremath{\boldsymbol{r}}}\right)\\
 & \quad+\frac{1}{N^{k}}\sum_{\underbar{\ensuremath{\boldsymbol{r}}}\in\Pi_{2}}\frac{N!}{r_{1}!\cdots r_{2k}!\left(N-\sum_{\ell=1}^{2k}r_{\ell}\right)!}\frac{\left(2k\right)!}{1!^{r_{1}}\cdots\left(2k\right)!^{r_{2k}}}\IE_{\beta,N}X\left(\underbar{\ensuremath{\boldsymbol{r}}}\right)\\
 & \geq-C_{1}\frac{1}{\sqrt{N}}+\sum_{r_{1}=0}^{k}\left(1-\frac{k+\frac{r_{1}}{2}-1}{N}\right)^{k+\frac{r_{1}}{2}}\\
 & \quad\cdot\left[\frac{1}{r_{1}!\left(k-\frac{r_{1}}{2}\right)!}\frac{\left(2k\right)!}{2^{k-\frac{r_{1}}{2}}}\left(r_{1}-1\right)!!\left(\frac{\beta}{1-\beta}\right)^{\frac{r_{1}}{2}}-\boldsymbol{C}_{\textup{high}}\frac{\left(\ln N\right)^{\frac{r_{1}+2}{2}}}{N}\right]\\
 & =-C_{1}\frac{1}{\sqrt{N}}-C_{2}\frac{\left(\ln N\right)^{\frac{r_{1}+2}{2}}}{N}+\left(1-\frac{3k-1}{2N}\right)^{\frac{3k}{2}}\sum_{r_{1}=0}^{k}\frac{1}{r_{1}!\left(k-\frac{r_{1}}{2}\right)!}\frac{\left(2k\right)!}{2^{k-\frac{r_{1}}{2}}}\left(r_{1}-1\right)!!\left(\frac{\beta}{1-\beta}\right)^{\frac{r_{1}}{2}}\\
 & \geq-C_{1}\frac{1}{\sqrt{N}}-C_{2}\frac{\left(\ln N\right)^{\frac{r_{1}+2}{2}}}{N}+\sum_{r_{1}=0}^{k}\frac{1}{r_{1}!\left(k-\frac{r_{1}}{2}\right)!}\frac{\left(2k\right)!}{2^{k-\frac{r_{1}}{2}}}\left(r_{1}-1\right)!!\left(\frac{\beta}{1-\beta}\right)^{\frac{r_{1}}{2}}-C_{3}\frac{1}{N^{\frac{3k}{2}}}\\
 & \geq\sum_{r_{1}=0}^{k}\frac{1}{r_{1}!\left(k-\frac{r_{1}}{2}\right)!}\frac{\left(2k\right)!}{2^{k-\frac{r_{1}}{2}}}\left(r_{1}-1\right)!!\left(\frac{\beta}{1-\beta}\right)^{\frac{r_{1}}{2}}-C\frac{1}{\sqrt{N}}.
\end{align*}
Putting together the two bounds, we obtain the limit
\[
\IE_{\beta,N}\frac{S^{2k}}{N^{k}}\xrightarrow[N\rightarrow\infty]{}\sum_{r_{1}=0}^{k}\frac{1}{r_{1}!\left(k-\frac{r_{1}}{2}\right)!}\frac{\left(2k\right)!}{2^{k-\frac{r_{1}}{2}}}\left(r_{1}-1\right)!!\left(\frac{\beta}{1-\beta}\right)^{\frac{r_{1}}{2}},
\]
which we proceed to simplify.
\begin{align*}
 & \sum_{r_{1}=1}^{k}\frac{1}{r_{1}!\left(k-\frac{r_{1}}{2}\right)!}\frac{\left(2k\right)!}{2^{k-\frac{r_{1}}{2}}}\left(r_{1}-1\right)!!\left(\frac{\beta}{1-\beta}\right)^{\frac{r_{1}}{2}}\\
 & =\frac{\left(2k\right)!}{k!\,2^{k}}\sum_{r_{1}=0}^{k}\frac{k!}{r_{1}!\left(k-\frac{r_{1}}{2}\right)!}\frac{1}{2^{-\frac{r_{1}}{2}}}\frac{r_{1}!}{\left(\frac{r_{1}}{2}\right)!2^{\frac{r_{1}}{2}}}\left(\frac{\beta}{1-\beta}\right)^{\frac{r_{1}}{2}}\\
 & =\left(2k-1\right)!!\sum_{r_{1}=0}^{k}\frac{k!}{\left(\frac{r_{1}}{2}\right)!\left(k-\frac{r_{1}}{2}\right)!}\left(\frac{\beta}{1-\beta}\right)^{\frac{r_{1}}{2}}\\
 & =\left(2k-1\right)!!\sum_{r_{1}=0}^{k}\left(\begin{array}{c}
k\\
\frac{r_{1}}{2}
\end{array}\right)\left(\frac{\beta}{1-\beta}\right)^{\frac{r_{1}}{2}}\\
 & =\left(2k-1\right)!!\left(1+\frac{\beta}{1-\beta}\right)^{k}\\
 & =\left(2k-1\right)!!\left(\frac{1}{1-\beta}\right)^{k},
\end{align*}
where we used the equality
\[
m!!=\frac{m!}{\left(\frac{m}{2}\right)!\,2^{m}},
\]
which holds for all even $m\in\IN$. This concludes the proof of the
statements concerning $\beta<1$.

Assume now $\beta>1$. We once again use the expression of the expectation
$\IE_{\beta,N}\frac{S^{2k}}{N^{2k}}$ as a sum over all possible profile
vectors of the correlation $\IE_{\beta,N}X_{i_{1}}\cdots X_{i_{2k}}$:
\[
\IE_{\beta,N}\frac{S^{2k}}{N^{2k}}=\frac{1}{N^{2k}}\sum_{\underbar{\ensuremath{\boldsymbol{r}}}\in\Pi}\frac{N!}{r_{1}!\cdots r_{2k}!\left(N-\sum_{\ell=1}^{2k}r_{\ell}\right)!}\frac{\left(2k\right)!}{1!^{r_{1}}\cdots\left(2k\right)!^{r_{2k}}}\IE_{\beta,N}X\left(\underbar{\ensuremath{\boldsymbol{r}}}\right).
\]
We define the partition of $\Pi$
\[
\Pi_{3}\coloneq\left\{ \left.\left(r_{1},\ldots,r_{2k}\right)\in\Pi\,\right|\,\exists\ell>1:\,r_{\ell}>0\right\} \quad\textup{and}\quad\Pi_{4}\coloneq\left\{ \left.\left(r_{1},\ldots,r_{2k}\right)\in\Pi\,\right|\,\forall\ell>1:\,r_{\ell}=0\right\} .
\]
Let $\underbar{\ensuremath{\boldsymbol{r}}}\in\Pi_{3}$, and assume
$r_{\ell^{*}}>0$ holds for $\ell^{*}>1$. We first note
\begin{align*}
\sum_{\ell=1}^{2k}r_{\ell} & =r_{\ell^{*}}+\sum_{1\leq\ell\leq2k,\ell\neq\ell^{*}}r_{\ell}\\
 & \leq\left(\ell^{*}-1\right)r_{\ell^{*}}+\sum_{1\leq j\leq k,j\neq j^{*}}r_{\ell}\\
 & \leq\left(\ell^{*}-1\right)r_{\ell^{*}}+\sum_{1\leq j\leq k,j\neq j^{*}}\ell r_{\ell}\\
 & =-r_{\ell^{*}}+\sum_{\ell=1}^{2k}\ell r_{\ell}\\
 & \leq-1+\sum_{\ell=1}^{2k}\ell r_{\ell}\\
 & =2k-1.
\end{align*}
Using the bound given above and Proposition \ref{prop:appr_correlations},
we obtain the upper bound
\begin{align}
 & \frac{1}{N^{2k}}\frac{N!}{r_{1}!\cdots r_{2k}!\left(N-\sum_{\ell=1}^{2k}r_{\ell}\right)!}\frac{\left(2k\right)!}{1!^{r_{1}}\cdots\left(2k\right)!^{r_{2k}}}\IE_{\beta,N}X\left(\underbar{\ensuremath{\boldsymbol{r}}}\right)\nonumber \\
 & \leq\frac{1}{N^{2k}}\frac{N^{\sum_{\ell=1}^{2k}r_{\ell}}}{r_{1}!\cdots r_{2k}!}\frac{\left(2k\right)!}{1!^{r_{1}}\cdots\left(2k\right)!^{r_{2k}}}\left[m\left(\beta\right)^{\sum_{\ell=1}^{k}\left(2\ell-1\right)r_{2\ell-1}}+\boldsymbol{C}_{\textup{low}}\frac{\left(\ln N\right)^{\frac{3}{2}}}{\sqrt{N}}\right]\nonumber \\
 & \leq\frac{1}{N}C\left[m\left(\beta\right)^{\sum_{\ell=1}^{k}\left(2\ell-1\right)r_{2\ell-1}}+\boldsymbol{C}_{\textup{low}}\frac{\left(\ln N\right)^{\frac{3}{2}}}{\sqrt{N}}\right]\nonumber \\
 & \leq\frac{1}{N}C.\label{eq:UB_low_temp_Pi_3}
\end{align}
Let $\underbar{\ensuremath{\boldsymbol{r}}}\in\Pi_{4}$. Then, by
definition of $\Pi_{4}$, $\underbar{\ensuremath{\boldsymbol{r}}}=\left(2k,0,\ldots,0\right)$,
and
\begin{align}
\frac{1}{N^{2k}}\frac{N!}{r_{1}!\cdots r_{2k}!\left(N-\sum_{\ell=1}^{2k}r_{\ell}\right)!}\frac{\left(2k\right)!}{1!^{r_{1}}\cdots\left(2k\right)!^{r_{2k}}}\IE_{\beta,N}X\left(\underbar{\ensuremath{\boldsymbol{r}}}\right) & \approx\frac{1}{N^{2k}}\frac{N^{2k}}{\left(2k\right)!}\left(2k\right)!\,m\left(\beta\right)^{2k}\nonumber \\
 & =m\left(\beta\right)^{2k}.\label{eq:asymp_Pi_2_low}
\end{align}
So, asymptotically, only the single summand with $\underbar{\ensuremath{\boldsymbol{r}}}=\left(2k,0,\ldots,0\right)$
contributes to the expectation $\IE_{\beta,N}\frac{S^{2k}}{N^{2k}}$.
We calculate the bounds
\begin{align*}
\IE_{\beta,N}\frac{S^{2k}}{N^{2k}} & =\frac{1}{N^{2k}}\sum_{\underbar{\ensuremath{\boldsymbol{r}}}\in\Pi}\frac{N!}{r_{1}!\cdots r_{2k}!\left(N-\sum_{\ell=1}^{2k}r_{\ell}\right)!}\frac{\left(2k\right)!}{1!^{r_{1}}\cdots\left(2k\right)!^{r_{2k}}}\IE_{\beta,N}X\left(\underbar{\ensuremath{\boldsymbol{r}}}\right)\\
 & \leq C\frac{1}{N}+\frac{1}{N^{2k}}\frac{N!}{r_{1}!\cdots r_{2k}!\left(N-\sum_{\ell=1}^{2k}r_{\ell}\right)!}\frac{\left(2k\right)!}{1!^{r_{1}}\cdots\left(2k\right)!^{r_{2k}}}\IE_{\beta,N}X\left(\left(2k,0,\ldots,0\right)\right)\\
 & \leq C\frac{1}{N}+\frac{1}{N^{2k}}\frac{N^{2k}}{\left(2k\right)!}\left(2k\right)!\left[m\left(\beta\right)^{2k}+\boldsymbol{C}_{\textup{low}}\frac{\left(\ln N\right)^{\frac{3}{2}}}{\sqrt{N}}\right]\\
 & \leq C_{1}\frac{1}{N}+m\left(\beta\right)^{2k}+C_{2}\frac{\left(\ln N\right)^{\frac{3}{2}}}{\sqrt{N}}\\
 & \leq m\left(\beta\right)^{2k}+C\frac{\left(\ln N\right)^{\frac{3}{2}}}{\sqrt{N}}
\end{align*}
and
\begin{align*}
\IE_{\beta,N}\frac{S^{2k}}{N^{2k}} & =\frac{1}{N^{2k}}\sum_{\underbar{\ensuremath{\boldsymbol{r}}}\in\Pi}\frac{N!}{r_{1}!\cdots r_{2k}!\left(N-\sum_{\ell=1}^{2k}r_{\ell}\right)!}\frac{\left(2k\right)!}{1!^{r_{1}}\cdots\left(2k\right)!^{r_{2k}}}\IE_{\beta,N}X\left(\underbar{\ensuremath{\boldsymbol{r}}}\right)\\
 & \geq-C\frac{1}{N}+\frac{1}{N^{2k}}\frac{N!}{\left(N-2k\right)!}\left[m\left(\beta\right)^{2k}-\boldsymbol{C}_{\textup{low}}\frac{\left(\ln N\right)^{\frac{3}{2}}}{\sqrt{N}}\right]\\
 & \geq-C\frac{1}{N}+\left(1-\frac{2k-1}{N}\right)^{2k}\left[m\left(\beta\right)^{2k}-\boldsymbol{C}_{\textup{low}}\frac{\left(\ln N\right)^{\frac{3}{2}}}{\sqrt{N}}\right]\\
 & \geq-C_{1}\frac{1}{N}+m\left(\beta\right)^{2k}-C_{2}\frac{\left(\ln N\right)^{\frac{3}{2}}}{\sqrt{N}}-C_{3}\frac{1}{N^{2k}}\\
 & \geq m\left(\beta\right)^{2k}-C\frac{\left(\ln N\right)^{\frac{3}{2}}}{\sqrt{N}}.
\end{align*}
This concludes the proof of the statements concerning $\beta>1$.
\end{proof}
\begin{rem}
\label{rem:overlap}Proposition \ref{prop:appr_moments} provides
an asymptotic expression for $\IE_{\beta,N}S^{2}$ which is not costly
to calculate. However, there is a trade-off involved in using these
approximations: whereas, by Proposition \ref{prop:ES2_fin}, $\beta\mapsto\IE_{\beta,N}S^{2}$
is strictly increasing, thus allowing us to uniquely identify the
maximum likelihood estimator (Definition 7 in \cite{BaMeSiTo2025})
using the optimality equation (\ref{eq:opt}), the approximations
in Proposition \ref{prop:appr_moments} do not have this monotonicity
property. In fact, the approximation for $\IE_{\beta,N}S^{2}$ becomes
arbitrarily large as $\beta\nearrow1$, making it so that the large
$N$ expression for $\IE_{\beta,N}S^{2}$ in the regime $\beta<1$
overlaps with the expression for $\beta>1$. Therefore, we cannot
properly estimate the parameter $\beta$ if its value lies very close
to 1, or else if the realisation of the sample is atypical.
\end{rem}

\subsection{Proof of Proposition \ref{prop:error}}
\begin{proof}[Proof of Proposition \ref{prop:error}]
 Recall the definitions of the statistic $T$ in Definition \ref{def:T_stat}
and of $S$ in (\ref{eq:S_lambda}). Also consult Definitions \ref{def:Legendre}-\ref{def:ess}
and Lemma \ref{lem:cumulant_entropy} for facts regarding entropy
functions of distributions, which we will employ in this proof.

Note that $\hat{\beta}_{N}^{\infty}\in I_{k}$ is equivalent to $T\in J_{k}$
for $k\in\left\{ h,l\right\} $ by Definition \ref{def:estimator_large_N}.
We show 
\[
\sup_{\beta\in I_{h}}\left\{ \IP\left\{ T\in J_{l}\right\} \right\} \leq2\exp\left(-\eta_{2}n\right),\quad n\in\IN.
\]
The random variable $S^{2}$ is bounded and not almost surely constant,
so Lemma \ref{lem:cumulant_entropy} applies to its entropy function
$\Lambda_{S^{2}}^{*}$. For all $\beta\in I_{h}$, $\IE\,T=\IE_{\beta,N}S^{2}\in J_{h}$
is satisfied. Also, the distance between $\IE\,T$ and the set $J_{l}$,
$\textup{dist}\left(\IE\,T,J_{l}\right)\coloneq\inf\left\{ \left|\IE\,T-x\right|\,|\,x\in J_{l}\right\} $,
is strictly positive since $J_{h}$ and $J_{l}$ are closed sets and
disjoint. $\Lambda_{S^{2}}^{*}$ is the Legendre transform (see Definition
\ref{def:Legendre}) of
\[
\Lambda_{S^{2}}\left(x\right)\coloneq\ln\,\IE\exp\left(xS^{2}\right),\quad x\in\IR.
\]
Recall that $\Lambda_{S^{2}}^{*}$ is strictly increasing on the interval
$\left(\IE_{\beta,N}S^{2},\infty\right)$ by Lemma \ref{lem:cumulant_entropy}.
Lemma \ref{lem:cumulant_entropy} yields $\Lambda_{S^{2}}^{*}\left(\IE_{\beta,N}S^{2}\right)=0$,
and because of $\inf J_{l}>\IE_{\beta,N}S^{2}$,
\begin{equation}
\eta\coloneq\inf_{x\in J_{l}}\Lambda^{*}\left(x\right)>0\label{eq:eta_2}
\end{equation}
holds. By Definition \ref{def:intervals}, there is some $a\in\IR$
such that $J_{l}=\left[a,\infty\right)$. We write
\[
\IP\left\{ T\in J_{l}\right\} =\IP\left\{ T\in\left[a,\infty\right)\right\} .
\]
By Lemma \ref{lem:cumulant_entropy}, $\Lambda_{S^{2}}$ is convex,
and Jensen's inequality yields
\begin{equation}
\Lambda_{S^{2}}\left(t\right)=\ln\,\IE\exp\left(tS^{2}\right)\geq\IE\left(\ln\exp\left(tS^{2}\right)\right)=t\,\IE\,S^{2}.\label{eq:Lambda}
\end{equation}
We rearrange terms in (\ref{eq:Lambda}) to obtain for all $t<0$
and $x\geq\IE_{\beta,N}S^{2}$,
\[
xt-\Lambda_{S^{2}}\left(t\right)\leq t\,\IE\,S^{2}-\Lambda_{S^{2}}\left(t\right)\leq0.
\]
Since $\Lambda_{S^{2}}0)=0$, $\Lambda_{S^{2}}^{*}(x)\geq0$ for all
$x\in\IR$ is a consequence of Definition \ref{def:Legendre}. This
and the last display yield
\begin{equation}
\Lambda_{S^{2}}^{*}(x)=\sup_{t\geq0}\left\{ xt-\Lambda_{S^{2}}(t)\right\} \label{eq:Lam_star_sup}
\end{equation}
for all $x\geq\IE_{\beta,N}S^{2}$. 

An application of Markov's inequality yields for all $t\geq0$,
\begin{align*}
\IP\left\{ T\in\left[a,\infty\right)\right\}  & =\IP\left\{ T-a\geq0\right\} \leq\IP\left\{ \exp\left(nt\left(T-a\right)\right)\geq1\right\} \leq\IE\exp\left(nt\left(T-a\right)\right)\\
 & =\exp\left(-nta\right)\prod_{s=1}^{n}\IE\exp\left(t\left(\sum_{i=1}^{N}X_{i}^{(s)}\right)^{2}\right)=\exp\left(-nta\right)\left[\IE\exp\left(tS^{2}\right)\right]^{n}\\
 & =\exp\left(-nta\right)\exp\left(n\Lambda_{S^{2}}\left(t\right)\right)=\exp\left(-n\left(ta-\Lambda_{S^{2}}\left(t\right)\right)\right).
\end{align*}

As this holds for all $t\geq0$, we use $\IE_{\beta,N}S^{2}<a$ and
(\ref{eq:Lam_star_sup}) to arrive at
\begin{equation}
\IP\left\{ T\in\left[a,\infty\right)\right\} \leq\exp\left(-n\Lambda_{S^{2}}^{*}\left(a\right)\right)=\exp\left(-\eta n\right).\label{eq:left_UB}
\end{equation}
Since this upper bound is independent of the value $\beta\in I_{h}$,
the claim follows. The proof of\\
$\sup_{\beta\in I_{l}}\left\{ \IP\left\{ T\in J_{h}\right\} \right\} \xrightarrow[n\rightarrow\infty]{}0$
is analogous. The constant $\eta_{2}$ from the statement of Proposition
\ref{prop:error} can be chosen as the minimum over all groups $\lambda$
of expressions as defined in (\ref{eq:eta_2}). Analogously, $\eta_{3}$
can be obtained.

Now we show the statement
\[
\IP\left\{ \hat{\beta}_{N}^{\infty}\in\left(-\infty,0\right)\cup\left\{ \pm\infty\right\} \right\} \leq2\exp\left(-\eta_{1}n\right),\quad n\in\IN.
\]
First of all we note that by Definition \ref{def:estimator_large_N},
$\hat{\beta}_{N}^{\infty}<0$ holds if and only if $T<N$, and, taking
also into account Definition \ref{def:m_beta} and Lemma \ref{lem:m_beta_increasing},
$\hat{\beta}_{N}^{\infty}=\infty$ if and only if $T=N^{2}$. We are
assuming $\beta>0$, so by Proposition \ref{prop:ES2_fin}, $N<\IE_{\beta,N}S^{2}<N^{2}$
holds. $\Lambda_{S^{2}}^{*}$ is strictly decreasing on $\left(\min\textup{Range}\left(S^{2}\right),\IE_{\beta,N}S^{2}\right)$
and strictly increasing on $\left(\IE_{\beta,N}S^{2},N^{2}\right)$
by Lemma \ref{lem:cumulant_entropy}, so we conclude that $\Lambda_{S^{2}}^{*}\left(N\right)>0$
and $\Lambda_{S^{2}}^{*}\left(N^{2}\right)>0$, and hence
\begin{equation}
\theta\coloneq\min\left\{ \Lambda_{S^{2}}^{*}\left(N\right),\Lambda_{S^{2}}^{*}\left(N^{2}\right)\right\} >0\label{eq:eta_1}
\end{equation}
holds. We express the probability $\IP\left\{ \hat{\beta}_{N}^{\infty}\in\left(-\infty,0\right)\cup\left\{ \pm\infty\right\} \right\} $
as
\[
\IP\left\{ T\notin\left[N,N^{2}\right)\right\} =\IP\left\{ T\in\left(-\infty,N\right)\right\} +\IP\left\{ T\in\left[N^{2},\infty\right)\right\} .
\]

For all $x\leq0$,
\begin{align*}
\IP\left\{ T\in\left(-\infty,N\right)\right\}  & =\IP\left\{ T-N<0\right\} \leq\IP\left\{ \exp\left(nx\left(T-N\right)\right)\geq1\right\} \leq\IE\exp\left(nx\left(T-N\right)\right)\\
 & =\exp\left(-nxN\right)\prod_{s=1}^{n}\IE\exp\left(x\left(\sum_{i=1}^{N}X_{i}^{(s)}\right)^{2}\right)=\exp\left(-nxN\right)\left[\IE\exp\left(xS^{2}\right)\right]^{n}\\
 & =\exp\left(-nxN\right)\exp\left(n\Lambda_{S^{2}}\left(x\right)\right)=\exp\left(-n\left(xN-\Lambda_{S^{2}}\left(x\right)\right)\right)
\end{align*}
holds, where we used Markov's inequality in step 3 above. Now we use
$N<\IE_{\beta,N}S^{2}$ and (\ref{eq:Lam_star_sup}) to arrive at
\begin{equation}
\IP\left\{ T\in\left(-\infty,N\right)\right\} \leq\exp\left(-n\Lambda_{S^{2}}^{*}\left(N\right)\right).\label{eq:left_UB-1}
\end{equation}
Similarly, we calculate the upper bound
\begin{equation}
\IP\left\{ T\in\left[N^{2},\infty\right)\right\} \leq\exp\left(-n\Lambda_{S^{2}}^{*}\left(N^{2}\right)\right).\label{eq:right_UB}
\end{equation}
Hence,
\[
\IP\left\{ T\notin\left[N,N^{2}\right)\right\} \leq2\exp\left(-\theta n\right)
\]
holds. We can choose $\eta_{1}$ to be the minimum over all groups
$\lambda$ of such expressions as in (\ref{eq:eta_1}).
\end{proof}
To apply the estimator $\hat{\beta}_{N}^{\infty}$, we propose two
algorithms. The choice of algorithm depends on whether the group size
$N$ is fixed or not.

\begin{algo}Let $N\in\IN$.
\begin{enumerate}
\item Fix an arbitrary upper bound $\varepsilon>0$ for the probability
of committing an error of obtaining $\hat{\beta}_{N}^{\infty}>1$
when $\beta<1$ or $\hat{\beta}_{N}^{\infty}<1$ when $\beta>1$.
\item Choose constants $b_{1}<1<b_{2}$ such that (\ref{eq:separation})
holds given $N$.
\item Choose the sample size $n$ large enough that\footnote{Note that the left hand side of this inequality converges to 0 as
$n\rightarrow\infty$ by Proposition \ref{prop:error}.}
\[
\max\left\{ \sup_{\beta\in I_{h}}\left\{ \IP\left\{ T\in J_{l}\right\} \right\} ,\sup_{\beta\in I_{l}}\left\{ \IP\left\{ T\in J_{h}\right\} \right\} \right\} <\varepsilon.
\]
\item Calculate the statistic $T$ and the estimator $\hat{\beta}_{N}^{\infty}$
from a sample $\left(x^{(1)},\ldots,x^{(n)}\right)\in\Omega_{N}^{n}$
of size $n$.
\end{enumerate}
\end{algo}

\vspace{.5\baselineskip}

\begin{algo}Let the constants $b_{1}<1<b_{2}$ and the intervals
$I_{h}$ and $I_{l}$ be as in Definition \ref{def:intervals}.
\begin{enumerate}
\item Fix an arbitrary upper bound $\varepsilon>0$ for the probability
of committing an error of obtaining $\hat{\beta}_{N}^{\infty}>1$
when $\beta<1$ or $\hat{\beta}_{N}^{\infty}<1$ when $\beta>1$.
\item Choose $N$ large enough that (\ref{eq:separation}) holds given constants
$b_{1}<1<b_{2}$.
\item Choose the sample size $n$ large enough that
\[
\max\left\{ \sup_{\beta\in I_{h}}\left\{ \IP\left\{ T\in J_{l}\right\} \right\} ,\sup_{\beta\in I_{l}}\left\{ \IP\left\{ T\in J_{h}\right\} \right\} \right\} <\varepsilon.
\]
\item Calculate the statistic $T$ and the estimator $\hat{\beta}_{N}^{\infty}$
from a sample $\left(x^{(1)},\ldots,x^{(n)}\right)\in\Omega_{N}^{n}$
of size $n$.
\end{enumerate}
\end{algo}
\begin{rem}
While the maximum likelihood estimator (see Definition 7 in \cite{BaMeSiTo2025})
using the exact value $\IE_{\beta,N}S^{2}$ provides an estimate for
$\beta$ given any sample of observations, including values in any
interval containing the value 1, the same does not hold for the estimator
$\hat{\beta}_{N}^{\infty}$ using asymptotic approximations for $\IE_{\beta,N}S^{2}$
from Definition \ref{def:estimator_large_N}. Due to the overlap of
these approximations, there is no way to avoid having an interval
around 1 in which no precise estimation of $\beta$ is possible. However,
as discussed above, this interval can be made arbitrarily small if
the number of voters is large enough. Figure \ref{fig:approx_ES2}
illustrates the asymptotic approximation of $\IE_{\beta,N}S^{2}$
given in Proposition \ref{prop:appr_moments} for $\beta\in I_{h}\cup I_{l}$.

\begin{figure}
\begin{centering}
\includegraphics{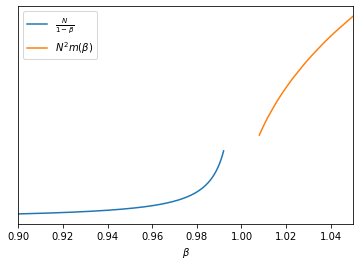}
\par\end{centering}
\caption{\label{fig:approx_ES2}Approximation of $\protect\IE_{\beta,N}S^{2}$
on $I_{h}\cup I_{l}$}
\end{figure}
\end{rem}

\subsection{Large Deviations Principle for $\left(\frac{S}{N}\right)^{2}$}

As a preliminary to the proof of Theorem \ref{thm:properties_bML_inf},
we show that the sequences $\left(\frac{S}{N}\right)_{N\in\IN}$ and
$\left(\left(\frac{S}{N}\right)^{2}\right)_{N\in\IN}$ satisfy large
deviation principles.
\begin{defn}
\label{def:Rademacher}The \emph{Rademacher distribution} $\textup{Rad}_{t}$
with parameter $t\in\left[-1,1\right]$ is a probability measure $P_{t}$
on $\left\{ -1,1\right\} $ given by $P_{t}\left\{ 1\right\} \coloneqq\frac{1+t}{2}$.
\end{defn}

We will need its entropy function $\Lambda_{P_{0}}^{*}$.
\begin{lem}
\label{lem:entropy_Rad}The entropy function $\Lambda_{P_{0}}^{*}:\IR\rightarrow\left[0,\infty\right]$
is given by
\[
\Lambda_{P_{0}}^{*}\left(x\right)=\begin{cases}
\frac{1-x}{2}\ln\left(1-x\right)+\frac{1+x}{2}\ln\left(1+x\right), & \textup{if }x\in\left[-1,1\right],\\
\infty, & \textup{if }x\notin\left[-1,1\right].
\end{cases}
\]
\end{lem}

\begin{proof}
Since $X_{1}\sim\textup{Rad}_{0}$ takes values in $\left\{ -1,1\right\} $,
by statement 2 of Lemma \ref{lem:cumulant_entropy}, we have $\Lambda_{P_{0}}^{*}\left(x\right)=\infty$
for all $x\notin\left[-1,1\right]$. Now let $x\in\left(-1,1\right)$.
Let for all $t\in\IR$
\[
f\left(t\right)\coloneq xt-\Lambda_{P_{0}}\left(t\right)
\]
and
\[
g\left(t\right)\coloneq\exp f\left(t\right)=\frac{\exp\left(xt\right)}{\IE\exp\left(tX_{1}\right)}.
\]
We calculate
\begin{align}
\Lambda_{P_{0}}\left(t\right)=\ln\IE\exp\left(tX_{1}\right) & =\ln\left(\frac{1}{2}\exp\left(-t\right)+\frac{1}{2}\exp\left(t\right)\right)=\ln\cosh t\label{eq:mom_gen_Rad}
\end{align}
for all $t\in\IR$. Hence,
\[
g\left(t\right)=2\frac{\exp\left(xt\right)}{\exp\left(-t\right)+\exp\left(t\right)}=\frac{2}{\exp\left(-t\left(1+x\right)\right)+\exp\left(t\left(1-x\right)\right)}.
\]
Since $\lim_{t\rightarrow\infty}\exp\left(-t\left(1+x\right)\right)=0$
and $\lim_{t\rightarrow\infty}\exp\left(t\left(1-x\right)\right)=\infty$,
$\lim_{t\rightarrow\infty}g\left(t\right)=0$ and $\lim_{t\rightarrow\infty}f\left(t\right)=-\infty$
hold.

Next we note that $\lim_{t\rightarrow-\infty}\exp\left(-t\left(1+x\right)\right)=\infty$
and $\lim_{t\rightarrow-\infty}\exp\left(t\left(1-x\right)\right)=0$,
so that $\lim_{t\rightarrow-\infty}g\left(t\right)=0$ and $\lim_{t\rightarrow-\infty}f\left(t\right)=-\infty$.
Hence, the continuous function $f$ reaches its maximum at some point
$t_{0}\in\IR$. We calculate $t_{0}$ by equating the derivative $f'$
to 0:
\[
f'\left(t_{0}\right)=x-\frac{\IE\,X_{1}\exp\left(t_{0}X_{1}\right)}{\IE\exp\left(t_{0}X_{1}\right)}=0.
\]
We have
\[
\IE\,X_{1}\exp\left(tX_{1}\right)=-\frac{1}{2}\exp\left(-t\right)+\frac{1}{2}\exp\left(t\right)=\sinh t
\]
for all $t\in\IR$. So the maximum $t_{0}$ satisfies
\[
x=\frac{\sinh t_{0}}{\cosh t_{0}}=\tanh t_{0},
\]
or equivalently
\[
t_{0}=\textup{artanh }x.
\]
Next we use the formula
\[
\textup{artanh }x=\frac{1}{2}\ln\frac{1+x}{1-x}
\]
for $x\in\left(-1,1\right)$:
\[
t_{0}=\frac{1}{2}\ln\frac{1+x}{1-x}.
\]

By Definition \ref{def:cumul_entropy},
\begin{equation}
\Lambda_{P_{0}}^{*}\left(x\right)=\sup_{t\in\IR}\left\{ xt-\Lambda_{P_{0}}\left(t\right)\right\} =\sup_{t\in\IR}\left\{ xt-\ln\cosh t\right\} =\frac{x}{2}\ln\frac{1+x}{1-x}-\ln\cosh\left(\frac{1}{2}\ln\frac{1+x}{1-x}\right).\label{eq:entropy_fn_P_0}
\end{equation}
We calculate
\begin{align*}
\cosh\left(\frac{1}{2}\ln\frac{1+x}{1-x}\right) & =\frac{1}{2}\left[\exp\left(\frac{1}{2}\ln\frac{1+x}{1-x}\right)+\exp\left(-\frac{1}{2}\ln\frac{1+x}{1-x}\right)\right]\\
 & =\frac{1}{2}\left[\sqrt{\frac{1+x}{1-x}}+\sqrt{\frac{1-x}{1+x}}\right]\\
 & =\frac{1}{2}\frac{2}{\sqrt{1-x^{2}}}=\frac{1}{\sqrt{1-x^{2}}}.
\end{align*}
Therefore,
\[
\ln\cosh\left(\frac{1}{2}\ln\frac{1+x}{1-x}\right)=-\frac{1}{2}\ln\left(1-x^{2}\right)=-\frac{1}{2}\left[\ln\left(1-x\right)+\ln\left(1+x\right)\right].
\]
Substituting the expression above into (\ref{eq:entropy_fn_P_0})
yields
\[
\Lambda_{P_{0}}^{*}\left(x\right)=\frac{1-x}{2}\ln\left(1-x\right)+\frac{1+x}{2}\ln\left(1+x\right).
\]
Finally, let $x=1$. In this case, we have for all $t\in\IR$, 
\[
g\left(t\right)=2\frac{\exp\left(t\right)}{\exp\left(-t\right)+\exp\left(t\right)}=\frac{2}{\exp\left(-2t\right)+1}.
\]
We see that for all $t\in\IR$, $g\left(t\right)<2$, $\lim_{t\rightarrow\infty}g\left(t\right)=2$,
and $\lim_{t\rightarrow\infty}f\left(t\right)=\ln2$, while $\lim_{t\rightarrow-\infty}g\left(t\right)=0$
and $\lim_{t\rightarrow-\infty}f\left(t\right)=-\infty$. Hence, $\Lambda_{P_{0}}^{*}\left(1\right)=\ln2$.
Analogously, one can show $\Lambda_{P_{0}}^{*}\left(-1\right)=\ln2$.
\end{proof}
We will also need a basic fact about logarithmic rates of convergence
stated in the next lemma.
\begin{lem}
\label{lem:log_conv_rates}Let $\left(a_{n}\right)_{n\in\IN}$ and
$\left(b_{n}\right)_{n\in\IN}$ be two sequences of non-negative real
numbers. Then
\[
\limsup_{n\rightarrow\infty}\frac{1}{n}\ln\left(a_{n}+b_{n}\right)=\max\left\{ \limsup_{n\rightarrow\infty}\frac{1}{n}\ln a_{n},\,\limsup_{n\rightarrow\infty}\frac{1}{n}\ln b_{n}\right\} .
\]
\end{lem}

\begin{proof}
The inequality
\[
\limsup_{n\rightarrow\infty}\frac{1}{n}\ln\left(a_{n}+b_{n}\right)\geq\max\left\{ \limsup_{n\rightarrow\infty}\frac{1}{n}\ln a_{n},\,\limsup_{n\rightarrow\infty}\frac{1}{n}\ln b_{n}\right\} 
\]
follows from $a_{n},b_{n}\geq0$ for each $n\in\IN$ and the monotonicity
of the natural logarithm. Next we calculate the upper bound
\begin{align*}
\limsup_{n\rightarrow\infty}\frac{1}{n}\ln\left(a_{n}+b_{n}\right) & \leq\limsup_{n\rightarrow\infty}\frac{1}{n}\ln\left(2\max\left\{ a_{n},b_{n}\right\} \right)\\
 & =\lim_{n\rightarrow\infty}\frac{1}{n}\ln2+\limsup_{n\rightarrow\infty}\frac{1}{n}\ln\max\left\{ a_{n},b_{n}\right\} \\
 & =\limsup_{n\rightarrow\infty}\frac{1}{n}\ln\max\left\{ a_{n},b_{n}\right\} \\
 & \leq\max\left\{ \limsup_{n\rightarrow\infty}\frac{1}{n}\ln a_{n},\limsup_{n\rightarrow\infty}\frac{1}{n}\ln b_{n}\right\} .
\end{align*}
The last inequality is a consequence of the following consideration.
If
\begin{align*}
\limsup_{n\rightarrow\infty}\frac{1}{n}\ln\max\left\{ a_{n},b_{n}\right\}  & >\max\left\{ \limsup_{n\rightarrow\infty}\frac{1}{n}\ln a_{n},\limsup_{n\rightarrow\infty}\frac{1}{n}\ln b_{n}\right\} 
\end{align*}
holds, then there is a subsequence with
\[
\limsup_{n\rightarrow\infty}\frac{1}{n}\ln\max\left\{ a_{n},b_{n}\right\} =\lim_{k\rightarrow\infty}\frac{1}{n_{k}}\ln\max\left\{ a_{n_{k}},b_{n_{k}}\right\} =\lim_{k\rightarrow\infty}\frac{1}{n_{k}}\ln a_{n_{k}}
\]
or
\[
\limsup_{n\rightarrow\infty}\frac{1}{n}\ln\max\left\{ a_{n},b_{n}\right\} =\lim_{k\rightarrow\infty}\frac{1}{n_{k}}\ln b_{n_{k}},
\]
because at least one of the sequences has to be the maximum an infinite
number of times in the subsequence that converges to $\limsup_{n\rightarrow\infty}\frac{1}{n}\ln\max\left\{ a_{n},b_{n}\right\} $.
Assume without loss of generality the latter. Then
\begin{align*}
\lim_{k\rightarrow\infty}\frac{1}{n_{k}}\ln b_{n_{k}} & >\max\left\{ \limsup_{n\rightarrow\infty}\frac{1}{n}\ln a_{n},\limsup_{n\rightarrow\infty}\frac{1}{n}\ln b_{n}\right\} \\
 & \geq\limsup_{n\rightarrow\infty}\frac{1}{n}\ln b_{n}
\end{align*}
is a contradiction.

\end{proof}
The next proposition gives large deviation principles for the large
$N$ behaviour of $\frac{S}{N}$ and $\left(\frac{S}{N}\right)^{2}$.
\begin{prop}
\label{prop:LDP_S_N}Let $\beta\geq0$. Then the sequence of random
variables $\left(\frac{S}{N}\right)_{N\in\IN}$ satisfies a large
deviations principle with rate $N$ and rate function $I_{\beta}:\IR\rightarrow\left[0,\infty\right]$,
\[
I_{\beta}\left(x\right)\coloneq\begin{cases}
-\frac{1}{2}\beta x^{2}+\frac{1-x}{2}\ln\left(1-x\right)+\frac{1+x}{2}\ln\left(1+x\right)-\psi & \textup{if }x\in\left[-1,1\right],\\
\infty & \textup{if }x\notin\left[-1,1\right],
\end{cases}
\]
where $\psi\coloneq\inf_{x\in\left[-1,1\right]}\left\{ -\frac{1}{2}\beta x^{2}+\frac{1-x}{2}\ln\left(1-x\right)+\frac{1+x}{2}\ln\left(1+x\right)\right\} $.

Let $J_{\beta}:\IR\rightarrow\left[0,\infty\right]$ be defined by
\[
J_{\beta}\left(y\right)\coloneq\begin{cases}
I_{\beta}\left(\sqrt{y}\right), & \textup{if }y\in\left[0,1\right],\\
\infty, & \textup{if }y\notin\left[0,1\right].
\end{cases}
\]
Then $\left(\left(\frac{S}{N}\right)^{2}\right)_{N\in\IN}$ satisfies
a large deviations principle with rate $N$ and rate function $J_{\beta}:\IR\rightarrow\left[0,\infty\right]$.

$I_{\beta}$ has one minimum at $m\left(\beta\right)=0$ if $\beta\leq1$
and two minima $\pm m\left(\beta\right)$ with $-1<-m\left(\beta\right)<0<m\left(\beta\right)<1$
if $\beta>1$. $J_{\beta}$ has one minimum at $m\left(\beta\right)^{2}$
for all $\beta\geq0$. There is for each closed $K\subset\IR$ that
does not contain $m\left(\beta\right)^{2}$ a constant $C_{K}$ such
that
\[
\IP\left\{ \left(\frac{S}{N}\right)^{2}\in K\right\} \leq C_{K}\exp\left(-\frac{N}{2}\inf_{y\in K}J\left(y\right)\right),\quad N\in\IN,
\]
where $\inf_{y\in K}J\left(y\right)$ is strictly positive.
\end{prop}

\begin{proof}
Let $P^{\otimes N}$ be the product measure with one-dimensional marginals
$P_{0}$ (recall Definition \ref{def:Rademacher} of the Rademacher
distribution). We write the CWM probabilities corresponding to a single
group (cf. Definition \ref{def:CWM}) as
\begin{align*}
\IP_{\beta,N}\left(X_{1}=x_{1},\ldots,X_{N}=x_{N}\right) & =Z_{\beta,N}^{-1}\,\exp\left(\frac{\beta}{2N}\left(\sum_{i=1}^{N}x_{i}\right)^{2}\right)\\
 & =\frac{\exp\left(N\frac{\beta}{2}\left(\frac{\sum_{i=1}^{N}x_{i}}{N}\right)^{2}\right)\frac{1}{2^{N}}}{\sum_{y\in\Omega_{N}}\exp\left(N\frac{\beta}{2}\left(\frac{\sum_{i=1}^{N}y_{i}}{N}\right)^{2}\right)\frac{1}{2^{N}}}\\
 & =\frac{\exp\left(N\frac{\beta}{2}\left(\frac{s\left(x\right)}{N}\right)^{2}\right)\frac{1}{2^{N}}}{\sum_{y\in\Omega_{N}}\exp\left(N\frac{\beta}{2}\left(\frac{s\left(y\right)}{N}\right)^{2}\right)\frac{1}{2^{N}}}\\
 & =\frac{\exp\left(N\frac{\beta}{2}\left(\frac{s\left(x\right)}{N}\right)^{2}\right)P^{\otimes N}\left(x\right)}{Z'_{\beta,N}}
\end{align*}
for all $x\in\Omega_{N}$, where we set $Z'_{\beta,N}\coloneq\sum_{y\in\Omega_{N}}\exp\left(N\frac{\beta}{2}\left(\frac{s\left(y\right)}{N}\right)^{2}\right)P^{\otimes N}\left(y\right)$.
The push-forward measure $\IP_{\beta,N}\circ\left(\frac{S}{N}\right)^{-1}$
is given by
\[
\IP_{\beta,N}\circ\left(\frac{S}{N}\right)^{-1}\left(w\right)=\frac{\exp\left(N\frac{\beta}{2}w^{2}\right)\left(P^{\otimes N}\circ\left(\frac{S}{N}\right)^{-1}\right)\left(w\right)}{Z'_{\beta,N}},\quad w\in\left\{ -1,\frac{-N+2}{N},\ldots,1\right\} .
\]
Since under the distribution $P^{\otimes N}$ the voting vector $X$
is composed of i.i.d.\! random variables $X_{1},\ldots,X_{N}$ with
$P_{0}$ distribution, the sequence $\left(P^{\otimes N}\circ\left(\frac{S}{N}\right)^{-1}\right)_{N\in\IN}$
satisfies a large deviations principle with rate $N$ and rate function
$I\coloneq\Lambda_{P_{0}}^{*}$ given in Lemma \ref{lem:entropy_Rad}.

We employ Theorem \ref{thm:Varadhan} to show that a large deviations
principle for $\left(\IP_{\beta,N}\circ\left(\frac{S}{N}\right)^{-1}\right)_{N\in\IN}$
with rate $N$ and rate function $I_{\beta}$ follows. We set $f:\IR\rightarrow\IR$,
\[
f\left(x\right)\coloneq\frac{\beta}{2}x^{2},\quad x\in\IR,
\]
$P_{N}\coloneq P^{\otimes N}\circ\left(\frac{S}{N}\right)^{-1}$ for
all $N\in\IN$, and verify the condition
\begin{equation}
\lim_{M\rightarrow\infty}\limsup_{N\rightarrow\infty}\frac{1}{N}\ln\int_{\left\{ f\left(x\right)\geq M\right\} }\exp\left(Nf\left(x\right)\right)\,P_{N}\left(\textup{d}x\right)=-\infty.\label{eq:Varadhan_cond}
\end{equation}
Since $\frac{S}{N}$ is a bounded random variable with $\left|\frac{S}{N}\right|\leq1$
for all $N\in\IN$, we have for all $M>\frac{\beta}{2}$
\[
\int_{\left\{ f\left(x\right)\geq M\right\} }\exp\left(Nf\left(x\right)\right)\,P_{N}\left(\textup{d}x\right)=0,
\]
and therefore
\[
\frac{1}{N}\ln\int_{\left\{ f\left(x\right)\geq M\right\} }\exp\left(Nf\left(x\right)\right)\,P_{N}\left(\textup{d}x\right)=-\infty
\]
holds for all $N\in\IN$, and the limit (\ref{eq:Varadhan_cond})
follows. Thus Theorem \ref{thm:Varadhan} supplies
\begin{equation}
\lim_{N\rightarrow\infty}\frac{1}{N}\ln\int_{\IR}\exp\left(Nf\left(x\right)\right)\,P_{N}\left(\textup{d}x\right)=\sup_{x\in\IR}\left\{ f\left(x\right)-I\left(x\right)\right\} .\label{eq:Varadhan}
\end{equation}
We show that $I_{\beta}$ is a rate function. First of all we note
that, as $I|\left[-1,1\right]^{c}=\infty$ and the functions $I|\left[-1,1\right]$
and $f$ are continuous, the supremum on the right hand side of (\ref{eq:Varadhan})
must be achieved on the compact set $\left[-1,1\right]$ and therefore
be finite. So
\[
-\psi=-\inf_{x\in\IR}\left\{ I\left(x\right)-f\left(x\right)\right\} =\sup_{x\in\IR}\left\{ f\left(x\right)-I\left(x\right)\right\} \in\IR.
\]
Next we note that $I-f-\psi$ is lower semi-continuous because $f$
is continuous and $I$ is lower semi-continuous. By definition of
$\psi$, $I_{\beta}=I-f-\psi$ takes values in $\left[0,\infty\right]$.
These facts together mean that $I_{\beta}$ is a rate function. Before
we show that the level sets of $I_{\beta}$ are compact, we demonstrate
the large deviations lower bound for open sets.

Let $G\subset\IR$ be open, $x_{0}\in G$, and $\varepsilon>0$. We
set $\tilde{G}\coloneq\left\{ x\in G\,|\,f\left(x\right)>f\left(x_{0}\right)-\varepsilon\right\} $,
which is a non-empty open subset of $G$ because $f$ is continuous.
Then
\begin{align*}
\liminf_{N\rightarrow\infty}\frac{1}{N}\ln\int_{G}\exp\left(Nf\left(x\right)\right)\,P_{N}\left(\textup{d}x\right) & \geq\liminf_{N\rightarrow\infty}\frac{1}{N}\ln\int_{\tilde{G}}\exp\left(Nf\left(x\right)\right)\,P_{N}\left(\textup{d}x\right)\\
 & \geq\liminf_{N\rightarrow\infty}\frac{1}{N}\ln\exp\left(N\left(f\left(x_{0}\right)-\varepsilon\right)\right)\int_{\tilde{G}}P_{N}\left(\textup{d}x\right)\\
 & =\liminf_{N\rightarrow\infty}\left[\left(f\left(x_{0}\right)-\varepsilon\right)+\frac{1}{N}\ln\int_{\tilde{G}}P_{N}\left(\textup{d}x\right)\right]\\
 & \geq f\left(x_{0}\right)-\varepsilon-\inf_{x\in\tilde{G}}I\left(x\right)\\
 & \geq f\left(x_{0}\right)-I\left(x_{0}\right)-\varepsilon.
\end{align*}
As $x_{0}\in G$, $\varepsilon>0$ were arbitrarily selected, we have
proved
\begin{equation}
\liminf_{N\rightarrow\infty}\frac{1}{N}\ln\int_{G}\exp\left(Nf\left(x\right)\right)\,P_{N}\left(\textup{d}x\right)\geq\sup_{x\in G}\left\{ f\left(x\right)-I\left(x\right)\right\} =-\inf_{x\in G}\left\{ I\left(x\right)-f\left(x\right)\right\} .\label{eq:LB_numerator}
\end{equation}
Using (\ref{eq:LB_numerator}) and taking into account (\ref{eq:Varadhan}),
we calculate the large deviations lower bound
\begin{align*}
\liminf_{N\rightarrow\infty}\frac{1}{N}\,\IP_{\beta,N}\circ\left(\frac{S}{N}\right)^{-1}\left(G\right) & =\liminf_{N\rightarrow\infty}\frac{1}{N}\ln\frac{\int_{G}\exp\left(Nf\left(x\right)\right)\,P_{N}\left(\textup{d}x\right)}{\int_{\IR}\exp\left(Nf\left(x\right)\right)\,P_{N}\left(\textup{d}x\right)}\\
 & =\liminf_{N\rightarrow\infty}\frac{1}{N}\ln\int_{G}\exp\left(Nf\left(x\right)\right)\,P_{N}\left(\textup{d}x\right)-\lim_{N\rightarrow\infty}\frac{1}{N}\ln\int_{\IR}\exp\left(Nf\left(x\right)\right)\,P_{N}\left(\textup{d}x\right)\\
 & \geq-\inf_{x\in G}\left\{ I\left(x\right)-f\left(x\right)\right\} -\sup_{x\in\IR}\left\{ f\left(x\right)-I\left(x\right)\right\} \\
 & =-\left[\inf_{x\in G}\left\{ I\left(x\right)-f\left(x\right)\right\} -\inf_{x\in\IR}\left\{ I\left(x\right)-f\left(x\right)\right\} \right]\\
 & =-\inf_{x\in G}I_{\beta}\left(x\right).
\end{align*}
This proves the lower bound $\liminf_{N\rightarrow\infty}\frac{1}{N}\,\IP_{\beta,N}\circ\left(\frac{S}{N}\right)^{-1}\left(G\right)\geq-\inf_{x\in G}I_{\beta}\left(x\right)$
and for any open set $G$.

Now let $\alpha\in\left[0,\infty\right)$. We show that the level
set
\[
\Gamma_{\alpha}\coloneq\left\{ x\in\IR\,|\,I_{\beta}\left(x\right)\leq\alpha\right\} 
\]
is sequentially compact. Let $\left(x_{n}\right)_{n\in\IN}$ be a
sequence in $\Gamma_{\alpha}$. We distinguish the cases $s\coloneq\sup_{n\in\IN}f\left(x_{n}\right)<\infty$
and $s=\infty$. First assume $s=\infty$. By the lower bound (\ref{eq:LB_numerator})
and (\ref{eq:Varadhan_cond}), we have
\begin{align*}
\lim_{M\rightarrow\infty}\inf\left\{ I\left(x\right)-f\left(x\right)\,|\,x\in\IR,f\left(x\right)>M\right\}  & =-\lim_{M\rightarrow\infty}\sup\left\{ f\left(x\right)-I\left(x\right)\,|\,x\in\IR,f\left(x\right)>M\right\} \\
 & \geq-\lim_{M\rightarrow\infty}\liminf_{N\rightarrow\infty}\frac{1}{N}\ln\int_{\left\{ f\left(x\right)>M\right\} }\exp\left(Nf\left(x\right)\right)\,P_{N}\left(\textup{d}x\right)\\
 & \geq-\lim_{M\rightarrow\infty}\limsup_{N\rightarrow\infty}\frac{1}{N}\ln\int_{\left\{ f\left(x\right)>M\right\} }\exp\left(Nf\left(x\right)\right)\,P_{N}\left(\textup{d}x\right)\\
 & =\infty.
\end{align*}
$s=\infty$ and $\lim_{M\rightarrow\infty}\inf\left\{ I\left(x\right)-f\left(x\right)\,|\,x\in\IR,f\left(x\right)>M\right\} =\infty$
together imply
\begin{align*}
\sup_{n\in\IN}I_{\beta}\left(x_{n}\right) & =\sup_{n\in\IN}\left\{ I\left(x_{n}\right)-f\left(x_{n}\right)\right\} -\psi=\infty,
\end{align*}
contradicting the upper bound $I_{\beta}\left(x_{n}\right)\leq\alpha<\infty$
which holds because of $x_{n}\in\Gamma_{\alpha}$ for all $n\in\IN$.

Now assume $s<\infty$. Then
\begin{align*}
\sup_{n\in\IN}I\left(x_{n}\right) & =\sup_{n\in\IN}\left\{ I_{\beta}\left(x_{n}\right)+f\left(x_{n}\right)\right\} +\psi\\
 & \leq\sup_{n\in\IN}I_{\beta}\left(x_{n}\right)+\sup_{n\in\IN}f\left(x_{n}\right)+\psi\\
 & \leq\alpha+s+\psi<\infty
\end{align*}
holds. Since $I$ is a good rate function, its level set $\left\{ x\in\IR\,|\,I\left(x\right)\leq\alpha+s+\psi\right\} $
is compact, so there is a subsequence $\left(x_{n_{k}}\right)_{k\in\IN}\subset\left(x_{n}\right)_{n\in\IN}$
that converges to some $x_{0}\in\IR$. $I_{\beta}$ is lower semi-continuous,
so
\[
\alpha\geq\liminf_{k\rightarrow\infty}I_{\beta}\left(x_{n_{k}}\right)\geq I_{\beta}\left(x_{0}\right),
\]
and $x_{0}\in\Gamma_{\alpha}$ holds. Above we first used that for
all $k\in\IN$, $x_{n_{k}}\in\Gamma_{\alpha}$, and then the convergence
$x_{n_{k}}\xrightarrow[k\rightarrow\infty]{}x_{0}$ and the lower
semi-continuity of $I_{\beta}$. This concludes the proof that $I_{\beta}$
is a good rate function.

It remains to show the large deviations upper bound for closed sets.
Let $L>0$ be a constant such that
\[
L\geq\sup_{x\in\IR}\left\{ f\left(x\right)-I\left(x\right)\right\} 
\]
holds. Define $f_{L}:\IR\rightarrow\IR$ by setting
\[
f_{L}\left(x\right)\coloneq\min\left\{ f\left(x\right),L\right\} ,\quad x\in\IR.
\]
Choose another constant $M$ such that
\[
M<\min\left\{ L,\sup_{x\in\IR}\left\{ f_{L}\left(x\right)-I\left(x\right)\right\} \right\} .
\]
Then $f_{L}$ is continuous and satisfies $f_{L}\leq L$ and $\sup_{x\in\IR}\left\{ f_{L}\left(x\right)-I\left(x\right)\right\} \leq L$
due to $I\ge0$. Our strategy is to demonstrate the upper bound for
$f_{L}$ first, and then use this as an auxiliary result for the upper
bound involving $f$.

Let $K\subset\IR$ be closed. Let $m\in\IN$, and we define the sets
\[
K_{j}\coloneq\left\{ x\in K\,\left|\,M+\frac{j-1}{m}\left(L-M\right)\leq f_{L}\left(x\right)\leq M+\frac{j}{m}\left(L-M\right)\right.\right\} ,\quad j\in\IN_{m}.
\]
Each $K_{j}$ is closed. Since $\left(P_{N}\right)_{N\in\IN}$ satisfies
a large deviations property with rate $N$ and rate function $I$,
we have the upper bounds
\begin{align*}
\limsup_{N\rightarrow\infty}\frac{1}{N}\ln\int_{K_{j}}\exp\left(Nf_{L}\left(x\right)\right)\,P_{N}\left(\textup{d}x\right) & \leq\limsup_{N\rightarrow\infty}\frac{1}{N}\ln\left[\exp\left(N\left(M+\frac{j}{m}\left(L-M\right)\right)\right)\int_{K_{j}}\,P_{N}\left(\textup{d}x\right)\right]\\
 & =\limsup_{N\rightarrow\infty}\left[M+\frac{j}{m}\left(L-M\right)+\frac{1}{N}\ln\int_{K_{j}}\,P_{N}\left(\textup{d}x\right)\right]\\
 & \leq M+\frac{j}{m}\left(L-M\right)-\inf_{x\in K_{j}}I\left(x\right)
\end{align*}
for each $j\in\IN_{m}$. Using Lemma \ref{lem:log_conv_rates}, we
obtain
\begin{align*}
\limsup_{N\rightarrow\infty}\frac{1}{N}\ln\int_{K}\exp\left(Nf_{L}\left(x\right)\right)\,P_{N}\left(\textup{d}x\right) & \leq\limsup_{N\rightarrow\infty}\frac{1}{N}\ln\sum_{j=1}^{m}\int_{K_{j}}\exp\left(Nf_{L}\left(x\right)\right)\,P_{N}\left(\textup{d}x\right)\\
 & =\max_{j\in\IN_{m}}\limsup_{N\rightarrow\infty}\frac{1}{N}\ln\int_{K_{j}}\exp\left(Nf_{L}\left(x\right)\right)\,P_{N}\left(\textup{d}x\right).
\end{align*}
The last two displays yield
\begin{align*}
\limsup_{N\rightarrow\infty}\frac{1}{N}\ln\int_{K}\exp\left(Nf_{L}\left(x\right)\right)\,P_{N}\left(\textup{d}x\right) & \leq\max_{j\in\IN_{m}}\left(M+\frac{j}{m}\left(L-M\right)-\inf_{x\in K_{j}}I\left(x\right)\right).
\end{align*}
We calculate for each $j\in\IN_{m}$,
\[
M+\frac{j}{m}\left(L-M\right)-\inf_{x\in K_{j}}I\left(x\right)=M+\frac{j-1}{m}\left(L-M\right)+\sup_{x\in K_{j}}\left\{ -I\left(x\right)\right\} +\frac{L-M}{m}.
\]
Since $f_{L}|K_{j}\geq M+\frac{j-1}{m}\left(L-M\right)$, we see that
\[
M+\frac{j-1}{m}\left(L-M\right)+\sup_{x\in K_{j}}\left\{ -I\left(x\right)\right\} \leq\sup_{x\in K_{j}}\left\{ f_{L}\left(x\right)-I\left(x\right)\right\} 
\]
holds. The last three displays lead to
\begin{align*}
\limsup_{N\rightarrow\infty}\frac{1}{N}\ln\int_{K}\exp\left(Nf_{L}\left(x\right)\right)\,P_{N}\left(\textup{d}x\right) & \leq\max_{j\in\IN_{m}}\sup_{x\in K_{j}}\left\{ f_{L}\left(x\right)-I\left(x\right)\right\} +\frac{L-M}{m}\\
 & =\sup_{x\in K}\left\{ f_{L}\left(x\right)-I\left(x\right)\right\} +\frac{L-M}{m}\\
 & =-\inf_{x\in K}\left\{ I\left(x\right)-f_{L}\left(x\right)\right\} +\frac{L-M}{m}.
\end{align*}
Letting $m\rightarrow\infty$ proves the upper bound for $f_{L}$
instead of $f$:
\begin{equation}
\limsup_{N\rightarrow\infty}\frac{1}{N}\ln\int_{K}\exp\left(Nf_{L}\left(x\right)\right)\,P_{N}\left(\textup{d}x\right)\leq-\inf_{x\in K}\left\{ I\left(x\right)-f_{L}\left(x\right)\right\} \label{eq:UB_f_L}
\end{equation}

Condition (\ref{eq:Varadhan_cond}) implies the existence of a constant
$L>0$ such that the inequality
\[
\limsup_{N\rightarrow\infty}\frac{1}{N}\ln\int_{K\cap\left\{ f\left(x\right)>L\right\} }\exp\left(Nf\left(x\right)\right)\,P_{N}\left(\textup{d}x\right)\leq\sup_{x\in K}\left\{ f\left(x\right)-I\left(x\right)\right\} 
\]
is satisfied. We have
\[
\int_{K}\exp\left(Nf\left(x\right)\right)\,P_{N}\left(\textup{d}x\right)=\int_{K\cap\left\{ f\left(x\right)\leq L\right\} }\exp\left(Nf\left(x\right)\right)\,P_{N}\left(\textup{d}x\right)+\int_{K\cap\left\{ f\left(x\right)>L\right\} }\exp\left(Nf\left(x\right)\right)\,P_{N}\left(\textup{d}x\right),
\]
and by definition of $f_{L}$,
\[
\int_{K\cap\left\{ f\left(x\right)\leq L\right\} }\exp\left(Nf\left(x\right)\right)\,P_{N}\left(\textup{d}x\right)\leq\int_{K}\exp\left(Nf_{L}\left(x\right)\right)\,P_{N}\left(\textup{d}x\right).
\]
Combining the last three displays, employing Lemma \ref{lem:log_conv_rates},
and the upper bound (\ref{eq:UB_f_L}) for $f_{L}$, we calculate
the upper bound
\begin{align*}
 & \quad\limsup_{N\rightarrow\infty}\frac{1}{N}\ln\int_{K}\exp\left(Nf\left(x\right)\right)\,P_{N}\left(\textup{d}x\right)\\
 & =\limsup_{N\rightarrow\infty}\frac{1}{N}\ln\left[\int_{K\cap\left\{ f\left(x\right)\leq L\right\} }\exp\left(Nf\left(x\right)\right)\,P_{N}\left(\textup{d}x\right)+\int_{K\cap\left\{ f\left(x\right)>L\right\} }\exp\left(Nf\left(x\right)\right)\,P_{N}\left(\textup{d}x\right)\right]\\
 & \leq\limsup_{N\rightarrow\infty}\frac{1}{N}\ln\left[\int_{K}\exp\left(Nf_{L}\left(x\right)\right)\,P_{N}\left(\textup{d}x\right)+\int_{\left\{ f\left(x\right)>L\right\} }\exp\left(Nf\left(x\right)\right)\,P_{N}\left(\textup{d}x\right)\right]\\
 & =\max\left\{ \limsup_{N\rightarrow\infty}\frac{1}{N}\ln\int_{K}\exp\left(Nf_{L}\left(x\right)\right)\,P_{N}\left(\textup{d}x\right),\limsup_{N\rightarrow\infty}\frac{1}{N}\ln\int_{\left\{ f\left(x\right)>L\right\} }\exp\left(Nf\left(x\right)\right)\,P_{N}\left(\textup{d}x\right)\right\} \\
 & \leq\max\left\{ -\inf_{x\in K}\left\{ I\left(x\right)-f_{L}\left(x\right)\right\} ,\sup_{x\in K}\left\{ f\left(x\right)-I\left(x\right)\right\} \right\} \\
 & =\max\left\{ \sup_{x\in K}\left\{ f_{L}\left(x\right)-I\left(x\right)\right\} ,\sup_{x\in K}\left\{ f\left(x\right)-I\left(x\right)\right\} \right\} \\
 & =\sup_{x\in K}\left\{ f\left(x\right)-I\left(x\right)\right\} \\
 & =-\inf_{x\in K}\left\{ I\left(x\right)-f\left(x\right)\right\} .
\end{align*}
Using (\ref{eq:Varadhan}), we calculate the upper bound
\begin{align*}
\limsup_{N\rightarrow\infty}\frac{1}{N}\,\IP_{\beta,N}\circ\left(\frac{S}{N}\right)^{-1}\left(K\right) & =\limsup_{N\rightarrow\infty}\frac{1}{N}\ln\frac{\int_{K}\exp\left(Nf\left(x\right)\right)\,P_{N}\left(\textup{d}x\right)}{\int_{\IR}\exp\left(Nf\left(x\right)\right)\,P_{N}\left(\textup{d}x\right)}\\
 & =\limsup_{N\rightarrow\infty}\frac{1}{N}\ln\int_{K}\exp\left(Nf\left(x\right)\right)\,P_{N}\left(\textup{d}x\right)-\lim_{N\rightarrow\infty}\frac{1}{N}\ln\int_{\IR}\exp\left(Nf\left(x\right)\right)\,P_{N}\left(\textup{d}x\right)\\
 & \leq-\inf_{x\in K}\left\{ I\left(x\right)-f\left(x\right)\right\} -\sup_{x\in\IR}\left\{ f\left(x\right)-I\left(x\right)\right\} \\
 & =-\left[\inf_{x\in K}\left\{ I\left(x\right)-f\left(x\right)\right\} -\inf_{x\in\IR}\left\{ I\left(x\right)-f\left(x\right)\right\} \right]\\
 & =-\inf_{x\in K}I_{\beta}\left(x\right).
\end{align*}
This proves the upper bound for closed sets and concludes the proof
that $\left(\IP_{\beta,N}\circ\left(\frac{S}{N}\right)^{-1}\right)_{N\in\IN}$
satisfies a large deviations property with rate $N$ and rate function
$I_{\beta}$.

To show the large deviations principle for the sequence $\left(\IP_{\beta,N}\circ\left(\left(\frac{S}{N}\right)^{2}\right)^{-1}\right)_{N\in\IN}$,
we use the large deviations principle for $\left(\IP_{\beta,N}\circ\left(\frac{S}{N}\right)^{-1}\right)_{N\in\IN}$
and Theorem \ref{thm:contr_princ}, which states that by defining
a good rate function $H:\IR\rightarrow\left[0,\infty\right]$ by
\[
H\left(y\right)\coloneq\inf\left\{ I_{\beta}\left(x\right)\,|\,x\in\IR,x^{2}=y\right\} ,\quad y\in\IR,
\]
we obtain a large deviations principle for $\left(\IP_{\beta,N}\circ\left(\left(\frac{S}{N}\right)^{2}\right)^{-1}\right)_{N\in\IN}$
with rate $N$ and rate function $H$. We now show that $H=J_{\beta}$.
Let $y<0$. Then the set $\left\{ x\in\IR\,|\,x^{2}=y\right\} $ is
empty, and therefore $H\left(y\right)=\infty=J_{\beta}\left(y\right)$.
For all $y>1$, we have $\sqrt{y}>1$ and $J_{\beta}\left(y\right)=\infty$.
On the other hand, $\left\{ x\in\IR\,|\,x^{2}=y\right\} =\left\{ -\sqrt{y},\sqrt{y}\right\} $
and $I_{\beta}\left(-\sqrt{y}\right)=I_{\beta}\left(\sqrt{y}\right)=\infty$.
Finally, let $y\in\left[0,1\right]$. Then, due to $I_{\beta}\left(x\right)=I_{\beta}\left(-x\right)$
for all $x\in\IR$, we have
\begin{align*}
H\left(y\right) & =\inf\left\{ I_{\beta}\left(x\right)\,|\,x\in\IR,x^{2}=y\right\} \\
 & =I_{\beta}\left(\sqrt{y}\right)=J_{\beta}\left(y\right).
\end{align*}
We turn to the calculation of the minima of $I_{\beta}$. As $I_{\beta}|\left[-1,1\right]^{c}=\infty$,
the minima must lie in $\left[-1,1\right]$. $I_{\beta}|\left(-1,1\right)$
is differentiable, and we calculate the derivative
\begin{align*}
I_{\beta}'\left(x\right) & =-\beta x+\frac{1}{2}\ln\frac{1+x}{1-x}.
\end{align*}
Equating the derivative to 0, we obtain
\[
\beta x=\frac{1}{2}\ln\frac{1+x}{1-x}=\textup{artanh }x,
\]
where we used the formula for the inverse hyperbolic function $\textup{artanh}$
which holds for all $x\in\left(-1,1\right)$. Taking the strictly
increasing tanh function of both sides of the equation above, we obtain
\[
\tanh\left(\beta x\right)=x,
\]
which we recognise as the Curie-Weiss equation (\ref{eq:CW}). By
Definition \ref{def:m_beta}, $m\left(\beta\right)$ is the largest
solution of (\ref{eq:CW}). It is clear that for $\beta\leq1$, $m\left(\beta\right)=0$.
Lemma \ref{lem:m_beta_increasing} states that for all $\beta>1$,
$m\left(\beta\right)>0$ holds. As both tanh and the identity function
are odd functions,
\[
\tanh\left(\pm m\left(\beta\right)\beta\right)=\pm m\left(\beta\right)
\]
holds. Hence, $I_{\beta}$ has the minima stated in the proposition.

The statement about the minima of $J_{\beta}$ follows from Lemma
\ref{lem:contr_princ_min}, and Lemma \ref{lem:exponential_conv}
supplies the exponential convergence statement for all closed sets
that do not contain $m\left(\beta\right)^{2}$.
\end{proof}

\subsection{Proof of Theorem \ref{thm:properties_bML_inf}}
\begin{defn}
\label{def:expec_S2_infty}Let for each $N\in\IN$ the function $\vartheta_{N}^{\infty}:\left[-\infty,\infty\right]\backslash I_{c}\rightarrow\IR$
be defined by
\[
\vartheta_{N}^{\infty}\left(\beta\right)\coloneq\frac{1}{1-\beta}\frac{1}{N},\;\beta\leq b_{1},\quad\vartheta_{N}^{\infty}\left(\beta\right)\coloneq m\left(\beta\right)^{2},\;\beta\geq b_{2},\quad\textup{and}\quad\vartheta_{N}^{\infty}\left(-\infty\right)\coloneq0.
\]
\end{defn}

\begin{lem}
\label{lem:expec_S2_infty}$\vartheta_{N}^{\infty}$ is strictly increasing
and continuously differentiable on $\IR$.\\
The inverse function $\left(\vartheta_{N}^{\infty}\right)^{-1}:\left[0,1\right]\backslash\left(\frac{1}{1-b_{1}}\frac{1}{N},m\left(b_{2}\right)^{2}\right)\rightarrow\left[-\infty,\infty\right]\backslash I_{c}$
exists and is strictly increasing and continuously differentiable
on $\left(0,\frac{1}{1-b_{1}}\frac{1}{N}\right]\cup\left[m\left(b_{2}\right)^{2},1\right)$.
The derivative has the value
\[
\left(\left(\vartheta_{N}^{\infty}\right)^{-1}\right)'\left(y\right)=\frac{1}{m'\left(m^{-1}\left(\sqrt{y}\right)\right)}\frac{1}{2\sqrt{y}},\quad y\in\left[m\left(b_{2}\right)^{2},1\right).
\]
\end{lem}

\begin{proof}
The existence of $\left(\vartheta_{N}^{\infty}\right)^{-1}$ and its
strict monotonicity follow from the strict monotonicity of $\vartheta_{N}^{\infty}$.
The continuous differentiability follows by the inverse function theorem
given the properties of the function $m$ in Lemmas \ref{lem:m_beta_increasing}
and \ref{lem:m_beta_diff}. Since $m$ is strictly increasing and
differentiable, so is $m^{-1}$, and therefore so is $\left(\vartheta_{N}^{\infty}\right)^{-1}$.
The value of $\left(\vartheta_{N}^{\infty}\right)^{-1}$, for any
$y\in\left[m\left(b_{2}\right)^{2},1\right]$, is
\[
\left(\vartheta_{N}^{\infty}\right)^{-1}\left(y\right)=m^{-1}\left(\sqrt{y}\right).
\]
Employing the inverse function theorem we obtain that $\left(\vartheta_{N}^{\infty}\right)^{-1}$
is differentiable and its derivative is
\[
\left(\left(\vartheta_{N}^{\infty}\right)^{-1}\right)'\left(y\right)=\left(m^{-1}\right)'\left(\sqrt{y}\right)\frac{1}{2\sqrt{y}}=\frac{1}{m'\left(m^{-1}\left(\sqrt{y}\right)\right)}\frac{1}{2\sqrt{y}}
\]
for all $y\in\left[m\left(b_{2}\right)^{2},1\right)$.
\end{proof}
\begin{rem}
We can use the function $\vartheta_{N}^{\infty}$ to express the estimator
$\hat{\beta}_{N}^{\infty}$ as
\[
\hat{\beta}_{N}^{\infty}=\left(\vartheta_{N}^{\infty}\right)^{-1}\circ T/N^{2}.
\]
\end{rem}

\begin{proof}[Proof of Theorem \ref{thm:properties_bML_inf}]
Recall Definitions \ref{def:T_stat}, \ref{def:estimator_large_N},
and \ref{def:beta_tilde} of $T$, $\hat{\beta}_{N}^{\infty}$, and
$\tilde{\beta}_{N}$ .
\begin{enumerate}
\item Since $\IE\,T=\IE_{\beta,N}S^{2}$ and $S^{2}$ is a bounded random
variable, the weak law of large numbers says
\begin{equation}
T\xrightarrow[n\rightarrow\infty]{\textup{p}}\IE_{\beta,N}S^{2}=\begin{cases}
\frac{N}{1-\tilde{\beta}_{N}} & \textup{if }\beta\in I_{h},\\
m\left(\tilde{\beta}_{N}\right)^{2}N^{2} & \textup{if }\beta\in I_{l}.
\end{cases}\label{eq:T_WLLN}
\end{equation}
Recall the intervals $J_{h}$ and $J_{l}$ from Definition \ref{def:intervals}.
For a fixed sample $\left(x^{(1)},\ldots,x^{(n)}\right)\in\Omega_{N}^{n}$,
by Definition \ref{def:estimator_large_N}, $\hat{\beta}_{N}^{\infty}$
satisfies
\begin{equation}
T\left(x^{(1)},\ldots,x^{(n)}\right)=\begin{cases}
\frac{N}{1-\hat{\beta}_{N}^{\infty}\left(x^{(1)},\ldots,x^{(n)}\right)} & \textup{if }T\left(x^{(1)},\ldots,x^{(n)}\right)\in J_{h},\\
m\left(\hat{\beta}_{N}^{\infty}\left(x^{(1)},\ldots,x^{(n)}\right)\right)^{2}N^{2} & \textup{if }T\left(x^{(1)},\ldots,x^{(n)}\right)\in J_{l}.
\end{cases}\label{eq:T_beta_inf}
\end{equation}
Let $\beta\in I_{h}$. We show $\frac{N}{1-\hat{\beta}_{N}^{\infty}.}\xrightarrow[n\rightarrow\infty]{\textup{p}}\frac{N}{1-\tilde{\beta}_{N}}$.
The statement then follows by Theorem \ref{thm:cont_mapping}.\\
Let $\varepsilon>0$. Assume without loss of generality that $\varepsilon$
is small enough that $\left|T-\IE_{\beta,N}S^{2}\right|\leq\varepsilon$
implies $T\in J_{h}$ (cf. Remark \ref{rem:intervals}). We write
\begin{align*}
\IP\left\{ \left|\frac{N}{1-\hat{\beta}_{N}^{\infty}}-\frac{N}{1-\tilde{\beta}_{N}}\right|>\varepsilon\right\}  & =\IP\left\{ \left|\frac{N}{1-\hat{\beta}_{N}^{\infty}}-\frac{N}{1-\tilde{\beta}_{N}}\right|>\varepsilon,\left|T-\IE_{\beta,N}S^{2}\right|\leq\varepsilon\right\} \\
 & \quad+\IP\left\{ \left|\frac{N}{1-\hat{\beta}_{N}^{\infty}}-\frac{N}{1-\tilde{\beta}_{N}}\right|>\varepsilon,\left|T-\IE_{\beta,N}S^{2}\right|>\varepsilon\right\} .
\end{align*}
The latter summand is smaller than or equal to
\[
\IP\left\{ \left|T-\IE_{\beta,N}S^{2}\right|>\varepsilon\right\} 
\]
which converges to $0$ by (\ref{eq:T_WLLN}).\\
We turn to the first summand. By Definition \ref{def:beta_tilde},
\[
\frac{N}{1-\tilde{\beta}_{N}}=\IE_{\beta,N}S^{2}.
\]
By Definition \ref{def:estimator_large_N} and under the assumption
$T\in J_{h}$, we have
\[
\frac{N}{1-\hat{\beta}_{N}^{\infty}}=T.
\]
Combining the last two displays, we obtain
\[
\IP\left\{ \left|\frac{N}{1-\hat{\beta}_{N}^{\infty}}-\frac{N}{1-\tilde{\beta}_{N}}\right|>\varepsilon,\left|T-\IE_{\beta,N}S^{2}\right|\leq\varepsilon\right\} =\IP\left\{ \left|T-\IE_{\beta,N}S^{2}\right|>\varepsilon,\left|T-\IE_{\beta,N}S^{2}\right|\leq\varepsilon\right\} =\IP\,\emptyset=0.
\]
We have therefore proved
\[
\IP\left\{ \left|\frac{N}{1-\hat{\beta}_{N}^{\infty}}-\frac{N}{1-\tilde{\beta}_{N}}\right|>\varepsilon\right\} \xrightarrow[n\rightarrow\infty]{}0,
\]
and thus $\frac{N}{1-\hat{\beta}_{N}^{\infty}}\xrightarrow[n\rightarrow\infty]{\textup{p}}\frac{N}{1-\tilde{\beta}_{N}}$
holds.\\
The case $\beta\in I_{l}$ is treated analogously: we first show $m\left(\hat{\beta}_{N}^{\infty}\right)^{2}N^{2}\xrightarrow[n\rightarrow\infty]{\textup{p}}m\left(\tilde{\beta}_{N}\right)^{2}N^{2}$,
and then use that the function $\beta\in\left(1,\infty\right)\mapsto m\left(\beta\right)\in\left(0,1\right)$
is strictly increasing by Lemma \ref{lem:m_beta_increasing}. Thus,
$\hat{\beta}_{N}^{\infty}\xrightarrow[n\rightarrow\infty]{\textup{p}}\tilde{\beta}_{N}$
follows from Lemma \ref{lem:expec_S2_infty}, Theorem \ref{thm:cont_mapping},
and $m\left(\hat{\beta}_{N}^{\infty}\right)^{2}N^{2}\xrightarrow[n\rightarrow\infty]{\textup{p}}m\left(\tilde{\beta}_{N}\right)^{2}N^{2}$.
\item As a corollary to Proposition \ref{prop:appr_moments}, we have
\begin{align*}
\frac{\IE_{\beta,N}S^{2}}{N} & \xrightarrow[N\rightarrow\infty]{}\frac{1}{1-\beta}\quad\textup{if }\beta\in I_{h},\\
\frac{\IE_{\beta,N}S^{2}}{N^{2}} & \xrightarrow[N\rightarrow\infty]{}m\left(\beta\right)^{2}\quad\textup{if }\beta\in I_{l}.
\end{align*}
By Proposition \ref{prop:ES2_fin}, the mapping $\beta\mapsto\IE_{\beta,N}S^{2}$
is continuous. Hence, for $\beta\in I_{h}$,
\[
\frac{1}{1-\tilde{\beta}_{N}}=\frac{\IE_{\beta,N}S^{2}}{N}\xrightarrow[N\rightarrow\infty]{}\frac{1}{1-\beta},
\]
which is equivalent to $\tilde{\beta}_{N}\xrightarrow[N\rightarrow\infty]{}\beta$.\\
Recall Definition \ref{def:m_beta} of $\beta\mapsto m\left(\beta\right)$
and its continuity which follows from Lemma \ref{lem:m_beta_diff}.
For $\beta\in I_{l}$,
\[
m\left(\tilde{\beta}_{N}\right)^{2}=\frac{\IE_{\beta,N}S^{2}}{N^{2}}\xrightarrow[N\rightarrow\infty]{}m\left(\beta\right)^{2}.
\]
The function $\beta\in\left(1,\infty\right)\mapsto m\left(\beta\right)\in\left(0,1\right)$
is strictly increasing by Lemma \ref{lem:m_beta_increasing}. Thus,
$\tilde{\beta}_{N}\xrightarrow[N\rightarrow\infty]{}\beta$ follows
from Lemma \ref{lem:m_beta_diff} and $m\left(\tilde{\beta}_{N}\right)^{2}\xrightarrow[N\rightarrow\infty]{}m\left(\beta\right)^{2}$.
\item We first show the result for $\beta\in I_{h}$. We once again use
Definitions \ref{def:T_stat}, \ref{def:estimator_large_N}, and \ref{def:beta_tilde}
of $T$, $\hat{\beta}_{N}^{\infty}$, and $\tilde{\beta}_{N}$, and
the law of large numbers (\ref{eq:T_WLLN}). In addition to the law
of large numbers, we also note that
\[
T-\IE\,T=\frac{1}{n}\sum_{t=1}^{n}\left[\left(\sum_{i=1}^{N}X_{i}^{(t)}\right)^{2}-\IE\,T\right]
\]
is the sum of i.i.d.\! random variables with
\[
\IE\left[\left(\sum_{i=1}^{N}X_{i}^{(t)}\right)^{2}-\IE\,T\right]=0,\qquad\IV\left[\left(\sum_{i=1}^{N}X_{i}^{(t)}\right)^{2}-\IE\,T\right]=\IE_{\beta,N}S^{2}-\left(\IE_{\beta,N}S\right)^{2}=\IV_{\beta,N}S^{2}
\]
for all $t\in\IN_{n}$. The central limit theorem yields
\begin{equation}
\sqrt{n}\left(T-\IE\,T\right)\xrightarrow[n\rightarrow\infty]{\textup{d}}\mathcal{N}\left(0,\IV_{\beta,N}S^{2}\right).\label{eq:CLT_T}
\end{equation}
For a fixed sample $\left(x^{(1)},\ldots,x^{(n)}\right)\in\Omega_{N}^{n}$
with $T\left(x^{(1)},\ldots,x^{(n)}\right)\in J_{h}$, by Definition
\ref{def:estimator_large_N}, $\hat{\beta}_{N}^{\infty}$ satisfies
\[
T\left(x^{(1)},\ldots,x^{(n)}\right)=\frac{N}{1-\hat{\beta}_{N}^{\infty}\left(x^{(1)},\ldots,x^{(n)}\right)}
\]
and by Definition \ref{def:beta_tilde}
\[
\IE_{\beta,N}T=\IE_{\beta,N}S^{2}=\frac{N}{1-\tilde{\beta}_{N}}.
\]
Joining the last two display, we obtain
\[
T-\IE\,T=\frac{N}{1-\hat{\beta}_{N}^{\infty}}-\frac{N}{1-\tilde{\beta}_{N}}.
\]
Next, we calculate
\[
\frac{1}{1-\hat{\beta}_{N}^{\infty}}-\frac{1}{1-\tilde{\beta}_{N}}=\frac{\hat{\beta}_{N}^{\infty}-\tilde{\beta}_{N}}{\left(1-\hat{\beta}_{N}^{\infty}\right)\left(1-\tilde{\beta}_{N}\right)}.
\]
The last display together with (\ref{eq:CLT_T}) and Theorem \ref{thm:slutsky}
implies
\begin{equation}
\frac{\hat{\beta}_{N}^{\infty}-\tilde{\beta}_{N}}{\left(1-\hat{\beta}_{N}^{\infty}\right)\left(1-\tilde{\beta}_{N}\right)}=\frac{\sqrt{n}}{N}\left(T-\IE\,T\right)\xrightarrow[n\rightarrow\infty]{\textup{d}}\mathcal{N}\left(0,\IV_{\beta,N}\frac{S^{2}}{N}\right).\label{eq:CLT_2}
\end{equation}
We define the following sequences of random variables:
\[
U_{n}\coloneq\left(\hat{\beta}_{N}^{\infty}-\tilde{\beta}_{N}\right)\frac{1-\tilde{\beta}_{N}}{1-\hat{\beta}_{N}^{\infty}},\qquad V_{n}\coloneq\hat{\beta}_{N}^{\infty}-\tilde{\beta}_{N},\qquad n\in\IN.
\]
Then (\ref{eq:CLT_2}) is equivalent to
\[
\frac{\sqrt{n}}{\left(1-\tilde{\beta}_{N}\right)^{2}}U_{n}\xrightarrow[n\rightarrow\infty]{\textup{d}}\mathcal{N}\left(0,\IV_{\beta,N}\frac{S^{2}}{N}\right).
\]
It follows with Theorem \ref{thm:slutsky} that
\begin{equation}
\sqrt{n}\,U_{n}\xrightarrow[n\rightarrow\infty]{\textup{d}}\mathcal{N}\left(0,\left(1-\tilde{\beta}_{N}\right)^{4}\IV_{\beta,N}\frac{S^{2}}{N}\right).\label{eq:CLT_3}
\end{equation}
Next, we show
\begin{equation}
\frac{1-\hat{\beta}_{N}^{\infty}}{1-\tilde{\beta}_{N}}\xrightarrow[n\rightarrow\infty]{\textup{p}}1.\label{eq:frac_conv_p}
\end{equation}
We have
\[
\left|\frac{1-\hat{\beta}_{N}^{\infty}}{1-\tilde{\beta}_{N}}-1\right|=\frac{\left|\hat{\beta}_{N}^{\infty}-\tilde{\beta}_{N}\right|}{1-\tilde{\beta}_{N}}\xrightarrow[n\rightarrow\infty]{\textup{p}}0,
\]
where the convergence in probability follows from statement 1 of this
theorem and Theorem \ref{thm:slutsky}.\\
Since
\[
\frac{V_{n}}{U_{n}}=\frac{1-\hat{\beta}_{N}^{\infty}}{1-\tilde{\beta}_{N}},
\]
we arrive at
\[
\sqrt{n}\,V_{n}=\sqrt{n}\,U_{n}\,\frac{V_{n}}{U_{n}}\xrightarrow[n\rightarrow\infty]{\textup{d}}\mathcal{N}\left(0,\left(1-\tilde{\beta}_{N}\right)^{4}\IV_{\beta,N}\frac{S^{2}}{N}\right)
\]
using (\ref{eq:CLT_3}), (\ref{eq:frac_conv_p}), and Theorem \ref{thm:slutsky}.\\
We use Proposition \ref{prop:appr_moments}:
\begin{align*}
\IV_{\beta,N}\frac{S^{2}}{N} & =\IE\left(\frac{S^{2}}{N}\right)^{2}-\left(\IE\left(\frac{S^{2}}{N}\right)\right)^{2}=\IE\left(\frac{S^{2}}{N}\right)^{2}-\frac{1}{\left(1-\tilde{\beta}_{N}\right)^{2}}\\
 & \approx3\,\frac{1}{\left(1-\tilde{\beta}_{N}\right)^{2}}-\frac{1}{\left(1-\tilde{\beta}_{N}\right)^{2}}=\frac{2}{\left(1-\tilde{\beta}_{N}\right)^{2}}.
\end{align*}
Employing Definition \ref{def:beta_tilde}, we obtain
\begin{align*}
\left(1-\tilde{\beta}_{N}\right)^{4}\IV_{\beta,N}\frac{S^{2}}{N} & \approx\left(1-\tilde{\beta}_{N}\right)^{4}\frac{2}{\left(1-\tilde{\beta}_{N}\right)^{2}}\\
 & =2\left(1-\tilde{\beta}_{N}\right)^{2}\\
 & \xrightarrow[N\rightarrow\infty]{}2\left(1-\beta\right)^{2}.
\end{align*}
Now let $\beta\in I_{l}$. The law of large numbers (\ref{eq:T_WLLN})
and the central limit theorem (\ref{eq:CLT_T}) hold, and for a fixed
sample $\left(x^{(1)},\ldots,x^{(n)}\right)\in\Omega_{N}^{n}$ with
$T\left(x^{(1)},\ldots,x^{(n)}\right)\in J_{l}$, by Definition \ref{def:estimator_large_N},
$\hat{\beta}_{N}^{\infty}$ satisfies
\[
T\left(x^{(1)},\ldots,x^{(n)}\right)=N^{2}m\left(\hat{\beta}_{N}^{\infty}\left(x^{(1)},\ldots,x^{(n)}\right)\right)^{2}
\]
and by Definition \ref{def:beta_tilde}
\[
\IE_{\beta,N}T=\IE_{\beta,N}S^{2}=N^{2}m\left(\tilde{\beta}_{N}\right)^{2}.
\]
Joining the last two displays, we obtain
\[
T-\IE\,T=N^{2}\left(m\left(\hat{\beta}_{N}^{\infty}\left(x^{(1)},\ldots,x^{(n)}\right)\right)^{2}-m\left(\tilde{\beta}_{N}\right)^{2}\right).
\]
By (\ref{eq:CLT_T}), we have
\[
\sqrt{n}\left(m\left(\hat{\beta}_{N}^{\infty}\right)^{2}-m\left(\tilde{\beta}_{N}\right)^{2}\right)\xrightarrow[n\rightarrow\infty]{\textup{d}}\mathcal{N}\left(0,\IV_{\beta,N}\left(\frac{S}{N}\right)^{2}\right).
\]
We want to apply Theorem \ref{thm:delta_method} to the transformation
$f\coloneq\left(\vartheta_{N}^{\infty}\right)^{-1}$, but face the
difficulty that $\left(\vartheta_{N}^{\infty}\right)^{-1}$ is not
differentiable on the entire range of $T$. We define
\[
W_{n}\coloneq\left(m\left(\hat{\beta}_{N}^{\infty}\right)^{2}-m\left(\tilde{\beta}_{N}\right)^{2}\right)\,\II_{\left\{ m\left(b_{2}\right)^{2}<m\left(\hat{\beta}_{N}^{\infty}\right)^{2}<\frac{\IE_{\beta,N}S^{2}+1}{2}\right\} },\quad n\in\IN,
\]
and apply Lemma \ref{lem:conv_restr_sequence} to the sequence $Y_{n}\coloneq\sqrt{n}\left(m\left(\hat{\beta}_{N}^{\infty}\right)^{2}-m\left(\tilde{\beta}_{N}\right)^{2}\right)$
with $K\coloneq\left(m\left(b_{2}\right)^{2},\frac{\IE_{\beta,N}S^{2}+1}{2}\right)^{c}$,
$B_{n}$ the intersection of $K$ with the range of $Y_{n}$ for each
$n\in\IN$, and $\nu\coloneq\mathcal{N}\left(0,\IV_{\beta,N}\left(\frac{S}{N}\right)^{2}\right)$.
We set $M_{n}\coloneq\sqrt{n}$. By the large deviations principle
in statement 4, which we will show next, we have due to $N^{2}m\left(b_{2}\right)^{2}<\IE_{\beta,N}S^{2}<N^{2}$,
for the closed set $K$,
\[
\IP\left\{ Y_{n}\in B\right\} \leq2\exp\left(-n\inf_{x\in K}\Lambda_{S^{2}}^{*}\left(x\right)\right),\quad n\in\IN.
\]
Therefore, $\IP\left\{ Y_{n}\in B\right\} =o\left(\frac{1}{M_{n}}\right)$
holds, and we can apply Lemma \ref{lem:conv_restr_sequence} to conclude
\[
W_{n}\xrightarrow[n\rightarrow\infty]{\textup{d}}\mathcal{N}\left(0,\IV_{\beta,N}\left(\frac{S}{N}\right)^{2}\right).
\]
We apply Theorem \ref{thm:delta_method} to the sequence $W_{n}$.
Set $D\coloneq\left[m\left(b_{2}\right)^{2},\frac{\IE_{\beta,N}\left(\frac{S}{N}\right)^{2}+1}{2}\right]$
and $f:D\rightarrow\IR$
\[
f\left(y\right)\coloneq\left(\vartheta_{N}^{\infty}\right)^{-1}\left(y\right)=m^{-1}\left(\sqrt{y}\right),\quad y\in D.
\]
$f$ is continuously differentiable and strictly positive on the compact
set $D$. Therefore,
\begin{align*}
 & \quad\sqrt{n}\left(\hat{\beta}_{N}^{\infty}-\tilde{\beta}_{N}\right)\II_{\left\{ m\left(b_{2}\right)^{2}<m\left(\hat{\beta}_{N}^{\infty}\right)^{2}<\frac{\IE_{\beta,N}S^{2}+1}{2}\right\} }\\
 & =\sqrt{n}\left(f\left(W_{n}\right)-f\left(\IE_{\beta,N}\left(\frac{S}{N}\right)^{2}\right)\right)\xrightarrow[n\rightarrow\infty]{\textup{d}}\mathcal{N}\left(0,\left(f'\left(\mu\right)\right)^{2}\sigma^{2}\right)
\end{align*}
holds. We apply Lemma \ref{lem:conv_restr_sequence} once more to
conclude that
\[
\sqrt{n}\left(\hat{\beta}_{N}^{\infty}-\tilde{\beta}_{N}\right)\xrightarrow[n\rightarrow\infty]{\textup{d}}\mathcal{N}\left(0,\left(f'\left(\mu\right)\right)^{2}\sigma^{2}\right)
\]
is satisfied. The statement concerning the limiting variance as $n\rightarrow\infty$
then follows by Lemma \ref{lem:expec_S2_infty}. Finally, we note
that
\[
\frac{1}{\left[2m\left(\tilde{\beta}_{N}\right)m'\left(\tilde{\beta}_{N}\right)\right]^{2}}\xrightarrow[N_{\lambda}\rightarrow\infty]{}\frac{1}{\left[2m\left(\beta\right)m'\left(\beta\right)\right]^{2}}
\]
holds due to the continuity of $m$ and $m'$, and
\[
\IV_{\beta,N}\left(\frac{S}{N}\right)^{2}\xrightarrow[N\rightarrow\infty]{}0
\]
follows from the law of large numbers
\[
\left(\frac{S}{N}\right)^{2}\xrightarrow[N\rightarrow\infty]{\textup{p}}m\left(\beta\right)^{2},
\]
which itself is a corollary of Proposition \ref{prop:LDP_S_N}.
\item For each $N\in\IN$, $\left(\frac{T}{N^{2}}\right)_{N\in\IN}$ satisfies
a large deviations principle with rate $n$ and rate function $I=\Lambda_{\left(\frac{S}{N}\right)^{2}}^{*}$
by Proposition \ref{prop:conv_stat}. Theorem \ref{thm:contr_princ}
yields a large deviations principle for $\hat{\beta}_{N}^{\infty}=\left(\vartheta_{N}^{\infty}\right)^{-1}\circ T/N^{2}$
with rate $n$ and good rate function 
\begin{equation}
J:\left[-\infty,\infty\right]\rightarrow\left[0,\infty\right],\quad J\left(y\right)\coloneq\inf\left\{ \Lambda_{\left(\frac{S}{N}\right)^{2}}^{*}\left(x\right)\,|\,x\in\left[-\infty,\infty\right],x^{2}=y\right\} ,\quad y\in\IR.\label{eq:J}
\end{equation}
We prove the exponential convergence statement.\\
We assume $\beta\in I_{l}$. The proof for $\beta\in I_{h}$ is analogous.
Let $\varepsilon>0$. The function mapping $\beta$ to $m\left(\beta\right)^{2}$
is strictly increasing by Lemma \ref{lem:m_beta_increasing} and continuous
by Lemma \ref{lem:m_beta_diff}. Therefore, it maps the interval $\left(\beta-\varepsilon,\beta+\varepsilon\right)$
bijectively to the interval $\left(m\left(\beta-\varepsilon\right)^{2},m\left(\beta+\varepsilon\right)^{2}\right)$.
Set $r_{1}\coloneq m\left(\beta\right)^{2}-m\left(\beta-\varepsilon\right)^{2}$
and $r_{2}\coloneq m\left(\beta+\varepsilon\right)^{2}-m\left(\beta\right)^{2}$.
Then $r_{1},r_{2}>0$, and we define $\delta\coloneq\min\left\{ r_{1},r_{2}\right\} >0$.\\
Next, we define the sets\\
\begin{align*}
G & \coloneq\left\{ \left|\hat{\beta}_{N}^{\infty}-\beta\right|<\varepsilon\right\} ,\quad A\coloneq\left\{ \left|m\left(\hat{\beta}_{N}^{\infty}\right)^{2}-m\left(\beta\right)^{2}\right|<\delta\right\} ,\\
D & \coloneq\left\{ \left|m\left(\hat{\beta}_{N}^{\infty}\right)^{2}-m\left(\tilde{\beta}_{N}\right)^{2}\right|<\frac{\delta}{2}\right\} ,\quad E\coloneq\left\{ \left|m\left(\tilde{\beta}_{N}\right)^{2}-m\left(\beta\right)^{2}\right|<\frac{\delta}{2}\right\} .
\end{align*}
We have the inclusions
\begin{equation}
D\cap E\subset A\subset G.\label{eq:set_incl}
\end{equation}
We use Proposition \ref{prop:LDP_S_N} and the large deviations principle
it provides for the sequence $\left(\frac{S}{N}\right)^{2}$ to show
the existence of some $N_{\varepsilon}\in\IN$ such that $\left|\IE_{\beta,N}\left(\frac{S}{N}\right)^{2}-m\left(\beta\right)^{2}\right|<\frac{\delta}{2}$
holds for all $N\geq N_{\varepsilon}$. Then, by Definition \ref{def:beta_tilde},
$\left|m\left(\tilde{\beta}_{N}\right)^{2}-m\left(\beta\right)^{2}\right|<\frac{\delta}{2}$,
and hence $\IP\,E^{c}=0$, is satisfied for all $N\geq N_{\varepsilon}$.\\
Proposition \ref{prop:LDP_S_N} provides constants $B,C>0$ such that
\[
\IP\left\{ \left|\left(\frac{S}{N}\right)^{2}-m\left(\beta\right)^{2}\right|\geq\frac{\delta}{4}\right\} \leq C\exp\left(-NB\right),\quad N\in\IN.
\]
We have
\begin{align*}
\IE_{\beta,N}\left(\frac{S}{N}\right)^{2} & =\IE_{\beta,N}\left(\frac{S}{N}\right)^{2}\II_{\left\{ \left(\frac{S}{N}\right)^{2}\leq m\left(\beta\right)^{2}-\frac{\delta}{4}\right\} }+\IE_{\beta,N}\left(\frac{S}{N}\right)^{2}\II_{\left\{ \left(\frac{S}{N}\right)^{2}>m\left(\beta\right)^{2}-\frac{\delta}{4}\right\} }\\
 & \geq0+\left(m\left(\beta\right)^{2}-\frac{\delta}{4}\right)\IP\left\{ \left(\frac{S}{N}\right)^{2}>m\left(\beta\right)^{2}-\frac{\delta}{4}\right\} \\
 & \geq\left(m\left(\beta\right)^{2}-\frac{\delta}{4}\right)\left(1-C\exp\left(-NB\right)\right).
\end{align*}
We define $N_{\varepsilon}$ to be the smallest natural number such
that
\[
\left(m\left(\beta\right)^{2}-\frac{\delta}{4}\right)\left(1-C\exp\left(-N_{\varepsilon}B\right)\right)>m\left(\beta\right)^{2}-\frac{\delta}{2}
\]
holds, that is
\[
N_{\varepsilon}\coloneq\min\left\{ N\in\IN\,\left|\,N>-\frac{1}{B}\ln\frac{\delta}{4C\left(m\left(\beta\right)^{2}-\frac{\delta}{4}\right)}\right.\right\} .
\]
Then by the above calculation, we have for all $N\geq N_{\varepsilon}$,
\[
\IP\,E^{c}=0.
\]
From now on assume $N\geq N_{\varepsilon}$. Now we use the large
deviations principle for $T/N^{2}$ with rate $n$ and rate function
$\Lambda_{\left(\frac{S}{N}\right)^{2}}^{*}$. If we set
\[
\Phi_{n}\coloneq2\exp\left(-n\min\left\{ \Lambda_{\left(\frac{S}{N}\right)^{2}}^{*}\left(m\left(\tilde{\beta}_{N}\right)^{2}-\frac{\delta}{2}\right),\Lambda_{\left(\frac{S}{N}\right)^{2}}^{*}\left(m\left(\tilde{\beta}_{N}\right)^{2}+\frac{\delta}{2}\right)\right\} \right),\quad n\in\IN,
\]
then Proposition \ref{prop:conv_stat} implies that we have for the
closed set $D^{c}$ the upper bound
\[
\IP\,D^{c}\leq\Phi_{n},\quad n\in\IN.
\]
Then we calculate
\[
\IP\left(D^{c}\cup E^{c}\right)\leq\IP\,D^{c}+\IP\,E^{c}=\IP\,D^{c},
\]
and
\begin{align*}
\IP\left(D\cap E\right) & =\IP\left(D^{c}\cup E^{c}\right)^{c}=1-\IP\left(D^{c}\cup E^{c}\right)\\
 & \geq1-\IP\,D^{c}\geq1-\Phi_{n}
\end{align*}
for all $n\in\IN$. As $G\supset D\cap E$ by (\ref{eq:set_incl}),
\[
\IP\,G\geq\IP\left(D\cap E\right)\geq1-\Phi_{n}
\]
holds for all $N\geq N_{\varepsilon}$ and all $n\in\IN$. This concludes
the proof of the exponential convergence statement.
\end{enumerate}
\end{proof}
\begin{rem}
\label{rem:bold_J}We define the multivariate good rate function $\boldsymbol{J}:\left[-\infty,\infty\right]^{M}\rightarrow\left[0,\infty\right]$
from the statement of Theorem \ref{thm:properties_bML_inf} by setting
\begin{equation}
\boldsymbol{J}\left(x\right)\coloneq\sum_{\lambda=1}^{M}J_{\lambda}\left(x_{\lambda}\right),\quad x\in\left[-\infty,\infty\right]^{M},\label{eq:bold_J}
\end{equation}
with each of the univariate rate functions $J_{\lambda}$ being of
the form (\ref{eq:J}).
\end{rem}

\section{\label{sec:Optimal-Weights}Optimal Weights in a Two-Tier Voting
System}

A two-tier voting system refers to a framework in which the population
of a state or a union of states is divided into $M$ distinct groups
(such as member states). Each group selects a representative who participates
in a central council responsible for making decisions on behalf of
the union. These representatives vote \textquoteleft yes\textquoteright{}
or \textquoteleft no\textquoteright{} in alignment with the majority
opinion within their own group.

We recall that each group $\lambda\in\mathbb{N}_{M}$ consists of
$N_{\lambda}\in\mathbb{N}$ individuals. The vote of an individual
is denoted by the random variable $X_{\lambda i}$, where $\lambda$
identifies the group and $i\in N_{\lambda}$ specifies the individual
within it. The group voting margins $S_{\lambda}$, as defined in
(\ref{eq:S_lambda}), are used to quantify majority support in each
group. The total voting margin across all groups is denoted by
\[
\ensuremath{\bar{S}\coloneq\sum_{\lambda=1}^{M}S_{\lambda}}.
\]

The council vote for each group is determined based on its corresponding
group voting margin $S_{\lambda}$.
\begin{defn}
\label{def:council-vote}The council vote of each group $\lambda\in\IN_{M}$
is
\begin{align*}
\chi_{\lambda} & \coloneq\begin{cases}
1 & \textup{if }S_{\lambda}>0,\\
-1 & \textup{otherwise.}
\end{cases}
\end{align*}
Since the groups generally differ in size, it makes sense to assign
distinct voting weights $w_{\lambda}$ to each representative, dependent
on the size of their respective group. In certain cases, it may be
possible to form groups of roughly equal population. For instance,
in parliamentary elections such as those for the U.S. House of Representatives,
the country is divided into districts with nearly equal population
sizes, each assigned one representative with a voting weight of 1.
While this method works well within a single nation, it becomes impractical
in broader contexts. Redrawing the boundaries of countries or member
states (like those in the United Nations or European Union) to create
equally populated districts would conflict with national sovereignty,
making such restructuring politically unfeasible. Therefore, the challenge
of assigning fair voting weights to groups of unequal size must be
addressed.
\end{defn}

We follow a procedure much studied in the literature which consists
in minimising what the so called \emph{democracy deficit} (see Definition
\ref{def:democracy-deficit} below), that is, the difference between
the actual council vote and the outcome of an ideal direct referendum
of the entire population. This idea was first introduced in the context
of binary voting by Felsenthal and Machover \cite{FelsMach1999} and
has since been expanded upon in various studies (\cite{Ki2007,ZyczSlom2014,KirsLang2014,MaasNape2012,To2020phd,KirsToth2022,KT2021c}
to cite but a few). In informal terms, minimising the democracy deficit
means assigning voting weights so that the council\textquoteright s
decisions closely approximate the outcome of a full-population referendum,
on average.

Alternative approaches to assigning optimal weights focus on welfare
metrics or on ensuring equal influence among all voters in the population.
A landmark example of this approach is Penrose\textquoteright s work
\cite{Pen1946}, which established the square root law for weights
which equalises all voters' influence under the assumption of independent
voting. Additional studies \cite{BeisBove2007,KoMaTrLa2013,KurMaaNa2017}
have further developed these ideas, examining voting weights through
the lenses of welfare and individual voter influence.

A necessary ingredient for the definition of the democracy deficit
is a model of how the population votes. For this, we adopt the CWM
$\mathbb{P}_{\boldsymbol{\beta},\boldsymbol{N}}$ (see Definition
\ref{def:CWM}), where $\boldsymbol{\beta}=(\beta_{1},\ldots,\beta_{M})\geq0$
and $\boldsymbol{N}=(N_{1},\ldots,N_{M})\in\mathbb{N}^{M}$ are fixed.
Once this model is in place, we can formally define the democracy
deficit, which will guide how voting weights should be optimally assigned
to the council representatives.
\begin{defn}
\label{def:democracy-deficit}The democracy deficit for a set of voting
weights $w_{1},\ldots,w_{M}\in\IR$ is given by
\[
\IE_{\boldsymbol{\beta},\boldsymbol{N}}\left[\bar{S}-\sum_{\lambda=1}^{M}w_{\lambda}\chi_{\lambda}\right]^{2}.
\]
We will call any vector $\left(w_{1},\ldots,w_{M}\right)\in\IR$ of
weights which minimises the democracy deficit `optimal.'
\end{defn}

The following result characterises the optimal weights in terms of
certain moments of the CWM.
\begin{prop}
\label{prop:opt_weights}For all $\boldsymbol{\beta}=\left(\beta_{1},\ldots,\beta_{M}\right)\geq0$,
the \emph{optimal weights} are given by
\[
w_{\lambda}=\IE_{\beta_{\lambda},N_{\lambda}}\left|S_{\lambda}\right|,\quad\lambda\in\IN_{M}.
\]
\end{prop}

\begin{proof}
This is Proposition 45 in \cite{BaMeSiTo2025}.
\end{proof}
Using asymptotic approximations for the moments of the voting margins,
we can further develop the above condition and give a theorem about
the asymptotic behaviour of the optimal weights. To prove Theorem
\ref{thm:opt_weights} below, we will employ the method of moments
(see Theorem \ref{thm:mm}).
\begin{thm}
\label{thm:opt_weights}For all $\beta_{1},\ldots,\beta_{M}\in\left[0,\infty\right)\backslash\left\{ 1\right\} $,
the \emph{optimal weights} are asymptotically equal to
\[
w_{\lambda}\approx\begin{cases}
\sqrt{\frac{2}{\pi}}\sqrt{\frac{1}{1-\beta_{\lambda}}}\sqrt{N_{\lambda}} & \textup{if }\beta_{\lambda}<1,\\
m\left(\beta\right)N_{\lambda} & \textup{if }\beta_{\lambda}>1,
\end{cases}\quad\lambda\in\IN_{M}.
\]
\end{thm}

\begin{proof}
This result was first proved in \cite{Ki2007}. For the reader's convenience,
we present a short proof that employs Proposition \ref{prop:appr_moments}.

By Proposition \ref{prop:opt_weights}, we have the optimality condition
\[
w_{\lambda}=\IE_{\beta_{\lambda},N_{\lambda}}\left|S_{\lambda}\right|,\quad\lambda\in\IN_{M}.
\]
Hence, we need to calculate the moment $\IE_{\beta_{\lambda},N_{\lambda}}\left|S_{\lambda}\right|$.
Instead of doing this lengthy calculation which follows the pattern
presented in Propositions \ref{prop:appr_correlations} and \ref{prop:appr_moments},
we use the method of moments to derive the limit of $\IE_{\beta_{\lambda},N_{\lambda}}\left|S_{\lambda}\right|$
as $N\rightarrow\infty$.

Let $\beta_{\lambda}\in I_{h}$. We can apply the method of moments
(Theorem \ref{thm:mm}), provided there are constants $A,B\in\IR$
and a sequence $\left(m_{k}\right)_{k\in\IN}$ of real numbers with
the property that
\[
\lim_{N\rightarrow\infty}\IE_{\beta_{\lambda},N_{\lambda}}\left(\frac{S_{\lambda}}{\sqrt{N_{\lambda}}}\right)^{k}=m_{k}
\]
 holds for all $k\in\IN$, and the sequence $\left(m_{k}\right)_{k\in\IN}$
satisfies
\[
m_{k}\leq AB^{k}k!,\quad k\in\IN.
\]
By Proposition \ref{prop:appr_moments}, this condition is satisfied
for
\[
m_{k}=\begin{cases}
k!!\,\left(\frac{1}{1-\beta_{\lambda}}\right)^{k} & \textup{if }k\textup{ is even,}\\
0 & \textup{if }k\textup{ is odd.}
\end{cases}
\]
These are the moments of the normal distribution $\mathcal{N}\left(0,\frac{1}{1-\beta_{\lambda}}\right)$.
By Theorem \ref{thm:mm}, the sequence $\left(\frac{S_{\lambda}}{\sqrt{N_{\lambda}}}\right)_{N_{\lambda}\in\IN}$
converges in distribution to $\mathcal{N}\left(0,\frac{1}{1-\beta_{\lambda}}\right)$.
By Theorem \ref{thm:cont_mapping}, the sequence $\left(\frac{\left|S_{\lambda}\right|}{\sqrt{N_{\lambda}}}\right)_{N_{\lambda}\in\IN}$
converges to a half-normal distribution. Let $\sigma^{2}>0$ and $Z\sim\mathcal{N}\left(0,\sigma^{2}\right)$.
Then the half-normal distribution with scale parameter $\sigma$ is
the distribution of the random variable $\left|Z\right|$. We calculate
the moment $\IE\,\left|Z\right|$:
\begin{align*}
\IE\,\left|Z\right| & =\sqrt{\frac{2}{\pi\sigma^{2}}}\int_{0}^{\infty}x\,\exp\left(-\frac{x^{2}}{2\sigma^{2}}\right)\,\textup{d}x\\
 & =\sqrt{\frac{2}{\pi}}\,\sigma\int_{0}^{\infty}\left(-\exp\left(-y\right)\right)\,\textup{d}y\\
 & =\sqrt{\frac{2}{\pi}}\,\sigma.
\end{align*}
In our application, $\sigma=\sqrt{\frac{1}{1-\beta_{\lambda}}}$.

Since the sequence $\left(\IE_{\beta_{\lambda},N_{\lambda}}\frac{S_{\lambda}^{2}}{N_{\lambda}}\right)_{M_{\lambda}\in\IN}$
converges by Proposition \ref{prop:appr_moments}, it is uniformly
bounded, and using the Cauchy-Schwarz inequality, so is
\[
\IE_{\beta_{\lambda},N_{\lambda}}\frac{\left|S_{\lambda}\right|}{\sqrt{N_{\lambda}}}\leq\left(\IE_{\beta_{\lambda},N_{\lambda}}\frac{S_{\lambda}^{2}}{N_{\lambda}}\right)^{\frac{1}{2}}.
\]
As the sequence of moments $\left(\IE_{\beta_{\lambda},N_{\lambda}}\frac{\left|S_{\lambda}\right|}{\sqrt{N_{\lambda}}}\right)_{N_{\lambda}\in\IN}$
is uniformly bounded over all $N_{\lambda}\in\IN$, the convergence
\[
\IE_{\beta_{\lambda},N_{\lambda}}\frac{\left|S_{\lambda}\right|}{\sqrt{N_{\lambda}}}\xrightarrow[N_{\lambda}\rightarrow\infty]{}\IE\,\left|Z\right|
\]
follows from
\[
\left(\frac{S_{\lambda}}{\sqrt{N_{\lambda}}}\right)_{N_{\lambda}\in\IN}\xrightarrow[N_{\lambda}\rightarrow\infty]{\textup{d}}\mathcal{N}\left(0,\frac{1}{1-\beta_{\lambda}}\right)
\]
and Theorem \ref{thm:cont_mapping}. This concludes the proof for
$\beta\in I_{h}$.

Now let $\beta\in I_{l}$. Recall Definition \ref{def:m_beta} of
the expression $m\left(\beta\right)$. By Proposition \ref{prop:appr_moments},
\[
\IE_{\beta_{\lambda},N_{\lambda}}\left(\frac{S_{\lambda}}{N_{\lambda}}\right)^{2}\xrightarrow[N_{\lambda}\rightarrow\infty]{}m\left(\beta_{\lambda}\right)^{2}
\]
holds. Therefore, employing Theorem \ref{thm:mm},
\[
\left(\left(\frac{S_{\lambda}}{N_{\lambda}}\right)^{2}\right)_{N_{\lambda}\in\IN}\xrightarrow[N_{\lambda}\rightarrow\infty]{\textup{d}}\delta_{m\left(\beta_{\lambda}\right)^{2}},
\]
where the symbol $\delta_{x}$ for any $x\in\IR$ stands for the Dirac
measure at $x$. Using Theorem \ref{thm:cont_mapping}, we obtain
\[
\left(\left|\frac{S_{\lambda}}{N_{\lambda}}\right|\right)_{N_{\lambda}\in\IN}\xrightarrow[N_{\lambda}\rightarrow\infty]{\textup{d}}\delta_{m\left(\beta_{\lambda}\right)}.
\]
Using the Cauchy--Schwarz inequality as above, we see that the sequence
of moments $\left(\IE_{\beta_{\lambda},N_{\lambda}}\left|\frac{S_{\lambda}}{N_{\lambda}}\right|\right)_{N_{\lambda}\in\IN}$
converges to $m\left(\beta\right)$. This yields the claim for $\beta\in I_{l}$.
\end{proof}
We now note that since
\[
\sqrt{\frac{1}{1-\beta_{\lambda}}}\leq\frac{1}{1-\beta_{\lambda}},\quad\beta_{\lambda}\in I_{h},\qquad m\left(\beta_{\lambda}\right)^{2}<m\left(\beta_{\lambda}\right),\quad\beta_{\lambda}\in I_{l}
\]
holds, we have
\begin{equation}
\sqrt{\frac{2}{\pi}}\sqrt{\frac{1}{1-\gamma}}\sqrt{N_{\lambda}}<m\left(\eta\right)N_{\lambda}\label{eq:weight_ineq}
\end{equation}
for all $\gamma\in I_{h}$ and all $\eta\in I_{l}$ by condition (\ref{eq:separation})
for all groups $\lambda\in\IN_{M}$. Moreover, we have
\begin{equation}
\lim_{N_{\lambda}\rightarrow\infty}\frac{\sqrt{\frac{2}{\pi}}\sqrt{\frac{1}{1-\gamma}}\sqrt{N_{\lambda}}}{m\left(\eta\right)N_{\lambda}}=0.\label{eq:weight_lim}
\end{equation}
Inequality (\ref{eq:weight_ineq}) implies that groups with stronger
interactions between voters, that is groups that are in the low temperature
regime, will receive a larger weight in the council than groups with
voters who interact more loosely with each other, i.e.\! groups that
are in the high temperature regime. If the group populations are very
large, (\ref{eq:weight_lim}) says the groups with strong interactions
receive essentially all the voting weight in the council. The intuitive
reason for this result is that the majorities in strongly interacting
groups are typically of order $N_{\lambda}$, whereas the majorities
in weakly interacting groups are of order $\sqrt{N_{\lambda}}$. These
results go beyond Corollary 48 in \cite{BaMeSiTo2025}, which only
stated that for groups of equal size, the group with a larger coupling
constant receives a larger optimal weight.

Theorem \ref{thm:opt_weights} yields an estimator for the optimal
weights of each group by substituting any estimator for the coupling
parameters. For example, an estimator for the optimal weights based
on $\hat{\boldsymbol{\beta}}_{\boldsymbol{N}}^{\infty}$ from Definition
\ref{def:estimator_large_N} can be defined as follows. Recall that
for all $\lambda\in\IN_{M}$, $\hat{\beta}_{N}^{\infty}\left(\lambda\right)$
stands for the estimator $\hat{\beta}_{N}^{\infty}:\Omega_{N_{\lambda}}^{n}\rightarrow\left[-\infty,\infty\right]\cup\left\{ \textup{u}\right\} $
calculated for a sample of observations of voting in group $\lambda$.
\begin{defn}
\label{def:estimator_weights}Consider the intervals from Definition
\ref{def:intervals} and assume $\beta_{\lambda}\in I_{h}\cup I_{l}$.
The estimator $\hat{w}_{\lambda}:\Omega_{N_{\lambda}}^{n}\rightarrow\left[0,\infty\right)\cup\left\{ \textup{u}\right\} $
for the optimal weights of each group $\lambda\in\IN_{M}$ based on
$\hat{\beta}_{N}^{\infty}\left(\lambda\right)$ is defined by
\begin{enumerate}
\item If $\left(\boldsymbol{T}\left(x^{(1)},\ldots,x^{(n)}\right)\right)_{\lambda}\in J_{h}$,
then $\hat{w}_{\lambda}\coloneq\sqrt{\frac{2}{\pi}}\sqrt{\frac{1}{1-\hat{\beta}_{N}^{\infty}\left(\lambda\right)}}\sqrt{N_{\lambda}}.$
\item If $\left(\boldsymbol{T}\left(x^{(1)},\ldots,x^{(n)}\right)\right)_{\lambda}\in J_{l}$,
then $\hat{w}_{\lambda}\coloneq m\left(\hat{\beta}_{N}^{\infty}\left(\lambda\right)\right)N_{\lambda}$.
\item If $\left(\boldsymbol{T}\left(x^{(1)},\ldots,x^{(n)}\right)\right)_{\lambda}\in J_{c}$,
then we say there is insufficient evidence in the sample to conclude
that $\beta_{\lambda}$ is significantly different from 1, and $\hat{w}_{\lambda}\coloneq\textup{u}$
is undefined.
\end{enumerate}
\end{defn}

\begin{rem}
Similarly to how an undefined value $\hat{\beta}_{N}^{\infty}\left(\lambda\right)=\textup{u}$
can be interpreted as saying that the true value $\beta_{\lambda}$
lies close to 1, we could also assign an estimate for $\hat{w}_{\lambda}$
which corresponds to an asymptotic expression analogous to those in
Theorem \ref{thm:opt_weights} for $\beta_{\lambda}=1$. Since such
samples are a low probability event by Proposition \ref{prop:error},
we will refrain from going into further detail.
\end{rem}

In analogy to the constant $\tilde{\beta}_{N}$ from Definition \ref{def:beta_tilde},
we define a constant $\tilde{w}_{N}$:
\begin{defn}
\label{def:w_tilde}Let the intervals $I_{h}$ and $I_{l}$ be as
in Definition \ref{def:intervals}. Let, for all $\lambda\in\IN_{M}$,
$N_{\lambda}\in\IN$ and all $\beta_{\lambda}\in I_{h}\cup I_{l}$,
$\tilde{w}_{N_{\lambda}}\left(\lambda\right)\geq0$ be the value which
satisfies
\[
\tilde{w}_{N_{\lambda}}\left(\lambda\right)=\IE_{\beta_{\lambda},N_{\lambda}}\left|S\right|.
\]
\end{defn}

The estimator $\hat{w}_{\lambda}$ inherits many of the properties
of $\hat{\beta}_{N}^{\infty}\left(\lambda\right)$. This is the subject
of the next theorem. Let for each $N\in\IN$ the function $\psi_{N}^{\infty}:\left[-\infty,\infty\right]\backslash I_{c}\rightarrow\IR$
be defined by
\[
\psi_{N}^{\infty}\left(\beta\right)\coloneq\sqrt{\frac{2}{\pi}}\sqrt{\frac{1}{1-\beta}}\sqrt{N},\;\beta\leq b_{1},\quad\psi_{N}^{\infty}\left(\beta\right)\coloneq m\left(\beta\right)N,\;\beta\geq b_{2},\quad\textup{and}\quad\psi_{N}^{\infty}\left(-\infty\right)\coloneq0.
\]

Recall the rate function $J_{\lambda}$ from (\ref{eq:J}), and let
the good rate function $H_{\lambda}:\IR\rightarrow\left[0,\infty\right]$
be defined by
\[
H_{\lambda}\left(z\right)\coloneq\inf\left\{ J_{\lambda}\left(x\right)\,|\,x\in\left[-\infty,\infty\right],\psi_{N}^{\infty}\left(x\right)=z\right\} ,\quad z\in\IR.
\]

\begin{thm}
Let $\lambda\in\IN_{M}$ and $N_{\lambda}\in\IN$. Let $w_{\lambda}$
be the optimal weight from Theorem \ref{thm:opt_weights}. Then the
following statements hold:
\begin{enumerate}
\item $\hat{w}_{\lambda}\xrightarrow[n\rightarrow\infty]{\textup{p}}\tilde{w}_{N_{\lambda}}\left(\lambda\right)$.
\item $\tilde{w}_{N_{\lambda}}\left(\lambda\right)\xrightarrow[N_{\lambda}\rightarrow\infty]{}w_{\lambda}$.
\item $\sqrt{n}\left(\hat{w}_{\lambda}-\tilde{w}_{N_{\lambda}}\left(\lambda\right)\right)\xrightarrow[n\rightarrow\infty]{\textup{d}}\mathcal{N}\left(0,\upsilon_{N_{\lambda}}^{2}\right)$
and
\[
\upsilon_{N_{\lambda}}^{2}\underset{N_{\lambda}\rightarrow\infty}{\approx}\begin{cases}
\frac{1}{\pi}\frac{1}{1-\beta_{\lambda}}N_{\lambda} & \textup{if }\beta_{\lambda}\in I_{h},\\
\frac{1}{4m\left(\beta_{\lambda}\right)^{2}}\,N_{\lambda}^{2}\,\IV_{\beta_{\lambda},N_{\lambda}}\left(\frac{S_{\lambda}}{N_{\lambda}}\right)^{2} & \textup{if }\beta_{\lambda}\in I_{l}.
\end{cases}
\]
\item $\hat{w}_{\lambda}$ satisfies a large deviation principle with rate
$n$ and rate function $H_{\lambda}$. Furthermore, for any $\varepsilon>0$
there exist an $N_{\varepsilon}\in\IN$ and $B_{\varepsilon},C_{\varepsilon}>0$
such that for all $N\geq N_{\varepsilon}$,
\[
\IP\left\{ \left|\hat{w}_{\lambda}-w_{\lambda}\right|\geq\varepsilon\right\} \leq C_{\varepsilon}\exp\left(-B_{\varepsilon}n\right),\quad n\in\IN.
\]
\end{enumerate}
\end{thm}

\begin{proof}
This theorem is proved in close analogy to Theorem \ref{thm:properties_bML_inf}.
\end{proof}

\appendix

\section*{Appendix}

We present a number of concepts and auxiliary results we use. Some
of these are proved in \cite{BaMeSiTo2025}, and we give references
to the relevant results. Other statements are well known and can be
found in textbooks on probability theory.
\begin{lem}
\label{lem:Z_deriv}The first derivative of $Z_{\boldsymbol{\beta},\boldsymbol{N}}$
with respect to $\beta_{\lambda}$, $\lambda\in\IN_{M}$ is
\begin{align*}
\frac{\textup{d}Z_{\boldsymbol{\beta},\boldsymbol{N}}}{\textup{d}\beta_{\lambda}} & =\frac{Z_{\boldsymbol{\beta},\boldsymbol{N}}}{2N_{\lambda}}\,\IE_{\boldsymbol{\beta},\boldsymbol{N}}S_{\lambda}^{2}.
\end{align*}
\end{lem}

\begin{proof}
This is Lemma 51 in \cite{BaMeSiTo2025}.
\end{proof}
\begin{prop}
\label{prop:ES2_fin}The function $\beta\in\IR\mapsto\IE_{\beta,N}S^{2}\in\IR$
is strictly increasing and continuous, and we have for all $\beta\in\left(0,\infty\right)$
and all $N\in\IN$,
\[
N<\IE_{\beta,N}S^{2}<N^{2}.
\]
\end{prop}

\begin{proof}
This is Proposition 20 in \cite{BaMeSiTo2025}.
\end{proof}
\begin{defn}
\label{def:Legendre}Let $f:\IR\rightarrow\IR$ be a convex function.
Then the Legendre transform $f^{*}:\IR\rightarrow\IR\cup\{\infty\}$
of $f$ is defined by $f^{*}(t)\coloneq\sup_{x\in\IR}\left\{ xt-f(x)\right\} ,t\in\IR$.
\end{defn}

\begin{defn}
\label{def:cumul_entropy}Let $P$ be a probability measure on $\IR$.
Then let $\Lambda_{P}\left(t\right)\coloneq\ln\,\int_{\IR}\exp\left(tx\right)P\left(\textup{d}x\right)$
for all $t\in\IR$ such that the expression is finite. We call $\Lambda_{P}$
the cumulant generating function of $P$ and the Legendre transform
$\Lambda_{P}^{*}$ its entropy function. Let $Y$ be a real random
variable with distribution $P$. We will then say that $\Lambda_{Y}\coloneq\Lambda_{P}$
and $\Lambda_{Y}^{*}\coloneq\Lambda_{P}^{*}$ are the cumulant generating
and entropy function of $Y$, respectively.
\end{defn}

\begin{defn}
\label{def:ess}Let $Y$ be a random variable. The value
\[
\textup{ess\,inf }Y\coloneq\sup\left\{ a\in\IR\,|\,\IP\left\{ Y<a\right\} =0\right\} 
\]
is called the essential infimum of $Y$. We convene that $\textup{ess\,inf }Y\coloneq-\infty$
if the set of essential lower bounds for $Y$ on the right hand side
of the display above is empty. The value 
\[
\textup{ess\,sup }Y\coloneq\inf\left\{ a\in\IR\,|\,\IP\left\{ Y>a\right\} =0\right\} 
\]
is called the essential supremum of $Y$. We convene that $\textup{ess\,sup }Y\coloneq\infty$
if the set on the right hand side above is empty.
\end{defn}

\begin{lem}
\label{lem:cumulant_entropy}Let $Y$ be a bounded random variable
which is not almost surely constant. The cumulant generating function
$\Lambda_{Y}$ of $Y$ and the entropy function $\Lambda_{Y}^{*}$
of $Y$ have the following properties:
\begin{enumerate}
\item $\Lambda_{Y}$ is convex and differentiable.
\item $\Lambda_{Y}^{*}$ is finite on the interval $\left(\textup{ess\,inf }Y,\textup{ess\,sup }Y\right)$
and infinite on $\left[\textup{ess\,inf }Y,\textup{ess\,sup }Y\right]^{c}$.
\item $\Lambda_{Y}^{*}$ is strictly convex on $\left(\textup{ess\,inf }Y,\textup{ess\,sup }Y\right)$.
\item $\Lambda_{Y}^{*}$ is strictly decreasing on the interval $\left(\textup{ess\,inf }Y,\IE\,Y\right)$
and strictly increasing on $\left(\IE\,Y,\textup{ess\,sup }Y\right)$.
\item $\Lambda_{Y}^{*}$ has a unique global minimum at $\IE\,Y$ with $\Lambda_{Y}^{*}\left(\IE\,Y\right)=0$.
\end{enumerate}
\end{lem}

\begin{proof}
This is Lemma 26 in \cite{BaMeSiTo2025}.
\end{proof}
\begin{thm}[Contraction Principle]
\label{thm:contr_princ}Let $\mathcal{X}$ and $\mathcal{Y}$ be
metric spaces. Let $\left(P_{n}\right)_{n\in\IN}$ be a sequence of
probability measures on $\mathcal{X}$, and let $f:D\rightarrow\mathcal{Y}$
be a continuous function with its domain $D\subset\mathcal{X}$ containing
the support of $P_{n}$ for each $n\in\IN$. Let $\left(a_{n}\right)_{n\in\IN}$
a sequence of positive numbers with $a_{n}\xrightarrow[n\rightarrow\infty]{}\infty$,
and $I:\mathcal{X}\rightarrow\left[0,\infty\right]$ a good rate function.
We define $J:\mathcal{Y}\rightarrow\left[0,\infty\right]$ by
\[
J\left(y\right)\coloneq\inf\left\{ I\left(x\right)\,|\,x\in D,f\left(x\right)=y\right\} ,\quad y\in\mathcal{Y}.
\]
Then $J$ is a good rate function. If $\left(P_{n}\right)_{n\in\IN}$
satisfies a large deviations principle with rate $a_{n}$ and rate
function $I$, then the sequence of push forward measures $\left(P_{n}\circ f^{-1}\right)_{n\in\IN}$
on $\mathcal{Y}$ satisfies a large deviations principle with rate
$a_{n}$ and rate function $J$.
\end{thm}

\begin{thm}[Varadhan's Lemma]
\label{thm:Varadhan}Let $\mathcal{X}$ be a metric space. Let $\left(P_{n}\right)_{n\in\IN}$
be a sequence of probability measures on $\mathcal{X}$, and let $f:\mathcal{X}\rightarrow\IR$
be a continuous function. Let $\left(a_{n}\right)_{n\in\IN}$ a sequence
of positive numbers with $a_{n}\xrightarrow[n\rightarrow\infty]{}\infty$,
and $I:\mathcal{X}\rightarrow\left[0,\infty\right]$ a good rate function.
Suppose $\left(P_{n}\right)_{n\in\IN}$ satisfies a large deviations
principle with rate $a_{n}$ and rate function $I$. If
\[
\lim_{M\rightarrow\infty}\limsup_{n\rightarrow\infty}\frac{1}{a_{n}}\ln\int_{\left\{ f\left(x\right)\geq M\right\} }\exp\left(a_{n}f\left(x\right)\right)\,P_{n}\left(\textup{d}x\right)=-\infty
\]
holds, then we have
\[
\lim_{n\rightarrow\infty}\frac{1}{a_{n}}\ln\int_{\mathcal{X}}\exp\left(a_{n}f\left(x\right)\right)\,P_{n}\left(\textup{d}x\right)=\sup_{x\in\mathcal{X}}\left\{ f\left(x\right)-I\left(x\right)\right\} .
\]
\end{thm}

\begin{lem}
\label{lem:contr_princ_min}Let $\mathcal{X}$ and $\mathcal{Y}$
be metric spaces, $I:\mathcal{X}\rightarrow\left[0,\infty\right]$
a good rate function, and $f:\mathcal{X}\rightarrow\mathcal{Y}$ a
continuous function. We define $J:\mathcal{Y}\rightarrow\left[0,\infty\right]$
by
\[
J\left(y\right)\coloneq\inf\left\{ I\left(x\right)\,|\,x\in\mathcal{X},f\left(x\right)=y\right\} ,\quad y\in\mathcal{Y}.
\]
Let $M_{I}\subset\mathcal{X}$ be the set of minima of $I$ and $M_{J}\subset\mathcal{Y}$
be the set of minima of $J$. Then $f\left(M_{I}\right)=M_{J}$ holds.
In particular, if $f$ is injective and $I$ has a unique minimum
at $x_{0}$, then $J$ has a unique minimum at $f\left(x_{0}\right)$.
\end{lem}

\begin{proof}
This is Lemma 35 in \cite{BaMeSiTo2025}.
\end{proof}
\begin{lem}
\label{lem:exponential_conv}Let $\left(Y_{n}\right)_{n\in\IN}$ be
a sequence of random variables that satisfies a large deviations principle
with rate $a_{n}$ and rate function $I$. Assume that $I$ has a
unique minimum at $x_{0}\in\mathcal{X}$. Let $K$ be a closed set
which does not contain $x_{0}$. Then there are constants $B_{K},C_{K}>0$
such that
\[
\IP\left\{ Y_{n}\in K\right\} \leq C_{K}\exp\left(-B_{K}a_{n}\right)
\]
holds for all $n\in\IN$.
\end{lem}

\begin{thm}[Slutsky]
\label{thm:slutsky}Let $\left(Y_{n}\right)_{n\in\IN}$ and $\left(Z_{n}\right)_{n\in\IN}$
be sequences of random variables, $Y$ a random variable, and $a\in\IR$
a constant such that $Y_{n}\xrightarrow[n\rightarrow\infty]{\textup{d}}Y$
and $Z_{n}\xrightarrow[n\rightarrow\infty]{\textup{p}}a$. Then
\[
Y_{n}+Z_{n}\xrightarrow[n\rightarrow\infty]{\textup{d}}Y+a\quad\text{and}\quad Y_{n}Z_{n}\xrightarrow[n\rightarrow\infty]{\textup{d}}aY.
\]
\end{thm}

\begin{thm}[Continuous Mapping]
\label{thm:cont_mapping}Let $\left(Y_{n}\right)_{n\in\IN}$ be a
sequence of random variables and $Y$ a random variable, each of them
taking values in some subset $A\subset\IR$, such that $Y_{n}\xrightarrow[n\rightarrow\infty]{\textup{p}}Y$,
and let $g:A\rightarrow\IR$ be a continuous function. Then
\[
g\left(Y_{n}\right)\xrightarrow[n\rightarrow\infty]{\textup{p}}g\left(Y\right).
\]
\end{thm}

\begin{thm}[Delta Method]
\label{thm:delta_method}Let $\left(Y_{n}\right)_{n\in\IN}$ be a
sequence of random variables and $\mu\in\IR,\sigma>0$ such that $\sqrt{n}\left(Y_{n}-\mu\right)\xrightarrow[n\rightarrow\infty]{\textup{d}}\mathcal{N}\left(0,\sigma^{2}\right)$
holds. Let $f:D\rightarrow\IR$ be a continuously differentiable function
with domain $D\subset\IR$ such that $Y_{n}\in D$ for all $n\in\IN$.
Assume $f'\left(\mu\right)\neq0$. Then
\[
\sqrt{n}\left(f\left(Y_{n}\right)-f\left(\mu\right)\right)\xrightarrow[n\rightarrow\infty]{\textup{d}}\mathcal{N}\left(0,\left(f'\left(\mu\right)\right)^{2}\sigma^{2}\right)
\]
is satisfied.
\end{thm}

\begin{lem}
\label{lem:conv_restr_sequence}Let $\left(Y_{n}\right)_{n\in\IN}$
be a sequence of random variables and $\left(M_{n}\right)_{n\in\IN}$
a sequence of positive numbers such that
\[
\left|Y_{n}\right|\leq M_{n}\quad\textup{a.s.},\quad n\in\IN,
\]
is satisfied. Let $\nu$ be a probability measure on $\IR$, and assume
the convergence $Y_{n}\xrightarrow[n\rightarrow\infty]{\textup{d}}\nu$.
Finally, let $\left(B_{n}\right)_{n\in\IN}$ be a sequence of measurable
sets which satisfies
\[
\IP\left\{ Y_{n}\in B_{n}\right\} =o\left(\frac{1}{M_{n}}\right).
\]
Then we have for all $z\in\IR$,
\[
\II_{\left\{ Y_{n}\in B_{n}^{c}\right\} }Y_{n}+z\II_{\left\{ Y_{n}\in B_{n}\right\} }\xrightarrow[n\rightarrow\infty]{\textup{d}}\nu.
\]
\end{lem}

\begin{proof}
This is Lemma 33 in \cite{BaMeSiTo2025}.
\end{proof}
\begin{prop}
\label{prop:conv_stat}Let $n,N\in\IN$ and let $R:\Omega_{N}^{n}\rightarrow\IR$
be a statistic of the form
\[
R\left(x^{(1)},\ldots,x^{(n)}\right)\coloneq\frac{1}{n}\sum_{t=1}^{n}f\left(x^{(t)}\right),\quad\left(x^{(1)},\ldots,x^{(n)}\right)\in\Omega_{N}^{n},
\]
for some function $f:\Omega_{N}\rightarrow\IR$. Let $X$ be a random
vector on $\Omega_{N}$ with Curie-Weiss distribution according to
Definition \ref{def:CWM}. Let $\mu\coloneq\IE_{\beta,N}f\left(X\right)$,
$\sigma^{2}\coloneq\IV_{\beta,N}f\left(X\right)$, and $\Lambda_{f\left(X\right)}^{*}$
the entropy function of $f\left(X\right)$.

Then the following three statements hold:
\begin{enumerate}
\item A law of large numbers holds for the sequence $R\left(x^{(1)},\ldots,x^{(n)}\right)$:
\[
R\left(x^{(1)},\ldots,x^{(n)}\right)\xrightarrow[n\rightarrow\infty]{\textup{p}}\mu.
\]
\item A central limit theorem holds for the sequence $\sqrt{n}\left(R\left(x^{(1)},\ldots,x^{(n)}\right)-\mu\right)$:
\[
\sqrt{n}\left(R\left(x^{(1)},\ldots,x^{(n)}\right)-\mu\right)\xrightarrow[n\rightarrow\infty]{\textup{d}}\mathcal{N}\left(0,\sigma^{2}\right).
\]
\item A large deviations principle holds for the sequence $R\left(x^{(1)},\ldots,x^{(n)}\right)$
with rate $n$ and rate function $I:\IR\rightarrow\left[0,\infty\right)\cup\left\{ \infty\right\} $,
\[
I\left(x\right)\coloneq\Lambda_{f\left(X\right)}^{*}\left(x\right),\quad x\in\IR.
\]
\item If $I$ has a unique global minimum at $x_{0}\in\IR$, then for any
closed set $K\subset\IR$ that does not contain $x_{0}$ we have $\delta\coloneq\inf_{x\in K}I\left(x\right)>0$,
and
\[
\IP\left\{ R\in K\right\} \leq2\exp\left(-\delta n\right)
\]
holds for all $n\in\IN$.
\end{enumerate}
\end{prop}

\begin{defn}
\label{def:exchange}Let $Y_{1},\ldots,Y_{n}$ be real random variables
with joint distribution $P$ on $\IR^{n}$. We say that $Y_{1},\ldots,Y_{n}$
are exchangeable if for all permutations $\pi$ on $\IN_{n}$ the
random vector $\left(Y_{\pi\left(1\right)},\ldots,Y_{\pi\left(n\right)}\right)$
has joint distribution $P$.
\end{defn}

\begin{lem}
\label{lem:exchange}The random variables $X_{\lambda1},\ldots,X_{\lambda N_{\lambda}}$
are exchangeable under the distribution $\IP_{\beta_{\lambda},N_{\lambda}}$
in Definition \ref{def:CWM}.
\end{lem}

\begin{proof}
Let $\pi$ be any permutation on $\IN_{N_{\lambda}}$ and $x\in\Omega_{N_{\lambda}}$
any voting configuration. Then we have
\begin{align*}
\IP_{\beta,N}\left(X_{\lambda\pi\left(1\right)}=x_{1},\ldots,X_{\lambda\pi\left(N_{\lambda}\right)}=x_{N_{\lambda}}\right) & =\IP_{\beta,N}\left(X_{\lambda1}=x_{\pi^{-1}\left(1\right)},\ldots,X_{\lambda N_{\lambda}}=x_{\pi^{-1}\left(N_{\lambda}\right)}\right)\\
 & =Z_{\beta_{\lambda}N_{\lambda}}^{-1}\,\exp\left(\frac{\beta_{\lambda}}{2N_{\lambda}}\left(\sum_{i=1}^{N_{\lambda}}x_{\pi^{-1}\left(i\right)}\right)^{2}\right)\\
 & =Z_{\beta_{\lambda},N_{\lambda}}^{-1}\,\exp\left(\frac{\beta_{\lambda}}{2N_{\lambda}}\left(\sum_{i=1}^{N_{\lambda}}x_{i}\right)^{2}\right)\\
 & =\IP_{\beta_{\lambda},N_{\lambda}}\left(X_{\lambda1}=x_{1},\ldots,X_{\lambda N_{\lambda}}=x_{N_{\lambda}}\right).
\end{align*}
\end{proof}
\begin{lem}
\label{lem:multiplicity}Let $k\in\IN$ and $\underbar{\ensuremath{\boldsymbol{r}}}\in\Pi$.
The number of index vectors $\underbar{\ensuremath{\boldsymbol{i}}}\in\IN_{N}^{k}$
such that $\underbar{\ensuremath{\boldsymbol{r}}}=\left(r_{1},\ldots,r_{k}\right)=\underbar{\ensuremath{\boldsymbol{\rho}}}\left(\underbar{\ensuremath{\boldsymbol{i}}}\right)$
is given by
\[
\frac{N!}{r_{1}!\cdots r_{k}!\left(N-\sum_{\ell=1}^{k}r_{\ell}\right)!}\frac{k!}{1!^{r_{1}}\cdots k!^{r_{k}}}.
\]
\end{lem}

\begin{proof}
This is Lemma 39 in \cite{BaMeSiTo2025}.
\end{proof}
\begin{thm}[Method of Moments]
\label{thm:mm}Let $\left(Y_{n}\right)_{n\in\IN}$ be a sequence
of random variables such that there are constants $A,B\in\IR$ and
a sequence $\left(m_{k}\right)_{k\in\IN}$ of real numbers with the
property that
\[
\lim_{n\rightarrow\infty}\IE\,Y_{n}^{k}=m_{k}
\]
 holds for all $k\in\IN$, and the sequence $\left(m_{k}\right)_{k\in\IN}$
satisfies
\[
m_{k}\leq AB^{k}k!,\quad k\in\IN.
\]
Then $\left(Y_{n}\right)_{n\in\IN}$ converges in distribution to
some probability measure $P$, and $P$ is uniquely determined by
its moments $m_{k}$, $k\in\IN$.
\end{thm}

\section*{Acknowledgements}

M.\! B.\! is a Fellow and G.\! T.\! is a Candidate of the Sistema
Nacional de Investigadoras e Investigadores. G.\! T.\! was supported
by a Secihti (formerly Conahcyt) postdoctoral fellowship.

\bibliographystyle{plain}
\bibliography{C:/Users/gabor/OneDrive/Documents/MyArticles/References}

\end{document}